\documentclass{amsart}
\usepackage{graphicx, color, amsmath, amsfonts, amsthm, hyperref, amssymb, xypic, stmaryrd, mathtools} % Required for inserting images

\numberwithin{equation}{section}

\newcommand{\bF}{\mathbb{F}}

\newcommand{\bZ}{\mathbb{Z}}
\newcommand{\bQ}{\mathbb{Q}}
\newcommand{\bA}{\mathbb{A}}
\newcommand{\bT}{\mathbb{T}}

\newcommand{\cO}{\mathcal{O}}
\newcommand{\cS}{\mathcal{S}}
\newcommand{\cQ}{\mathcal{Q}}
\newcommand{\cJ}{\mathcal{J}}
\newcommand{\fm}{\mathfrak{m}}

\newcommand{\fg}{\mathfrak{g}}
\newcommand{\ft}{\mathfrak{t}}
\newcommand{\loc}{\mathrm{loc}}

\newcommand{\ad}{\operatorname{ad}}
\newcommand{\rhobar}{\overline{\rho}}
\newcommand{\rhobarss}{\overline{\rho}^{\textnormal{ss}}}
\newcommand{\GL}{\mathrm{GL}}

\newcommand{\Ind}{\operatorname{Ind}}
\newcommand{\Gal}{\operatorname{Gal}}
\newcommand{\Frob}{\mathrm{Frob}}
\newcommand{\Hom}{\operatorname{Hom}}

\newcommand{\tld}[1]{\widetilde{#1}}

\newcommand{\trace}{\operatorname{tr}}
\newcommand{\diag}{\operatorname{diag}}

\newcommand{\image}{\operatorname{im}}
\newcommand{\isom}{\xrightarrow{\sim}}

\newcommand{\fn}{\mathfrak{n}}
\newcommand{\lbar}[1]{\overline{#1}}
\newcommand{\Art}{\operatorname{Art}}

\newcommand{\Iw}{\mathrm{Iw}}

\newcommand{\St}{\mathrm{St}}

\newcommand{\aux}{\mathrm{aux}}

\newcommand{\nr}{\mathrm{nr}}

\newcommand{\der}{\mathrm{der}}
\newcommand{\Lie}{\operatorname{Lie}}

\newcommand{\length}{\operatorname{length}}
\newcommand{\ord}{\operatorname{ord}}
\newcommand{\CH}{\operatorname{CH}}
\newcommand{\chibar}{\overline{\chi}}

\DeclareMathOperator{\tr}{tr}
\DeclareMathOperator{\CNL}{CNL}
\DeclareMathOperator{\Spec}{Spec}
\DeclareFontFamily{U}{wncy}{}
\DeclareFontShape{U}{wncy}{m}{n}{<->wncyr10}{}
\DeclareSymbolFont{mcy}{U}{wncy}{m}{n}
\DeclareMathSymbol{\Sha}{\mathord}{mcy}{"58}

\theoremstyle{definition}
\newtheorem{defn}[subsubsection]{Definition}
\newtheorem{cor}[subsubsection]{Corollary}
\newtheorem{pp}[subsubsection]{Proposition}
\newtheorem{thm}[subsubsection]{Theorem}
\newtheorem{lemma}[subsubsection]{Lemma}

\theoremstyle{remark}
\newtheorem{rem}[subsubsection]{Remark}

\title[Modularity of residually reducible Galois representations]{On the modularity of residually reducible Galois representations of rank 2}

\begin{document}
\author[A. Dotto and J. A. Thorne]{Andrea Dotto and Jack A. Thorne}
\begin{abstract}
We use a new version of the Taylor--Wiles method, based on a tame analogue of the techniques introduced by Skinner--Wiles, to prove new modularity theorems for residually reducible Galois representations of rank 2 over totally real fields. 
\end{abstract}

\maketitle

\tableofcontents

\section{Introduction}

\subsection{Context} The relation between modular forms and Galois representations is one of the central themes of the Langlands programme. Galois representations can be used to study properties of modular forms: see for example their use, by Serre and Swinnerton-Dyer, to characterize the congruences satisfied by Ramanujan's $\tau$ function. They can also be used as an intermediary in the conjectural dictionary between automorphic representations and motives \cite{Clo90, Tay02}. Serre's conjecture \cite{Ser87}, which predicts a dictionary between Galois representations and modular forms modulo $p$ that is very tightly constrained by local-global compatibility, has become a key organising principle, as well as having far-reaching arithmetic consequences. It is now a theorem \cite{Kha09a,Kha09}, but generalisations to base fields other than the rational numbers so far remain out of reach. 

One of the main tools in studying the modular--Galois dictionary is the idea of modularity lifting theorems, that prove the modularity of Galois representations with $p$-adic coefficients starting from modulo $p$ information. The first modularity lifting theorems were proved by Wiles and Taylor on the way to the proof of the modularity of semistable elliptic curves over $\mathbb{Q}$ (and therefore Fermat's Last Theorem) \cite{Wil95, Tay95}. These theorems applied to Galois representations $\rho : G_{\mathbb{Q}} \to \GL_2(\overline{\mathbb{Q}}_p)$ with (in retrospect) rather nice properties: for example, that $\rho$ arises from a $p$-divisible group over $\mathbb{Z}[1/N]$ for some $N$ prime to $p$, and that the semisimple residual representation $\overline{\rho}^{\textnormal{ss}} : G_{\mathbb{Q}} \to \GL_2(\overline{\mathbb{F}}_p)$ is both irreducible and modular (in the sense of Serre's conjecture). 

The Taylor--Wiles method has now been refined and generalised in works by many authors (see e.g. \cite{Cal23} for an excellent survey), which include applications both to number fields other than $\mathbb{Q}$ (including CM fields, with applications to the modularity of elliptic curves over imaginary quadratic fields \cite{Car25}) and Galois representations with repeated Hodge--Tate weights (including applications to the modularity of abelian surfaces over $\mathbb{Q}$ \cite{Box21}). The most powerful results have generally included some version of the condition that $\rhobarss$ is irreducible, or even the stronger condition that $\rhobarss|_{G_{\mathbb{Q}(\zeta_p)}}$ is irreducible, the so-called \emph{Taylor--Wiles hypothesis}. 

One result that falls outside this framework is the theorem of Skinner--Wiles \cite{Ski99}, which proves the modularity of many representations $\rho : G_{\mathbb{Q}} \to \GL_2(\overline{\mathbb{Q}}_p)$ for which the associated residual representation is \emph{reducible}. The difficulty is that in this case the deformation theory of $\rhobarss$, which is the basic tool in the proofs of all known modularity lifting theorems, is very hard to control. On the other hand, when such results are available, they can have significant consequences: the results of \cite{Ski99} form an indispensable part of the proof of Serre's conjecture, starting in \cite{Kha06}, and form the stepping stone for the far-reaching results towards the Fontaine--Mazur conjecture (still in the case where the base field is $\mathbb{Q}$) given in \cite{Pan22}. 
More precisely, \cite{Ski99} and \cite{Pan22} give a general result when $F / \mathbb{Q}$ is totally real and the splitting field of $\rhobarss : G_F \to \GL_2(\overline{\mathbb{F}}_p)$ is abelian over $\mathbb{Q}$: but this condition is seldom satisfied when $F \neq \mathbb{Q}$. 

\subsection{Results of this paper} In this paper we prove new modularity results for 2-dimensional $p$-adic Galois representations over totally real fields, whose residual representations are reducible, therefore generalising the main theorem of \cite{Ski99}. Our main result is as follows:
\begin{thm}\label{introthm_automorphy_lifting}
    Let $F$ be a totally real field,
    let $p$ be an odd prime, and let $\rho : G_F \to \GL_2(\overline{\mathbb{Q}}_p)$ be a continuous, irreducible representation satisfying the following conditions:
    \begin{enumerate}
        \item $\rho$ is unramified at all but finitely many places, and $\det \rho = \epsilon^{-1}$. 
        \item There exists $\mathbf{\lambda} \in (\mathbb{Z}^2_+)^{\Hom(F, \overline{\mathbb{Q}}_p)}$ such that $\rho$ is ordinary of weight $\mathbf{\lambda}$. 
        \item The semisimple residual representation 
        $\overline{\rho}^{\textnormal{ss}} : G_F \to \GL_2(\overline{\mathbb{F}}_p)$ is reducible, isomorphic to $\overline{\chi}_1 \oplus \overline{\chi}_2$ for $\overline{\chi}_1, \overline{\chi}_2 : G_F \to \overline{\mathbb{F}}_p^\times$ continuous characters, and $(\overline{\chi}_1 / \overline{\chi}_2)|_{G_{F(\zeta_p)}}$ is non-trivial. 
    \end{enumerate}
    Then $\rho$ is automorphic: there exists an isomorphism $\iota : \overline{\mathbb{Q}}_p \to \mathbb{C}$ and a cuspidal, regular algebraic, $\iota$-ordinary automorphic representation $\pi$ of $\GL_2(\mathbb{A}_F)$ of weight $\iota \mathbf{\lambda}$ such that $\rho \cong r_{\pi, \iota}$. 
\end{thm}
(We recall what it means for $\rho$ to be ``ordinary of weight $\lambda$'' and ``automorphic of weight $\iota \lambda$'' in \S \ref{subsec_notation} below.) Combining this with the theorems of \cite{Ski99} and their refinements in \cite{Pan22}, we immediately deduce the following corollary.
One can easily write down other applications, such as to the modularity of elliptic curves. 

\begin{cor}\label{introcor_modularity_in_abelian_case}
    Let $F$ be a totally real field such that $F / \mathbb{Q}$ is abelian, and let $p \geq 5$ be a prime that is unramified in $F$. Let $\rho : G_F \to \GL_2(\overline{\mathbb{Q}}_p)$ be a continuous, irreducible representation satisfying the following conditions:
    \begin{enumerate}
        \item  $\rho$ is unramified at all but finitely many places, and $\det \rho = \epsilon^{-1}$. 
        \item There exists $\mathbf{\lambda} \in (\mathbb{Z}^2_+)^{\Hom(F, \overline{\mathbb{Q}}_p)}$ such that $\rho$ is ordinary of weight $\mathbf{\lambda}$. 
        \item The semi-simplified residual representation $\rhobarss$ is reducible. 
    \end{enumerate}
    Then $\rho$ is automorphic, associated to a cuspidal, regular algebraic, ordinary automorphic representation of $\GL_2(\mathbb{A}_F)$. 
\end{cor}
\begin{proof}
    If $(\overline{\chi}_1 / \overline{\chi}_2)|_{G_{F(\zeta_p)}}$ is non-trivial, this follows from Thm.\ \ref{introthm_automorphy_lifting}. Otherwise, $\overline{\chi}_1 / \overline{\chi}_2$ is a power of the (mod $p$) cyclotomic character, in which case we can appeal to \cite[Thm.\ 5.1.1]{Pan22}.%
\end{proof}

\subsection{Methods} 
The proof of Thm.\ \ref{introthm_automorphy_lifting} involves several different ingredients. The first is a modularity lifting theorem, in the residually reducible case, with some auxiliary ramification conditions (Thm.\ \ref{thm_tame_SW_black_box}); note that the residual modularity comes for free, thanks to the existence of Eisenstein series (Lem.\ \ref{lem_residual_automorphy}). 
We first discuss the available methods for proving such a theorem, and the difficulties they encounter.

\subsubsection{The Taylor--Wiles method}
Given an irreducible residual representation $\overline{\rho}$, one can form a Galois deformation ring $R$, representing the functor of deformations of $\overline{\rho}$ satisfying certain conditions (on allowed ramification, $p$-adic Hodge-theoretic properties, etc.). 
The existence and local-global compatibility of the Galois representations associated to suitable automorphic forms is then expressed in terms of the existence of a map 
\begin{equation*}
    \xymatrix{ R \ar[r] &  \mathbb{T} \curvearrowright  H }
\end{equation*}
arising from the universal property of $R$, where $\mathbb{T}$ is the Hecke algebra that acts faithfully on the space $H$ of automorphic forms. Proving a modularity lifting theorem means proving that this map is an isomorphism. In the context of the Taylor--Wiles method, this is done by introducing ramification at auxiliary sets $Q$ of Taylor--Wiles places; one gets auxiliary diagrams
\begin{equation*}
    \xymatrix{  R \ar[r] &  \mathbb{T} \curvearrowright  H \\
    R_Q \ar[r]\ar[u]  & \mathbb{T}_Q\ar[u] \curvearrowright  H_Q  \\
    R_\infty \ar[r]\ar[u]  & \mathbb{T}_\infty\ar[u] \curvearrowright  H_\infty, 
    }
\end{equation*}
where $R_Q$, $\mathbb{T}_Q$, $H_Q$ are the objects with auxiliary ramification, and $R_\infty$, $\mathbb{T}_\infty$, $H_\infty$ are the ``patched'' objects that arise by varying the Taylor--Wiles data and appealing to compactness. Typically one now uses a lower bound on the size of $H_\infty$ (arising from the freeness of the action of diamond operators at Taylor--Wiles places) and an upper bound on the size of $R_\infty$ (arising from a tangent space calculation, using Tate's Euler characteristic formula for global Galois cohomology), which together force $R = \mathbb{T}$ (or perhaps a weaker statement still sufficient for the modularity lifting theorem, such as that the surjective map $R \to \mathbb{T}$ has nilpotent kernel). 

\subsubsection{The Skinner--Wiles argument}
When the residual Galois representation $\overline{\rho}$ is reducible, one must use different objects. If $\overline{\rho}$ is semisimple, the deformation functor is not representable, so must be replaced by the functor of pseudodeformations, with its associated representing object $P$. For each indecomposable representation $\overline{\rho}_b$ with semisimplification $\overline{\rho}$, the deformation functor of $\overline{\rho}_b$ is representable, but the representing object $R_b$ no longer admits a map to the Hecke algebra (as the automorphic Galois representations carried by $\mathbb{T}$ might not all admit an invariant lattice reducing to $\overline{\rho}_b$). One gets a collection of diagrams indexed by non-zero cohomology classes $[b] \in H^1(F_S / F, \overline{\chi})$
\[ \xymatrix{ R_b  & P\ar[l] \ar[r] & \mathbb{T} \curvearrowright  H. } \]
Now $R_b$ admits a tangent-obstruction theory close to that of $R$ in the residually irreducible case, but $P$ does not (although see \cite{Wan20}). The approach taken in \cite{Ski99} is to pass from the maximal ideal $\mathfrak{m} \leq P$ to dimension 1 prime ideals $\mathfrak{p}$ such that the specialisation of the universal pseudodeformation to $P / \mathfrak{p}$ satisfies the Taylor--Wiles hypothesis relative to $k(\mathfrak{p}) \coloneqq \operatorname{Frac} (P / \mathfrak{p})$ (which is a local field). For a suitable $\overline \rho_b$, depending on $\mathfrak{p}$, this prime lifts to a prime $\mathfrak{p}_b \leq R_b$, and the induced map on localisations and completions $\widehat{(P_{\mathfrak{p}})} \to \widehat{(R_{b, \mathfrak{p}_b})}$ is an isomorphism. One can then use something like the usual Taylor--Wiles formalism to prove that these two isomorphic rings act faithfully on the localized and completed space $\widehat{(H_{\mathfrak{p}})}$ of modular forms. 

One is then faced with the problem of propagating this ``local'' modularity statement across the whole of $\Spec P$. This means constructing enough primes $\mathfrak{p}$ of a suitable type. This is possible provided the reducible locus in each ring $R_b$ has small enough dimension relative to the dimension of the Hecke algebra $\mathbb{T}$. The dimension of the reducible locus is controlled by a Galois cohomology group $H^1(F_S / F, \overline{\chi})$. Skinner--Wiles work with ordinary forms of variable weight, leading to a Hecke algebra of expected relative dimension $[ F : \mathbb{Q} ]$, and achieve a positive result when $\dim H^1(F_S/F, \overline{\chi})$ is small compared to $[F : \mathbb{Q}]$.
If the splitting field of $\overline{\chi}$ over $\mathbb{Q}$ is abelian, they show that this can be achieved after an abelian base change, using a theorem of Washington \cite{Was78}.

\subsubsection{A tame analogue of the Skinner--Wiles argument}\label{subsubsec_tame_SW_method}
Our point of departure is to pursue a tame analogue of the Skinner--Wiles argument. 
Instead of ordinary forms, we introduce sets of Taylor--Wiles primes and construct patched diagrams 
\[ \xymatrix{ R_b  & P\ar[l] \ar[r] & \mathbb{T} \curvearrowright  H \\
R_{b, \infty} \ar[u] & P_\infty \ar[l] \ar[r]\ar[u]  & \mathbb{T}_\infty \curvearrowright H_\infty. \ar[u]} \]
We do not kill the dual Selmer group at this point, so although this patched diagram exists for the usual reasons of compactness, and $H_\infty$ is free over a patched ring of diamond operators, there is no hope of getting an upper bound on the size of $R_{b, \infty}$ (or $P_\infty$) using Galois cohomology. Instead, we hope to keep the size of the reducible locus, or its patched version, under control. This turns out to be possible if we introduce a hypothesis of auxiliary ramification: we suppose that we are restricting to automorphic forms that are of Steinberg type at some auxiliary set $T$, with $q_v \not\equiv 1 \text{ mod }p$ for each $v \in T$. If we choose our Taylor--Wiles sets $Q$ to be well-placed with respect to $T$, then we can ensure that the reducible locus stays of bounded dimension, while the dimension (and importantly, the connectedness dimension) of $R_{b, \infty}$ can be increased. The key part of the argument is a calculation in class field theory, which forms part of the proof of Prop.\ \ref{prop_existence_of_TW_places_first_patching}.

We then aim to run something like the Skinner--Wiles argument for the rings $P_\infty$ and $R_{b, \infty}$. These rings no longer have a purely arithmetic interpretation, but they carry (pseudo-)representations of patched Galois groups (Lem.\ \ref{lem_patched_ring_represents_a_functor}), and the patched reducibility locus is in fact just the reducibility locus of these (pseudo-)representations.
Implicit in these statements is the fact that the relevant Galois groups can be patched using the Taylor--Wiles method; the requisite finiteness properties are a consequence of Mazur's condition~$\Phi_p$ (Lem.\ \ref{lem_quantitative_continuity_of_pseudocharacters}).
We then show that if $\mathfrak{p} \leq P_\infty$ is a dimension 1 prime, outside the reducible locus, and which lifts to a prime $\mathfrak{p}_b \leq R_{b, \infty}$, then the induced map on localisations and completions is an isomorphism (Cor.\ \ref{cor_comparison_of_patched_objects}), and these rings act nearly faithfully on the localised and completed patched space of modular forms (Prop.\ \ref{prop_local_R_equals_T}). This necessitates a further iteration of the Taylor--Wiles argument, where we introduce further sets of Taylor--Wiles places and carry out a second diagonalisation argument (\S\ref{subsec_second_patching}). 

An important point to note is that we do not need to boost the dimension of the Hecke algebra by working with ordinary forms: a large enough boost comes just from the first patching argument, and many of our arguments could certainly be adapted to non-ordinary cases.
However, we use the ordinarity condition in two places: first,
we need the local lifting ring at $p$-adic places to be normal, in order that it remain irreducible after localisation and completion at the contraction of the one-dimensional prime~$\mathfrak{p}_b$ of~$R_{b, \infty}$, and the only general situation in which we know how to prove this is the ordinary case. 
As pointed out to us by Lue Pan, this normality hypothesis could be removed by working with one-dimensional primes that contract to the maximal ideal.  
Second, we use the ordinarity condition to ensure that the image of our target representation~$\rho$ contains a conjugate of an open subgroup of $\GL_2(\mathbb{Z}_p)$ (Lem.\ \ref{lem_ordinary_implies_split_image}), which plays an important role in the arguments to be described
in the next section.

\subsubsection{Deducing our main result}
Ultimately, our analogue of the Skinner--Wiles argument leads to a modularity lifting theorem, in the case where the set $T$ of auxiliary places of Steinberg ramification is large enough relative to the dimension of $H^1(F_S / F, \overline{\chi})$. Further arguments are required in order to deduce Thm.\ \ref{introthm_automorphy_lifting}. 
Our strategy is, starting with a suitable $\cO[G_F]$-lattice in a representation $\rho$ as in the statement of the theorem, to introduce auxiliary sets $T$ and construct lifts $\rho'_T : G_F \to \GL_2(\cO)$ of $\rho \text{ mod }\varpi^N$ which are of Steinberg type at $T$. The modularity of these lifts can then be established using our modularity lifting theorem with auxiliary ramification. The lifts are constructed using a variation on the ideas introduced by Fakhruddin--Khare--Patrikis \cite{FKPI, FKPII}, which are an effective tool for constructing lifts of residually reducible Galois representations. 

We then hope to ``lower the level'', proving that $\rho'_T \text{ mod }\varpi^N$ is modular already at minimal (i.e.\ unramified at $T$) level. Level-lowering is a well-known question in the theory of automorphic forms (going back e.g.\ to \cite{Rib90}), but we are not aware of any results in the situation we need, which is level-lowering mod $\varpi^N$ for a residually reducible Galois representation. 
Matters are further complicated by the fact that the arguments of \cite{FKPII} require introducing ramification at sets $S_{\textnormal{aux}}$ of so-called ``trivial'' places (where $q_v \equiv 1 \text{ mod }p$ and the residual Galois representation is trivial on restriction to the decomposition group), which fall outside of the most robustly understood context for level-lowering, where Mazur's principle applies. Happily, it turns out that the classical arguments for proving Mazur's principle work well in our context, once one makes the observation that the action of $\mathbb{T}[G_F]$ on the \'etale cohomology of a Shimura curve necessarily factors through its maximal Cayley--Hamilton quotient: in this $\GL_2$ situation, this follows from the Eichler--Shimura relation. The price to be paid is that in passage from $T$-new level to minimal level, some ``error terms'' are introduced, giving only modularity of $\rho'_T \text{ mod }\varpi^{N-C_1 |T| - C_2 |S_{\aux}|}$ for some $C_1, C_2 \geq 0$. 

In order to complete the argument, we need to carefully control these parameters. For a given set $T$ and $N \geq 1$, the method of \cite{FKPII} gives a lift of $\rho \text{ mod }\varpi^N$ that is ramified at $T$ and also some unknown set $S_{\textnormal{aux}}$ of trivial places; for the modularity lifting theorem to apply, we need $|T|$ to exceed  $\dim H^1(F_{S \cup S_{\textnormal{aux}}} / F, \overline{\chi}) \leq \dim H^1(F_S / F, \overline{\chi}) + |S_{\textnormal{aux}}|$; and for there to be a meaningful result after applying level-lowering, we need the error term $C_1 |T| + C_2 |S_{\aux}|$ to be small relative to $N$. The situation is thus delicately balanced. However, one sees that by choosing $T$ in such a way that the dual Selmer group for the non-split residual extension of $\overline \chi_2$ by~$\overline \chi_1$ associated to $\rho$ remains bounded, the size of the set $S_{\textnormal{aux}}$ that must be introduced can also be bounded independently of~$N$, and this makes it possible to close the circle. (See Prop.\ \ref{prop:bounding mod p Selmer group} for this part of the argument.)  

\subsection{Related works} The inspiration for the approach taken in this paper is of course the work of Skinner--Wiles \cite{Ski99} and later generalisations and refinements (e.g. \cite{Tho15, Pan22}). The idea of proving modularity by first introducing auxiliary ramification mod $p^N$, and then lowering the level, is due to Khare \cite{Kha03, Kha04}, and was taken up by the second author in order to prove modularity theorems in a situation where the Taylor--Wiles hypothesis does not hold (although the residual representation is still irreducible) \cite{Tho16}. We also mention that one can prove more refined $R = \mathbb{T}$ type theorems in the residually reducible case under stronger simplifying assumptions, such as associated Selmer groups being 1-dimensional (see e.g.\ \cite{Ski97, Ber13, Deo23}). 

\subsection{Structure of this paper} 
The first two sections are aimed at proving our modularity lifting theorem with auxiliary ramification, which is Thm.\ \ref{thm_tame_SW_black_box}. 
In Section \ref{sec_pseudocharacters} we record some preparatory remarks about pseudocharacters and their deformation theory. In Section \ref{sec_an_ALT} we prove Thm.\ \ref{thm_tame_SW_black_box}, using the argument sketched in Section \ref{subsubsec_tame_SW_method} above. 
As already remarked, the same methods would suffice to prove results also for non-ordinary Galois representations, given e.g.\ the normality of the corresponding local lifting ring at the $p$-adic places. 

In Section \ref{sec: construction of T-new lifts} we give our version of the arguments of \cite{FKPII}, constructing the necessary $T$-new lifts. 
(In fact, we are able to skip some of the hardest parts of the arguments of \emph{op.\ cit.}, based on ``doubling'', since we are starting with a $p$-adic representation $\rho$ and just need to construct lifts of $\rho \text{ mod }\varpi^N$ for sufficiently large integers $N$.) In Section \ref{sec_level_lowering} we give our new level-lowering theorems in the residually reducible case, and show how they combine with the theorem proved in Section \ref{sec: construction of T-new lifts} to give a useful modularity criterion, analogous to e.g. \cite[Cor.\ 4.15]{Tho16}. Finally, in Section \ref{sec_deduction_of_main_theorem} we combine all of the above-mentioned ingredients in order to deduce our main Thm.\ \ref{introthm_automorphy_lifting}.

\subsection{Notation}\label{subsec_notation}

If $F$ is a field, then we will write $\overline{F}$ for a fixed choice of algebraic closure, and if $F$ is perfect, $G_F = \Gal(\overline{F} / F)$. If $F$ is a number field, we will also fix an embedding $\overline{F} \to \overline{F}_v$ extending the canonical embedding $F \to F_v$ for each place $v$ of $F$. This determines an embedding $G_{F_v} \to G_F$ of the decomposition group. If $v$ is a finite place, then we write $\varpi_v \in F_v$ for a choice of uniformizer, $\cO_{F_v} \leq F_v$ for the ring of integers, $k(v) = \cO_{F_v} / (\varpi_v)$ for the residue field, $q_v = | k(v) |$ for the size of the residue field, $\Frob_v \in G_{F_v}$ for a fixed choice of geometric Frobenius lift, and $\Art_{F_v} : F_v^\times \to G_{F_v}^{\textnormal{ab}}$ for the local Artin map, normalised so that $\Art_{F_v}(\varpi_v)$ is a geometric Frobenius lift. If $v$ is an infinite place, then we write $c_v \in G_F$ for the image of the complex conjugation in $G_{F_v}$. 

If $p$ is a prime, then we fix an algebraic closure $\overline{\mathbb{Q}}_p$ of $\mathbb{Q}_p$, and call a finite extension $E / \mathbb{Q}_p$ inside $\overline{\mathbb{Q}}_p$ a coefficient field. We will write $[\cdot] : k \to \cO$ for the Teichm\"uller lift. We say that the coefficient field $E$ is sufficiently large (with respect to a number field $F$) if it contains the image of every embedding $F \hookrightarrow \overline{\mathbb{Q}}_p$. We will denote by $\cO$ the ring of integers of $E$, by $\varpi$ a fixed choice of uniformizer, and set $k = \cO / (\varpi)$. We will write $\CNL_\cO$ for the category of complete Noetherian local $\cO$-algebras $A$ with residue field $k$. We write $\widehat{\GL}_2$ for the group-valued functor on $\CNL_\cO$ that associates to any $A$ the group $\widehat{\GL}_2(A) = \ker(\GL_2(A) \to \GL_2(k))$. A prime $p$ having been fixed, we will write $\epsilon$ for the $p$-adic cyclotomic character (either over $\mathbb{Z}_p$, or over any $\mathbb{Z}_p$-algebra $A$, depending on the context). If $G$ is a group, then we write $G(p)$ for the pro-$p$ completion of $G$. 

We will write $S_p$ for the set of $p$-adic places of the base number field $F$. If $S$ is a finite set of finite places of $F$, then we write $F_S$ for the maximal subextension of $\overline{F} / F$ that is unramified outside $S$ and the infinite places,
and $G_{F, S} = \Gal(F_S / F)$. If $M$ is a discrete $\mathbb{Z}[G_{F, S}]$-module, then we write $H^i(F_S / F, M) = H^i(G_{F, S}, M)$ for the associated Galois cohomology group; we write analogously $H^i(F_v, M) = H^i(G_{F_v}, M)$ in the local case. If a coefficient field has been fixed, then we will also use the notation $h^i = \dim_k H^i$ (when $M$ is a $k[G_{F, S}]$- or $k[G_{F_v}]$-module) and $H^i(F_S / F, \overline{\chi}) = H^i(F_S / F, k(\overline{\chi}))$ (when $\overline{\chi} : G_{F, S} \to k^\times$ is a character). Slightly more generally, if $M$ is an $A$-module of finite length, where $A \in \CNL_\cO$, then we will write $h^i = \operatorname{length}_A H^i$ (if $M$ is annihilated by $\mathfrak{m}_A$, this agrees with the previous usage). 

Let $\mathbb{Z}^2_+ = \{ (\lambda_1, \lambda_2) \in \mathbb{Z}^2 \mid \lambda_1 \geq \lambda_2 \}$. If $v \in S_p$ and $\lambda_v = (\lambda_{v, \tau})_{\tau \in \Hom(F_v, E)} \in (\mathbb{Z}^2_+)^{\Hom(F_v, E)}$, then we say that a continuous representation $\rho_v : G_{F_v} \to \GL_2(E)$ is ordinary of weight $\lambda_v$ 
if there is an isomorphism
\begin{equation}\label{eqn_def_of_ordinary} \rho_v \sim \begin{pmatrix} \psi_{v, 1} & * \\ 0 & \psi_{v, 2} \end{pmatrix}, 
\end{equation}
where for each $i = 1, 2$, $\psi_{v, i} \circ \Art_{F_v}|_{\cO_{F_v}^\times} : \cO_{F_v}^\times \to \overline{\mathbb{Q}}_p^\times$ agrees, on an open subgroup of $\cO_{F_v}^\times$, with the character 
\[ x \mapsto \prod_{\tau \in \Hom(F_v, E)} \tau(x)^{-(\lambda_{\tau, 3-i}+i-1)}. \]
Similarly, if $\rho : G_F \to \GL_2(E)$ is a continuous representation, we say that $\rho$ is ordinary of weight $\lambda \in (\mathbb{Z}^2_+)^{\Hom(F, E)}$ if for every $v \in S_p$, $\rho|_{G_{F_v}}$ is ordinary of weight $\lambda_v$, where $\lambda_v$ is defined by $\lambda_{v, \tau} = \lambda_\tau$ if $\tau \in \Hom(F_v, E) \subset \Hom(F, E)$.

Let $F$ be a totally real number field. We recall that if $\pi$ is a cuspidal, regular algebraic automorphic representation of $\GL_2(\mathbb{A}_F)$ then there is a unique $\lambda \in (\mathbb{Z}^2_+)^{\Hom(F, \mathbb{C})}$ such that $\pi_\infty$ has the same infinitesimal character as the dual of the representation of $(\operatorname{Res}_{F / \mathbb{Q}} \GL_2)_{\mathbb{C}}$ of highest weight $\lambda$; in this case, we say that $\pi$ has weight $\lambda$. If $\iota : \overline{\mathbb{Q}}_p \to \mathbb{C}$ is an isomorphism, then there exists a continuous irreducible representation $r_{\pi, \iota} : G_F \to \GL_2(\overline{\mathbb{Q}}_p)$ which is de Rham and which satisfies compatibility with the local Langlands correspondence at all finite places of $F$. Moreover, if $\pi$ is $\iota$-ordinary, in the sense of \cite{Ger19}, then $r_{\pi, \iota}$ is ordinary of weight $\iota^{-1} \lambda$, where we define $(\iota^{-1} \lambda)_\tau = \lambda_{\iota \tau}$.

\subsection{Acknowledgements} 
A.D. was supported by a Royal Society University Research Fellowship during the writing of this paper.

J.T.'s work was funded by the European Union (ERC CoG-101169866). Views and opinions expressed are however those of the author(s) only and do not necessarily reflect those of the European Union or the European Research Council. Neither the European Union nor the granting authority can be held responsible for them.

\subsection{AI disclosure}
The first version of this paper was written without the use of AI tools.
Afterwards, the second-named author used an LLM to check for typos and mathematical errors, and to investigate the scope of our methods.
This resulted in a machine-generated argument for removing the hypothesis that $(\overline \chi_1 / \overline \chi_2)|_{G_{F(\zeta_p)}} \ne 1$ from Thm.\ \ref{introthm_automorphy_lifting}. The details of this argument are being elaborated, and will appear elsewhere.
None of the text or ideas in this paper was generated by AI. 

\section{Recollections on pseudocharacters}\label{sec_pseudocharacters}

\subsection{Pseudocharacters}\label{subsec:pseudocharacters}
We recall here some basic results on rank-2 pseudocharacters. Let $\Gamma$ be a group, and let $R$ be a $\mathbb{Z}[1/2]$-algebra. By definition, a rank-2 pseudocharacter is a function $t : \Gamma \to R$ (equivalently, an $R$-module  homomorphism $R[\Gamma] \to R$) satisfying the relations 
$t(1) = 2$, $t(\sigma \tau) = t(\tau \sigma)$, and 
    \[ t(\eta \sigma \tau) + t(\eta \tau \sigma) - t(\eta \tau)t(\sigma) - t(\tau \sigma) t(\eta) - t(\sigma \eta) t(\tau) + t(\eta) t(\sigma) t(\tau) = 0 \]
    for all $\eta, \sigma, \tau \in \Gamma$. 
    A pseudocharacter has a determinant $\det t(\sigma) \coloneqq \frac{1}{2}(t(\sigma)^2 - t(\sigma^2))$. We recall that if $\rho : \Gamma \to \GL_2(R)$ is a representation, then $\tr \rho$ is a pseudocharacter; and that if $R$ is an algebraically closed field, then this operation induces a bijection between the set of isomorphism classes of semisimple representations, and the set of pseudocharacters over $R$ (see e.g.\ \cite{Rou96}).
    
    If $\Gamma$ is a profinite group and $R$ is a topological ring, we say that $t$ is continuous if it is continuous as a function $\Gamma \to R$. 
    
There are several different equivalent definitions of pseudocharacter in our context: in terms of ``pseudo matrix entries'' $a(\sigma)$, $d(\sigma)$, and $x(\sigma, \tau)$ as in \cite{Ski99} (this requires the presence of an element of order 2 and trace 0 in $\Gamma$), in terms of group determinants as in \cite{Che14}, and in terms of the geometric invariant theory of the group $\GL_2$ (see \cite{Laf18}, \cite{Eme23}). We will make use of the first two here. The matrix entry definition is useful, as in \cite{Ski99, Pan22}, in analysing quotients of pseudodeformation rings, while the group determinant formulation is useful in analysing finiteness properties of such rings. Finally, an important role is played by the generalized matrix algebras associated to a pseudocharacter, as described in \cite{Bel09} and \cite{Che14}. 

\subsection{Pseudo matrix entries}\label{subsec:pseudo matrix entries}

Let $t : \Gamma \to R$ be a rank-2 pseudocharacter, and suppose given an element $c_\infty \in \Gamma$ of order 2 such that $t(c_\infty) = 0$. (The notation is supposed to suggest complex conjugation at an infinite place.) For any $z, w \in R[\Gamma]$ we can define further elements of $R$ by the formulae
\[ a(z) \coloneqq \frac{1}{2}(-t(c_\infty z) + t(z)), \]
\[ d(z) \coloneqq \frac{1}{2}(t(c_\infty z) + t(z)), \]
and
\[ x(z, w) \coloneqq a(z w) - a(z) a(w). \]
(We are essentially  following \cite[\S 2]{Pan22} here, but have extended these functions from $\Gamma$ to $R[\Gamma]$, which seems to be necessary at some points of the argument, and reversed the roles of $a$ and $d$, because we prefer our residual representations to be upper-triangular, but our main argument requires us to take $a(c_\infty) = -1$ and $d(c_\infty) = 1$.)
Suppose that $t$ is continuous and that $R$ is either a topological field, or a complete discrete valuation ring with its valuation topology. Then we can always find a continuous representation $\rho : \Gamma \to \GL_2(R)$ with $\tr \rho = t$. We recall the recipe from \cite[\S 2.1.5]{Pan22} in the case that $R$ is a complete discrete valuation ring with uniformizer $\varpi_R$. If all elements $x(\sigma, \tau)$ are 0 in $R$, then we can define
\begin{equation}\label{eq_rep_of_reducible_ps} \rho(\sigma) \coloneqq  \begin{pmatrix} a(\sigma) & 0 \\ 0 & d(\sigma) \end{pmatrix}. 
\end{equation}
If there exist $\sigma, \tau$ with $x(\sigma, \tau) \neq 0$, then choose $x(\sigma_0, \tau_0)$ with minimal $\varpi_R$-adic valuation. Then we can define
\begin{equation}\label{eq_rep_of_irreducible_ps} \rho(\sigma) \coloneqq  \begin{pmatrix} a(\sigma) & x(\sigma, \tau_0) / x(\sigma_0, \tau_0) \\ x(\sigma_0, \sigma) & d(\sigma) \end{pmatrix}. 
\end{equation}
In the case that $R$ is a field, this representation is reducible if $x(\sigma, \tau) = 0$ for all $\sigma, \tau \in \Gamma$, and (absolutely) irreducible otherwise. This uses the existence of $c_\infty$. 
In general, if $t$ is a pseudocharacter over a ring $R$, then we define the reducibility ideal $\mathcal{R} \leq R$ of $t$ to be the ideal generated by all elements $x(\sigma, \tau)$ ($\sigma, \tau \in \Gamma$): then a prime ideal $\mathfrak{p}$ contains $\mathcal{R}$ if and only if the (unique up to isomorphism) semisimple representation $\Gamma \to \GL_2(\overline{k(\mathfrak{p})})$ with trace $t$ is reducible.
(Here, and in what follows, we write~$k(\mathfrak p)$ for the fraction field of $R/\mathfrak{p}$.)

\subsection{Cayley--Hamilton algebras}\label{subsec:pseudocharacters for level-lowering}
Let $p$ be an odd prime, and let $E / \mathbb{Q}_p$ be a coefficient field.
Let~$\Gamma$ be a group, $R$ a $\mathbb{Z}[1/2]$-algebra, 
and $t: \Gamma \to R$ a pseudocharacter of rank~$2$.
The aim of this subsection is to prove Prop.\ \ref{prop_pseudocharacters_with_large_image}, which gives a technical property of modules over the Cayley--Hamilton algebra associated to~$t$, and which will be useful in the level-lowering arguments of Section~\ref{sec_level_lowering}.
We begin with some preliminary material.

As in~\cite[Ex.\ 7.6]{Che14}, we associate to~$t$ a $2$-dimensional determinant (i.e.\ a multiplicative homogeneous $R$-polynomial law of degree~$2$) denoted
$D : R[\Gamma] \to R$.
By~\cite[Ex.\ 1.2(ii)]{Che14}, the evaluation of~$D$ at~$R$ is a quadratic form $D_R: R[\Gamma] \to R$, and it determines~$D$ uniquely.
More precisely, if~$B$ denotes the symmetric $R$-bilinear form on~$R[\Gamma]$ defined by
\[
(x, y) \mapsto D_R(x+y)-D_R(x)-D_R(y),
\]
then, for all~$n \geq 1$ and $x_1, \ldots, x_n \in R[\Gamma]$, we have
\begin{equation}\label{eqn:quadratic form}
D_{R[T_1, \ldots, T_n]}(\sum_{i=1}^n x_i \otimes T_i) = \sum_{i=1}^n D_R(x_i)T_i^2 + \sum_{i < j}B(x_i, x_j)T_iT_j,
\end{equation}
cf.\ \cite[Prop.\ II.1]{Rob63}.

Since~$D$ corresponds to the pseudocharacter~$t$, \cite[Lem.\ 1.9]{Che14} shows that
$B(g, h) = D_R(h)t(gh^{-1})$ for all~$g, h \in \Gamma$, and $D_R : \Gamma \to R^\times$ is a 
group homomorphism satisfying $D_R(g) \coloneqq  \frac 1 2(t(g)^2-t(g^2))$, by~\cite[Lem.\ 1.9]{Che14}. Hence~$D_R|_{\Gamma}$ coincides with~$\det(t)$ in the sense of Section~\ref{subsec:pseudocharacters}.

If~$g \in \Gamma$, we define $g^* \coloneqq  D_R(g)g^{-1} \in R[\Gamma]$.
This extends uniquely to an $R$-linear anti-involution of~$R[\Gamma]$, denoted~$x \mapsto x^*$.
If we extend~$t$ to an $R$-linear function $R[\Gamma] \to R$, we thus see that
\begin{equation}\label{eqn:quadratic form and involution}
B(x, y) = t(xy^*),
\end{equation}
for all~$x, y \in R[\Gamma]$.

As usual, if~$x \in R[\Gamma]$ we define the $D$-characteristic polynomial of~$x$ as%
\[
f_x(T) \coloneqq  D_{R[T]}(T-x),
\]
which coincides with $T^2 -t(x^*)T + D_R(x)$.
Note that there is a homogeneous $R$-polynomial law of degree~$2$ defined by
\[
\chi: R[\Gamma] \to R[\Gamma] \]
\[ x \mapsto x^2-t(x^*)x+D_R(x).
\]
Following~\cite[Section~1.4]{Che14}, we define the Cayley--Hamilton ideal $\mathcal{I}_{\CH}(t) \subset R[\Gamma]$ as the
two-sided ideal generated by the coefficients of %
\[
\chi_{R[T_1, \ldots, T_n]}(\sum_{i=1}^n x_i \otimes T_i) \in R[\Gamma] \otimes_R R[T_1, \ldots, T_n], 
\]
as~$n$ runs through the non-negative integers, and $(x_1, \ldots, x_n)$ runs through $R[\Gamma]^{\times n}$.
Since~$\chi$ has itself degree~$2$, 
a consideration of~\eqref{eqn:quadratic form}
with~$D$ replaced by~$\chi$
shows that this ideal is generated by $(\chi_R(x): x \in R[\Gamma])$, i.e.\ by $(f_x(x): x \in R[\Gamma])$.

\begin{lemma}\label{lem:computation of CH ideals}
    Let~$\Gamma$ be a group, let~$R$ be a $\mathbb Z[1/2]$-algebra, and let $t : \Gamma \to R$ be a rank-two pseudocharacter of~$\Gamma$.
    Let~$J$ be the two-sided ideal of~$R[\Gamma]$ generated by $(f_g(g): g \in \Gamma)$.
    Then~$\mathcal{I}_{\CH}(t) = J$.
\end{lemma}
\begin{proof}
    We follow the proof of~\cite[Prop.\ 1]{BLR}.
   
    We see from~\eqref{eqn:quadratic form and involution}, and the fact that~$B$ is symmetric, that $t(x) = t(x^*)$ for all~$x \in \Gamma$, hence for all~$x \in R[\Gamma]$.
    Similarly, the fact that~$D$ is homogeneous of degree~$2$ implies that $D_R(g^*) = D_R(g)$ for all~$g \in \Gamma$.
    Hence~$*$ preserves~$J$, and so it induces an anti-involution of $R[\Gamma]/J.$

    We now show that, for all~$x \in R[\Gamma]$, we have $x + x^* = t(x)$ and $x^*x = D_R(x)$ in~$R[\Gamma]/J$.
    These identities imply that $f_x(x) = 0$ in $R[\Gamma]/J$, and so $\mathcal{I}_{\CH}(t) \subset J$, as desired.

    The equality $x+x^* = t(x)$ follows by $R$-linearity from the case in which $x \in \Gamma$, which holds because $x^2-t(x)x+D_R(x) = 0$ in $R[\Gamma]/J$.

    To check that $x^*x = D_R(x),$ it suffices to check that $\{x\in R[\Gamma]/J : x^*x = D_R(x)\}$ is additively closed (since it contains $rx$ for all~$r \in R, x \in \Gamma$).
    Computing in $R[\Gamma]/J$, we have
    \[
    (x+y)^*(x+y) = x^*x + y^*y + x^*y+y^*x = x^*x+y^*y+t(x^*y),
    \]
    where the second equality follows from the trace identity in~$R[\Gamma]/J$ proved in the previous paragraph.
    So it suffices to check that $D_R(x+y) = D_R(x)+D_R(y)+t(x^*y)$ for all~$x, y \in R[\Gamma]$.
    This is a consequence of~\eqref{eqn:quadratic form} with~$n = 2$, as well as~\eqref{eqn:quadratic form and involution} and the fact that~$t$ is a central function.
\end{proof}

\begin{rem}\label{rem:higher rank pseudocharacters}
    As is well-known, Lem.\ \ref{lem:computation of CH ideals} can fail when the rank of~$t$ is greater than~$2$, even for absolutely irreducible pseudocharacters valued in~$\cO$: see~\cite[Section~5]{BLR} for a counterexample.
\end{rem}

\begin{pp}\label{prop_pseudocharacters_with_large_image}
    Let $\Gamma$ be a profinite group, and let $\rho : \Gamma \to \GL_2(\cO)$ be a continuous representation such that $\rho \otimes_\cO E$ is absolutely irreducible. We extend $\rho$ to an $\cO$-algebra homomorphism $\cO[\Gamma] \to M_2(\cO)$. Suppose that $\overline{\rho}^{\textnormal{ss}}$ is isomorphic to a direct sum of two distinct characters. Then:
    \begin{enumerate}
        \item There exists an integer $C \geq 0$ such that the image $\rho(\cO[\Gamma]) \leq M_2(\cO)$ contains $\varpi^C M_2(\cO)$. 
        \item Let $t \coloneqq  \operatorname{tr}(\rho)$, and let $S \coloneqq  \cO[\Gamma]/\mathcal{I}_{\CH}(t)$ denote the maximal Cayley--Hamilton quotient of $\cO[\Gamma]$. 
        Then for any $C$ as in (1), and any $x \in S$ such that $\rho(x) = 0$, we have $\varpi^{2C} x = 0$; and for any $x \in S$ and $N \geq 2C$ such that $\rho(x) \equiv 0 \text{ mod }\varpi^N$, we have $\varpi^{2C} x \in \varpi^N S$. 
        \item Let $N \geq 2C$, let $M$ be an $S / (\varpi^N)$-module,  and let $y \in M$ be an element of exact order $\varpi^N$. Let $z \in S$ be an element that annihilates $S y$. Then $\rho(z) \equiv 0 \text{ mod } \varpi^{N-2C}$.
        \item Let $N \geq 2C$, let $M$ be an $S / (\varpi^N)$-module, and let $y \in M$ be an element of exact order $\varpi^N$. Let $\phi \in \Gamma$ be such that the characteristic polynomial of~$\rho(\phi)$ is $(X-\alpha)(X-\beta)$, where $\alpha, \beta \in \cO$ satisfy $\mathrm{ord}_\varpi(\alpha-\beta) \leq D$, for some $D \geq 0$ with $D < N - 2C$. Then $(S y)[\phi - \alpha]$ contains an $\cO$-submodule isomorphic to $\cO / (\varpi^{N - 2C - D})$. 
    \end{enumerate}
\end{pp}
\begin{proof}
(1) is simply Burnside's theorem (which says that, since $\rho \otimes_\cO E$ is absolutely irreducible, the map $E[\Gamma] \to M_2(E)$ is surjective).

For (2), we may apply \cite[Lem.\ 1.4.3]{Bel09} to conclude the existence of idempotents $e_1, e_2 \in S$ such that $e_1 + e_2 = 1$ and the restriction of $t = \mathrm{tr} \rho$ to $e_1 S e_1$ or $e_2 S e_2$ is a rank-one pseudocharacter. After conjugating $\rho$ by an element of $\GL_2(\cO)$, we can assume that $\rho(e_i)$ is the diagonal matrix with $1$ in the $(i, i)$-entry (for each $i = 1, 2$). 
It also follows from \emph{loc.\ cit.} that $t$ induces an isomorphism $e_i S e_i \to \cO$ (for each $i = 1, 2$). 

Using (1), we can find $b, c \in S$ such that $\rho(b) = \left( \begin{smallmatrix} 0 & \varpi^C \\ 0 & 0 \end{smallmatrix}\right)$ and $\rho(c) = \left( \begin{smallmatrix} 0 & 0 \\ \varpi^C & 0 \end{smallmatrix}\right)$. After multiplying on the left and right, we can assume moreover that $b = e_1 b e_2$ and $c = e_2 c e_1$. 
Now, since $x \in \ker \rho \leq \ker t$, if we write  $x = e_1 x e_1 + e_1 x e_2 + e_2 x e_1 + e_2 x e_2$, then each summand also lies in $\ker \rho$. Since $t$ induces an isomorphism $e_i S e_i \to \cO$, we have $e_1 x e_1 = e_2 x e_2 = 0$ and $c b = \varpi^{2C} e_2$. We also have $\rho(e_1 x e_2 c) = 0$ and $e_1 x e_2 c \in e_1 S e_1$, hence $e_1 x e_2 c = e_1 x c = 0$. It follows that
\[ \varpi^{2C} e_1 x e_2 =  e_1 x e_2 c b = e_1 x c b = 0. \]
Reversing the roles of $e_1$ and $e_2$ shows too that $\varpi^{2C} e_2 x e_1 = 0$, hence $\varpi^{2C} x = 0$, as desired.

Next suppose instead that $N \geq 2C$ and $\rho(x) \equiv 0 \text{ mod }\varpi^N$. Let us write $x = \left( \begin{smallmatrix} x_{11} & x_{12} \\ x_{21} & x_{22} \end{smallmatrix} \right)$, with $x_{ij} = e_i x e_j$. %
Then $x_{11} \in \varpi^N \cO e_1$ and $x_{22} \in \varpi^N \cO e_2$. We have $x_{12} c \in e_1 S e_1$ and $\rho(x_{12} c) \equiv 0 \text{ mod }\varpi^{N}$, hence $x_{12} c \in \varpi^{N} \cO e_1$, hence $x_{12} c  b = \varpi^{2 C} x_{12} \in \varpi^{N} \cO b$, hence $\varpi^{2C} x_{12} \in \varpi^{N} S$. Similarly we see $\varpi^{2C} x_{21} \in \varpi^N S$, hence $\varpi^{2C} x \in \varpi^N S$. This completes the proof of (2). 

To prove (3), note first that at least one of $e_1 y$, $e_2 y$ must have exact order $\varpi^N$; suppose without loss of generality that it is $e_1 y$. If $z$ annihilates $S y$, then it annihilates $S e_1 y$. Continuing with the notation of the above paragraphs, let us write $z = \left( \begin{smallmatrix} z_{11} & z_{12} \\ z_{21} & z_{22} \end{smallmatrix} \right)$, where $z_{ij} = e_i z e_j$. 
Then $z_{11} e_1 y = 0$, $z_{12} c e_1 y = 0$, $z_{21} e_1 y = 0$, and $z_{22} c e_1 y = 0$. 

On the other hand, we have $e_1 S e_1 y = \cO / \varpi^N \cO \cdot e_1 y$. 
Writing $z_{11} = a e_1$ and $z_{22} = d e_2$ with $a, d \in \cO$, we have $z_{11} e_1 y = a e_1 y = 0$, hence $a \in \varpi^N \cO$. 
We have $b z_{22} c = d b c = d \varpi^{2C} e_1$, hence $b z_{22} c e_1 y = d \varpi^{2C} e_1 y = 0$, showing that $d \in \varpi^{N -2C} \cO$, and $\rho(z_{22}) \equiv 0 \text{ mod }\varpi^{N - 2C}$. 
We also have $ z_{12} c e_1 y = 0$, hence $z_{12} c \in \varpi^N \cO e_1$; 
multiplying on the right by~$b$ we deduce that $\varpi^{2C} z_{12} \in \varpi^N \cO b$, and $\rho(z_{12}) \equiv 0 \text{ mod }\varpi^{N-C}$. 
Finally, we have $z_{21} e_1 y = 0$, hence $b z_{21} e_1 y = 0$, $bz_{21} \in \varpi^N \cO e_1$, $\varpi^{2C} z_{21} \in \varpi^N \cO c$, hence $\rho(z_{21}) \equiv 0 \text{ mod }\varpi^{N-C}$. We have shown that $\rho(z) \equiv 0 \text{ mod }\varpi^{N-2C}$, as required.
This concludes the proof of~(3).

To prove (4), %
observe that $(\phi-\beta) (Sy) \leq (Sy)[\phi-\alpha]$, so it suffices to show that $\varpi^{N-2C - D - 1} (\phi-\beta) (Sy) \neq 0$. Suppose for contradiction that $\varpi^{N-2C - D - 1} (\phi-\beta) (Sy) = 0$, and apply (3) with $z = \varpi^{N-2C - D - 1}(\phi-\beta)$ to conclude that $\rho(z) \equiv 0 \text{ mod } \varpi^{N - 2C}$, hence that $\rho(\phi-\beta) \equiv 0 \text{ mod }\varpi^{D+1}$. This implies that $\mathrm{tr} \rho(\phi) = \alpha + \beta \equiv 2 \beta \text{ mod }\varpi^{D+1}$, hence that $\alpha \equiv \beta \text{ mod }\varpi^{D+1}$. This contradiction completes the proof. 
\end{proof}

\subsection{Pseudodeformations}
Suppose that~$p$ is an odd prime, and that $E / \mathbb{Q}_p$ is a coefficient field. 
Assume that~$\Gamma$ satisfies Mazur's condition~$\Phi_p$, and
fix a continuous residual pseudocharacter $\overline{t} : \Gamma \to k$. 
We can consider the functor $\operatorname{PDef} : \operatorname{CNL}_\cO \to \operatorname{Sets}$ of continuous pseudocharacters lifting $\overline{t}$ (otherwise said, pseudodeformations). This is representable by an object $P(\Gamma) \in \operatorname{CNL}_\cO$. 

Let $\Delta \leq \Gamma$ be a normal open subgroup such that $\overline{t}$ factors through $\Gamma / \Delta$. Let $P_1(\Delta) = \Delta$ and, for $j \geq 1$, define $P_{j+1}(\Delta) \coloneqq  \overline{P_j(\Delta)^p [ P_1(\Delta), P_j(\Delta)]}$, an open normal subgroup of $\Gamma$ (cf.\ \cite[Ch.\ I, \S 4.2, Ex.\ 6]{Ser02}; here we use that $\Gamma$ satisfies $\Phi_p$).
Then $(P_j(\Delta))_{j \geq 1}$ is the lower $p$-central series of $\Delta$, and the inverse limit
\[ \varprojlim_j \Delta/P_j(\Delta) \]
may be identified with the pro-$p$ completion~$\Delta(p)$ of $\Delta$. Let $P_\infty(\Delta) \coloneqq  \cap_{j \geq 1} P_j(\Delta)$. We have the following useful lemma.
\begin{lemma}\label{lem_quantitative_continuity_of_pseudocharacters}
    With the notation and assumptions of the previous paragraph:
    \begin{enumerate}
        \item Any pseudodeformation of $\overline{t}$ factors through the quotient $\Gamma / P_\infty(\Delta)$.
        \item Let $A \in \CNL_\cO$ be such that $\mathfrak{m}_A^N = 0$, and let $t : \Gamma \to A$ be a pseudodeformation of $\overline{t}$. Then $t$ factors through $\Gamma / P_{3N}(\Delta)$. 
    \end{enumerate}
\end{lemma}
\begin{proof}
    (1) is well-known -- see e.g. \cite[Lem.\ 3.8]{Che14}. For (2), let $R \coloneqq  A[\Gamma]/\mathcal{I}_{\CH}(t)$ be the maximal Cayley--Hamilton quotient of~$A[\Gamma]$,  and
     let $\overline{\rho} : \Gamma \to \GL_2(k)$ be a semisimple representation giving rise to $\overline{t}$. Then $t$ factors through an $A$-linear map $t' : R \to A$. Let $I$ denote the image of $\ker \overline{\rho}$ in $R$. Then $I$ is a 2-sided ideal, equal to the radical of $R$, and \cite[Lem.\ 2.10]{Che14} shows that $I^{3 N} = 0$. 

     Let $P_j'(\Delta)$ denote the sequence of normal subgroups of $\Delta$ defined by $P'_1(\Delta) = \Delta$, $P'_{j+1}(\Delta) = P'_j(\Delta)^p [ P'_1(\Delta), P'_j(\Delta) ]$ if $j \geq 1$. We will show that $t$ factors through $\Gamma / P'_{3N}(\Delta)$. Since $t$ is continuous, this will imply that it in fact factors through $\Gamma / P_{3N}(\Delta)$, since the closure of $P'_j$ is $P_j$ (as follows from an easy induction).

    Let $\rho : \Gamma \to R^\times$ denote the tautological homomorphism. We have $\rho(\Delta) \leq 1 + I$, by construction. We show by induction that $\rho(P'_j(\Delta)) \leq 1 + I^j$ for all $j \geq 1$. It suffices, since $\rho$ is continuous, to show that if $x\in P'_j(\Delta)$ and $y \in P'_1(\Delta)$, then $\rho(x^p) \in 1 + I^{j+1}$ and $\rho([x, y]) \in 1 + I^{j+1}$. Suppose that $\rho(x) = 1 + z$, $\rho(y) = 1 + w$, where $z \in I$, $w \in I^j$. Then $\rho(x^p) = 1 + \sum_{k=1}^p {p \choose k} z^k$, and we see that each term in the sum lies in $I^{j+1}$. On the other hand, we have 
    \[ (1+z)(1+w)(1+z)^{-1}(1+w)^{-1} \equiv (1+z)(1+w)(1-z+\dots+(-z)^j)(1-w) \equiv 1 \text{ mod }I^{j+1}. \] 
    Since $I^{3N} = 0$, it follows that $\rho(P'_{3N}(\Delta)) = 1$. We have $t = t' \circ \rho$, so if $\gamma \in \Gamma$, $\delta \in P_{3N}'(\Delta)$, it follows that $t( \gamma \delta ) = t'(\rho(\gamma \delta)) = t'(\rho(\gamma)) = t(\gamma)$, as required. 
\end{proof}
Suppose now that there is a non-trivial quadratic character $\chi : \Gamma \to \{ \pm 1 \}$ such that $\overline{t} = \overline{t} \otimes \chi$, and let $K \coloneqq  \ker \chi$. In this case, we can form the quotient $P(\Gamma)^{\chi\textnormal{-dih}}$ of $P(\Gamma)$ defined by imposing the relation $t = t \otimes \chi$ on the universal pseudocharacter. (This is supposed to capture the pseudocharacters that are ``dihedral, induced from $K$''. The pseudocharacter $t \otimes \chi$ is of course defined as $(t \otimes \chi)(\sigma) \coloneqq  \chi(\sigma)t(\sigma)$ for all~$\sigma \in \Gamma$.) 
\begin{pp}\label{prop_universal_dihedral_pseudocharacter}
    Let $\Lambda \coloneqq  \cO\llbracket  K^{\textnormal{ab}}(p) \rrbracket $, let $c \in \Gamma - K$, and let $\sigma : \Lambda \to \Lambda$ be the involution induced by the map $\lambda \mapsto c \lambda c^{-1}$ on group elements. Suppose that there are characters $\overline{\chi}_1, \overline{\chi}_2 : \Gamma \to k^\times$ such that $\overline{t} = \overline{\chi}_1 + \overline{\chi}_2$. Then there is an isomorphism $P(\Gamma)^{\chi\textnormal{-dih}} \cong \Lambda^\sigma$.
\end{pp}
\begin{proof}
    After twisting, we can assume that $\overline{\chi}_1 = 1$ and $\overline{\chi}_2 = \chi$. Let $X : K \to \Lambda^\times$ denote the universal character. Taking the pseudocharacter of the induction $\Ind_K^\Gamma X$ over $\Lambda$ gives a map $P(\Gamma) \to \Lambda$ that factors through the quotient $P(\Gamma)^{\chi\textnormal{-dih}}$ and takes values in $\Lambda^\sigma$. We need to show that the induced map $P(\Gamma)^{\chi\textnormal{-dih}} \to \Lambda^\sigma$ is an isomorphism. We begin by noting that it is surjective: indeed, if $g \in K$ then $t(g) \in P(\Gamma)^{\chi\textnormal{-dih}}$ is sent to $X(g) + X(c g c^{-1}) = g + c g c^{-1} \in \Lambda^\sigma$, and elements of this form span a dense $\cO$-submodule of $\Lambda^\sigma$.

    Let $A = P(\Gamma)^{\chi\textnormal{-dih}}$, and let $R$ denote the maximal Cayley--Hamilton quotient of $A[\Gamma]$. Then there is an anti-involution on $R$, given by the formula $x^\ast \coloneqq  t(x) - x$ (cf. the proof of Lem.\ \ref{lem:computation of CH ideals}). There is also an involution $\iota$, given by the formula $\iota(g) = \chi(g) g$ on group elements. Thus there is a $\mathbb{Z} / 2 \mathbb{Z}$-grading $R = R_0 \oplus R_1$, where $R_0$ is precisely the image of $A[K]$ in $R$. Note that $R_0$ is preserved by $(\cdot)^\ast$.

    Let $x \in R_0$. Then $t(xc) = 0$, hence $(xc)^\ast = - x c = c^\ast x^\ast = -c x^\ast$, hence $x c - c x^\ast = (x - c x^\ast c^{-1}) c = 0$ in $R$. The element $c$ is a unit, so we conclude that $x = c x^\ast c^{-1}$ for all $x \in R_0$. If $x, y \in R_0$, we conclude that $(x y)^\ast = x^\ast y^\ast$, hence (since $(\cdot)^\ast$ is an anti-involution) that $R_0$ is commutative. 

    It follows that the homomorphism $K \to R_0^\times$ factors through a homomorphism $K^{\textnormal{ab}} \to R_0^\times$. In fact, $R_0$ is a finite $A$-algebra, so this homomorphism induces a homomorphism $\Lambda \to R_0$. By construction, the restriction of this map to $\Lambda^\sigma$ sends an element $g + c g c^{-1}$ ($g \in K$) to $g + c g c^{-1} = g + g^\ast = t(g) \in P(\Gamma)^{\chi\textnormal{-dih}} \leq R$. These elements span a dense $\cO$-submodule of $\Lambda^\sigma$, so we get an induced homomorphism $\Lambda^\sigma \to P(\Gamma)^{\chi\textnormal{-dih}}$. This gives a splitting to the surjective map $P(\Gamma)^{\chi\textnormal{-dih}} \to \Lambda^\sigma$, which must therefore be an isomorphism. 
\end{proof}

\subsection{Pseudodeformations and deformations}\label{subsec:pseudodeformations and deformations}

Continue to assume that~$p$ is an odd prime, and that $E/\bQ_p$ is a coefficient field. 
Fix a surjection $\Gamma \to \Gamma_0$ of profinite groups satisfying Mazur's condition $\Phi_p$, and choose an element $c_\infty \in \Gamma$ of order 2. We suppose given an indecomposable residual representation
\[ \overline{\rho}_b= \begin{pmatrix}  \overline{\chi}_1 & * \\ 0 & \overline{\chi}_2 \end{pmatrix} : \Gamma_0 \to \GL_2(k) \]
such that $\overline{\rho}_b(c_\infty) = \operatorname{diag}(-1, 1)$; and we set $\overline{t} \coloneqq  \tr \overline{\rho}_b$. 
In this case we can consider the functor $\operatorname{Def}_b : \operatorname{CNL}_\cO \to \operatorname{Sets}$ of deformations of the $\Gamma$-representation $\overline{\rho}_b$, i.e.\ $\widehat \GL_2$-conjugacy classes of lifts $\rho : \Gamma \to \GL_2$ of~$\rhobar_b$, 
which is represented by an object $R_b = R_b(\Gamma) \in \CNL_\cO$. 

The universal property of $P\coloneqq  P(\Gamma)$ implies the existence of a morphism $P \to R_b$ in $\operatorname{CNL}_\cO$. We write $\rho_b^{\textnormal{univ}} : \Gamma \to \GL_2(R_b)$ for a representative of the universal deformation, chosen so that $\rho_b^{\textnormal{univ}}(c_\infty) = \operatorname{diag}(-1, 1)$. 
This condition implies that, for any~$\sigma, \tau \in \Gamma$, we have
    \begin{equation}\label{eqn:pseudo and actual matrix entries}
    \rho_{b}^{\textnormal{univ}}(\sigma)_{11} = a(\sigma), \rho_{b}^{\textnormal{univ}}(\sigma)_{22} = d(\sigma), \rho_{b}^{\textnormal{univ}}(\sigma)_{12}\rho_{b}^{\textnormal{univ}}(\tau)_{21} = x(\sigma, \tau).
    \end{equation}
    
If $\mathfrak{q}_b \leq R_b$ is a prime ideal, we write $\rho_{\mathfrak{q}_b} : \Gamma \to \GL_2(R_b / \mathfrak{q}_b)$ for the specialisation of $\rho_b^{\textnormal{univ}}$. 
The following lemma is used to compare the universal pseudodeformation and deformation rings after localization and completion at~$\mathfrak{q}_b$.

\begin{lemma}\label{lem_comparison_of_P_and_R_b}
    Let $\mathfrak{q}_b \leq R_b$ be a prime ideal, and let $\mathfrak{q}$ denote its pullback to $P$. Then:
    \begin{enumerate}
        \item If $\mathfrak{q}_b$ has dimension 1 and $\rho_{\mathfrak{q}_b} \otimes k(\mathfrak{q}_b)$ is irreducible, then the natural map 
        $\widehat{(P_{\mathfrak{q}})} \to \widehat{(R_{b, \mathfrak{q}_b})}$ is an isomorphism, and $\Spec R_{b, \mathfrak{q}_b} \to \Spec P_{\mathfrak{q}}$ is surjective. 
        \item If $\rho_{\mathfrak{q}_b} \otimes k(\mathfrak{q}_b)$ is irreducible, then $\dim R_b / \mathfrak{q}_b \leq \dim P / \mathfrak{q}$. Furthermore, if $\mathfrak{q}_b$ is a minimal prime, then so is $\mathfrak{q}$. 
    \end{enumerate}
\end{lemma}
We remark that $R_{b, \mathfrak{q}_b}$ denotes the localisation of $R_b$ at the prime ideal $\mathfrak{q}_b$, and $\widehat{(R_{b, \mathfrak{q}_b})}$ denotes the completion of this local ring at its maximal ideal. 
\begin{proof}
    This is \cite[Cor.\ 2.3.6]{Pan22}.
\end{proof}
We also give a more technical result that is a slight generalisation of~\cite[Prop.\ 2.4.2]{Pan22}. 
\begin{pp}\label{prop_height_of_ideal_in_P}
    Let $\mathfrak{q}_b \leq R_b$ be a prime ideal such that $\rho_{\mathfrak{q}_b} \otimes k(\mathfrak{q}_b)$ is irreducible, 
    and suppose given $z \in \cO[\Gamma]$ such that $\rho_{\mathfrak{q}_b}(z)_{12} = 1$. Let $\mathfrak{q}$ denote the pullback of $\mathfrak{q}_b$ to $P$, let $I_c \leq P / \mathfrak{q}$ denote the ideal generated by the elements $x(z, \tau)$ ($\tau \in \Gamma$), and suppose that $R_b / (\mathfrak{q}_b, \mathfrak{m}_P)$ is not Artinian. Then $\operatorname{ht}(I_c) \leq 1$. 
\end{pp}
\begin{proof}
    We follow closely the proof of \cite[Prop.\ 2.4.2]{Pan22}.%
    By definition, we must show that there exists a prime of $P / \mathfrak{q}$ containing $I_c$ of height less than or equal to 1. Let $B_1$ denote the normalisation of $P / \mathfrak{q}$ in $k(\mathfrak{q})$. Since $P / \mathfrak{q}$ is complete local Noetherian, $B_1$ is a finite, integral $P / \mathfrak{q}$-algebra. We will in fact show that $I_c B_1$ has height $\leq 1$; this implies that $I_c$ has height $\leq 1$ by \cite[\href{https://stacks.math.columbia.edu/tag/02IJ}{Tag 02IJ}]{stacks-project} (and since complete Noetherian local rings are universally catenary). 

    By construction, we have a diagram of inclusions
    \[
    \xymatrix{
    &R_b/\mathfrak{q}_b \ar[dr]&\\
    P \ar[ur] \ar[dr]/\mathfrak{q}&&k(\mathfrak{q}_b)\\
    &B_1 \subset k(\mathfrak{q}) \ar[ur]&
    }
    \]
    Let $B_1'$ denote the intersection of $B_1$ with $R_b / \mathfrak{q}_b$ inside $k(\mathfrak{q}_b)$. Then $B_1'$ is a finite $P / \mathfrak{q}$-algebra (because so is~$B_1$) and so $B_1' / (\mathfrak{m}_P)$ is Artinian. 

    Suppose for contradiction that $I_c B_1$ has height greater than 1. Then for any height 1 prime $\mathfrak{p}$ of $B_1$, there exists $\tau \in \Gamma$ such that $x(z, \tau) \not\in \mathfrak{p}$. If $\sigma \in \Gamma$, then we find that
    \[ \rho_{\mathfrak{q}_b}(\sigma)_{12} = \frac{ \rho_{\mathfrak{q}_b}(\sigma)_{12} \rho_{\mathfrak{q}_b}(\tau)_{21} }{ \rho_{\mathfrak{q}_b}(\tau)_{21} } = \frac{ x(\sigma, \tau) }{x(z, \tau)} \in (B_1)_{\mathfrak{p}}. \]
    (The second equality used the fact that $\rho_{\mathfrak{q_b}}(z)_{12} = 1$.)
    Since $\mathfrak{p}$ is arbitrary amongst height 1 primes of~$B_1$, and $B_1$ is normal, it follows that in fact $\rho_{\mathfrak{q}_b}(\sigma)_{12} \in B_1$ for all $\sigma \in \Gamma$; hence $\rho_{\mathfrak{q}_b}$ takes values in $\GL_2(B_1')$, since the remaining matrix entries are contained in $P/\mathfrak{q}$, by~\eqref{eqn:pseudo and actual matrix entries}. 
    Since $B_1'$ is a complete, hence closed subring of $R_b / \mathfrak{q}_b$,
    and the entries of $\rho_{\mathfrak{q}_b}$ are topological generators of $R_b/\mathfrak{q}_b$,
    it follows that $B_1' = R_b / \mathfrak{q}_b$, and therefore that $B_1' / (\mathfrak{m}_P) \cong R_b / (\mathfrak{q}_b, \mathfrak{m}_P)$. Since the source of this isomorphism is Artinian and the target is not, this is a contradiction. This contradiction completes the proof. 
\end{proof}

\section{An automorphy lifting theorem}\label{sec_an_ALT}

In this section, we develop our tame analogue of the Skinner--Wiles argument, proving our main automorphy lifting theorem (Thm.\ \ref{thm_tame_SW_black_box}) below. We carry out Taylor--Wiles patching twice: first, in Section \ref{subsec:first patching argument}, to boost the size of the (pseudo-)deformation ring relative to the size of the reducible locus; second, in Section \ref{subsec_second_patching}, to actually prove our ``$\widehat{(R_\mathfrak{p})} = \widehat{(\mathbb{T}_{\mathfrak{p}})}$'' theorem. To be more precise, the second stage may be carried out many times, in order to propagate pro-modularity throughout the whole of the patched (pseudo-)deformation ring constructed during the first stage. 

\subsection{Statement of main result}

 Let $p$ be an odd prime, let $F$ be a totally real field, and let $E / \mathbb{Q}_p$ be a (sufficiently large) coefficient field. In this section, we consider a continuous, absolutely irreducible representation $\rho : G_F \to \GL_2(E)$ and a weight $\mathbf{\lambda} \in (\mathbb{Z}^2_+)^{\Hom(F, E)}$ satisfying the following conditions:
 \begin{itemize}
     \item $\rho$ is unramified at all but finitely many places. 
     \item $\det \rho = \epsilon^{-1}$.
     \item For each $v \in S_p$, $\rho|_{G_{F_v}}$ is ordinary of weight $\mathbf{\lambda}_v$.
     \item Writing $\overline{\rho}^{\textnormal{ss}}$ for the semisimplified mod~$\varpi$ reduction of~$\rho$, there is an isomorphism $\overline{\rho}^{\textnormal{ss}} \cong \overline{\chi}_1 \oplus \overline{\chi}_2$, for continuous characters $\overline{\chi}_1, \overline{\chi}_2 : G_F \to k^\times$.
 \end{itemize}
 Let $u$ be an infinite place of~$F$. After possibly re-ordering, we can assume that $\overline{\chi}_1(c_u) = -1$ and $\overline{\chi}_2(c_u) = 1$. We set $\overline{\chi} \coloneqq \overline{\chi}_1 / \overline{\chi}_2$. 
 Our goal in Section~\ref{sec_an_ALT} will be to prove the following theorem.
\begin{thm}\label{thm_tame_SW_black_box}
    With hypotheses as above, suppose that the following further conditions are satisfied:
    \begin{enumerate}
        \item $\overline{\chi}|_{G_{F(\zeta_p)}} \neq 1$.
        \item For each $v \in S_p$, $\rhobarss|_{G_{F_v}}$ is trivial and $[F_v : \mathbb{Q}_p] > 2$. 
        Moreover, if we write $\rho|_{G_{F_v}} \sim \left( \begin{smallmatrix} \psi_{v, 1} & * \\ 0 & \psi_{v, 2} \end{smallmatrix} \right)$ for characters $\psi_{v, i} : G_{F_v} \to \overline{\mathbb{Q}}_p^\times$, as in the definition of $\lambda_v$-ordinary, then $\psi_{v, 1} \circ \Art_{F_v}$ is trivial on the subgroup of $p$-power roots of unity of $\cO_{F_v}^\times$.
        \item There are finite sets $T$, $S_{\textnormal{aux}}$ of finite places of $F$ with the following properties:
        \begin{enumerate}
            \item $\rho$ is unramified outside $S_p \cup T \cup S_{\textnormal{aux}}$. 
            \item For each $v \in T$: $q_v \equiv -1 \text{ mod }p$, $\rhobarss |_{G_{F_v}}$ is unramified,            
            $\overline{\chi}_1(\Frob_v) = q_v$, $\overline{\chi}_2(\Frob_v) = 1$, and $(\rho|_{G_{F_v}}\otimes_\cO E)^{\textnormal{ss}} \cong 1 \oplus \epsilon^{-1}$.
            \item For each $v \in S_{\textnormal{aux}}$: $q_v \equiv 1\text{ mod }p$, $\rhobarss|_{G_{F_v}}$ is trivial, and $(\rho|_{G_{F_v}} \otimes_\cO E)^{\textnormal{ss}} \cong 1 \oplus \epsilon^{-1}$.
        \end{enumerate}
        \item $[F : \mathbb{Q}] + |T \cup S_{\textnormal{aux}}|$ is even. 
        \item $|T| > h^1_{T\textnormal{-split}}(F_{S_p \cup S_{\textnormal{aux}}} / F, \overline{\chi}) + 1$.
        \item There exists a cuspidal, regular algebraic automorphic representation $\pi$ of $\GL_2(\mathbb{A}_F)$ of weight $\lambda_\pi$ and trivial central character, and an isomorphism $\iota : \overline{\mathbb{Q}}_p \to \mathbb{C}$, all satisfying the following conditions:
        \begin{enumerate}
            \item  There is an isomorphism $\overline{r}_{\pi, \iota}^{\textnormal{ss}} \cong \rhobarss$. 
            \item $\pi$ is $\iota$-ordinary. Thus for each $v \in S_p$ there is an isomorphism
            \[ r_{\pi, \iota}|_{G_{F_v}} \sim \begin{pmatrix} \psi_{v, 1} & \ast \\ 0 & \psi_{v, 2} \end{pmatrix}, \]
            as in \eqref{eqn_def_of_ordinary}.
            \item For each $v \in S_p$, the character $\psi_{v, 1} \circ \Art_{F_v}$ is trivial on the subgroup of $p$-power roots of unity of $\cO_{F_v}^\times$.             
            \item For each $v \in T \cup S_{\textnormal{aux}}$, $\pi_v$ is the Steinberg representation. In particular, there is an isomorphism $(r_{\pi, \iota}|_{G_{F_v}})^{\textnormal{ss}} \cong 1 \oplus \epsilon^{-1}$.
            \item $\pi$ is unramified outside $S_p \cup T \cup S_{\textnormal{aux}}$. 
        \end{enumerate}
    \end{enumerate}
    Then $\rho$ is modular. 
\end{thm}
We are free to extend $\cO$, so we can assume (by Ribet's lemma) that there is a representation $\rho' : G_F \to \GL_2(\cO)$ such that $\rho' \otimes_\cO \overline{\mathbb{Q}}_p \cong r_{\pi, \iota}$ and an equality (not just isomorphism)
\[ \rho' \text{ mod }\varpi = \begin{pmatrix} \overline{\chi}_1 & * \\ 0 & \overline{\chi}_2 \end{pmatrix}, \]
with the property that $(\rho' \text{ mod }\varpi)(c_u)$ is diagonal and the extension class is non-split. 
\subsection{Deformation and pseudodeformation rings}\label{subsec_deformation_and_pseudodeformation_rings}

Throughout the remainder of Section \ref{sec_an_ALT}, we take hypotheses as in Thm.\ \ref{thm_tame_SW_black_box}. If $v \in S_p$, we let $\Lambda_v \coloneqq  \cO \llbracket  \cO_{F_v}^\times(p)^{\textnormal{tf}} \rrbracket $, $\widetilde{\Lambda}_v \coloneqq  \cO \llbracket  F_v^\times(p)^{\textnormal{tf}} \rrbracket $, $\Lambda \coloneqq  \widehat{\otimes}_{v \in S_p} \Lambda_v$, $\widetilde{\Lambda} \coloneqq  \widehat{\otimes}_{v \in S_p} \widetilde{\Lambda}_v$, where the superscript ``tf'' indicates that we take the quotient by the (finite) subgroup of $p$-power torsion elements. Via local class field theory, $\widetilde{\Lambda} \in \CNL_\cO$ represents the functor of tuples $(\psi_v)_{v \in S_p}$ of lifts of the trivial character of $G_{F_v}$ ($v \in S_p$) such that $\psi_v$ is trivial on the subgroup of $p$-power torsion elements of $G_{F_v}^{\textnormal{ab}}$. 
\subsubsection{Ordinary pseudodeformation rings}\label{subsubsec:ordinary pseudodeformation rings}
Let $\operatorname{PDef}^{\textnormal{ord}}: \operatorname{CNL}_\cO \to \operatorname{Sets}$ be the functor associating to any $A$ the set of pairs $(( \psi_v )_{v \in S_p}, t)$ as follows:
\begin{itemize}
    \item For each $v \in S_p$, $\psi_v : G_{F_v} \to A^\times$ is a lift of the trivial character, trivial on the $p$-power torsion subgroup of $G_{F_v}^{\textnormal{ab}}$.
    \item $t : G_{F, S_p \cup T \cup S_{\textnormal{aux}}} \to A$ is a pseudocharacter lifting $\overline{\chi}_1 + \overline{\chi}_2$ and satisfying the following conditions (cf.\ \cite[Def.\ 5.3.1]{Pan22}). First, for any $v \in S_p$, $\sigma, \tau \in G_{F_v}$, and $\eta \in G_F$, we have
    \[ t( \sigma) = \psi_v(\sigma) + (\epsilon \psi_v)(\sigma)^{-1} \]
    \[ t( (\sigma - \psi_v(\sigma))(\tau - (\epsilon \psi_v)(\tau)^{-1}) \eta) = 0. \]
    Second, for any 
    $v \in T \cup S_{\textnormal{aux}}$, $\sigma, \tau \in G_{F_v}$, and $\eta \in G_F$, we have
    \[ t(\sigma) = 1 + \epsilon(\sigma)^{-1} \]
    and
    \[ t((\sigma-1)(\tau - \epsilon(\tau)^{-1}) \eta) = 0. \]
    \item $\det t = \epsilon^{-1}$.
\end{itemize}
The functor~$\operatorname{PDef}^{\textnormal{ord}}$ is represented by an object $P^{\textnormal{ord}}  \in \operatorname{CNL}_\cO$ that has a natural structure of $\widetilde{\Lambda}$-algebra, and is a quotient of the ring $P(G_{F, S_p \cup T \cup S_{\textnormal{aux}}}) \widehat{\otimes}_\cO \widetilde{\Lambda}$. 

\subsubsection{Taylor--Wiles data}\label{subsubsec:Taylor--Wiles data}
We call a \emph{Taylor--Wiles datum} a tuple 
\[
(Q, ( \alpha_v^{(1)}, \alpha^{(2)}_v)_{v \in Q}, (\lambda_v)_{v \in Q}) 
\]
consisting of a finite ordered set $Q = (v_1, \dots, v_q)$ of finite places of $F$, a  tuple of pairs $( \alpha_v^{(1)}, \alpha^{(2)}_v )$ of elements of $k$, and a tuple of surjective homomorphisms $\lambda_v: \mathbb{Z}_p \to k(v)^\times(p)$, all satisfying the following conditions:
\begin{itemize}
    \item For each $v \in Q$, $q_v \equiv 1 \text{ mod }p$.
    \item For each $v \in Q$, $\alpha_v^{(1)}, \alpha_v^{(2)}$ are the distinct eigenvalues of $\rhobarss(\Frob_v)$. 
\end{itemize}
We say that the datum has level $N$ if for each $v \in Q$, $q_v \equiv 1 \text{ mod }p^N$. Writing $\Delta_Q \coloneqq  \prod_{v \in Q} k(v)^\times(p)$, the ordering of $Q$ and the maps $\lambda_v$ then determine a surjection $\lambda_Q: \mathbb{Z}_p^{|Q|} \to \Delta_Q$, with kernel contained in $p^N \mathbb{Z}_p^{|Q|}$. We say that a (possibly unordered) set $Q$ of places of $F$ is a set of Taylor--Wiles places if it forms part of a Taylor--Wiles datum. 

Given a Taylor--Wiles datum, we write $\operatorname{PDef}^{\textnormal{ord}}_{Q}$ for the functor of pairs $(( \psi_v )_{v \in S_p}, t)$ satisfying the same conditions as in Section~\ref{subsubsec:ordinary pseudodeformation rings}, except that $t$ is now allowed to be ramified at places $v \in Q$. It is similarly represented by an object $P^{\textnormal{ord}}_{Q} \in \operatorname{CNL}_\cO$, quotient of the ring $P(G_{F, S_p \cup T \cup S_{\textnormal{aux}} \cup Q}) \widehat{\otimes}_\cO \widetilde{\Lambda}$.

\subsubsection{Ordinary deformation rings}\label{subsubsec:ordinary deformation rings}
We now want to define deformation rings (as opposed to pseudodeformation rings), along with their framed counterparts. We recall that we have fixed an infinite place $u$ with $\overline{\chi}_1(c_u) = -1$. Any cohomology class in $H^1(F, \overline{\chi})$ is represented by a unique cocycle $b$ such that $b(c_u) = 0$, in which case we say that $b$ is normalised. All cocycles appearing in \S \ref{sec_an_ALT} will be assumed to be normalised. 

We write $H^1_{T \textnormal{-split}}(F_{S_p \cup S_{\textnormal{aux}}} / F, \overline{\chi}) \leq H^1(F_{S_p \cup S_{\textnormal{aux}}}/F, \overline{\chi})$ for the subgroup of elements locally trivial at places of $T$. If $[b] \in H^1_{T \textnormal{-split}}(F_{S_p \cup S_{\textnormal{aux}}} / F, \overline{\chi})$ is non-zero, then we write 
\[ \overline{\rho}_b \coloneqq  \begin{pmatrix} \overline{\chi}_1 & \overline{\chi}_2 b \\ 0 & \overline{\chi}_2 \end{pmatrix} : G_{F, S_p \cup S_{\textnormal{aux}}} \to \GL_2(k) \]
for the corresponding non-split residual representation. This is well-defined, because we assume $b$ to be normalised.

Given such a normalized cocycle~$b$, we write $\operatorname{Def}^{\textnormal{ord}} : \CNL_\cO \to \operatorname{Sets}$ for the functor of pairs $(( \psi_v )_{v \in S_p}, [r])$, where $\psi_v$ ($v \in S_p$) is a lift of the trivial character of $G_{F_v}$, trivial on the $p$-power torsion subgroup of $G_{F_v}^{\textnormal{ab}}$, and $[r]$ is a deformation (i.e. $\widehat{\GL}_2$-conjugacy class of liftings $r : G_{F, S_p \cup T \cup S_{\textnormal{aux}}} \to \GL_2$) of $\overline{\rho}_b$ satisfying the following conditions:
\begin{itemize}
    \item $\det r = \epsilon^{-1}$.
    \item For each $v \in S_p$ and $\sigma, \tau \in G_{F_v}$, we have $\tr r(\sigma) = \psi_v(\sigma) + (\epsilon \psi_v)(\sigma)^{-1}$ and 
    \[ (r(\sigma) - \psi_v(\sigma))(r(\tau) - (\epsilon \psi_v)(\tau)^{-1}) = 0. \]
    \item For each $v \in T \cup S_{\textnormal{aux}}$ and $\sigma, \tau \in G_{F_v}$, we have $\tr r(\sigma) = 1 + \epsilon(\sigma)^{-1}$ and
    \[ (r(\sigma) - 1)(r(\tau) - \epsilon(\tau)^{-1}) = 0. \] 
\end{itemize}
This functor is represented by a quotient $R^{\textnormal{ord}}_{b}$ of $R_b(G_{F, S_p \cup T \cup S_{\textnormal{aux}}}) \widehat{\otimes}_\cO \widetilde{\Lambda}$, where the first factor is the universal deformation ring defined in Section~\ref{subsec:pseudodeformations and deformations}. If $Q$ is a set of Taylor--Wiles places, we write $\operatorname{Def}^{\textnormal{ord}}_{Q}$ for the functor of pairs $(( \psi_v )_{v \in S_p}, [r])$ satisfying the same conditions as above, except that $r$ is now allowed to be ramified at the places in $Q$. This functor is likewise represented by an object $R^{\textnormal{ord}}_{b, Q} \in \CNL_\cO$, which is a quotient of $R_b(G_{F, S_p \cup T \cup S_{\textnormal{aux}} \cup Q}) \widehat{\otimes}_\cO \widetilde{\Lambda}$.

\subsubsection{Local lifting rings}\label{subsubsec:local lifting rings}
If $v$ is a finite place of $F$, and $b \in H^1_{T\text{-split}}(F_{S_p \cup S_{\aux}} / F, \overline \chi)$, we write $R_{b, v}^\square$ for the ring representing the functor of all liftings of $\overline{\rho}_b|_{G_{F_v}}$ of determinant $\epsilon^{-1}$. We introduce the following variants:
\begin{itemize}
    \item If $v \in S_p$, the quotient $R_{b, v}^{\textnormal{ord}}$ of $R_{b,v}^{\square} \widehat{\otimes}_\cO \widetilde{\Lambda}_v$ classifying pairs $(\psi_v, r_v)$, where  $\psi_v$ is a lift of the trivial character of $G_{F_v}$, trivial on the subgroup of $p$-power torsion elements in $G_{F_v}^{\textnormal{ab}}$, and $r_v : G_{F_v} \to \GL_2$ is a lift of $\overline{\rho}_{b}|_{G_{F_v}}$ such that for all $\sigma, \tau \in G_{F_v}$, we have $\tr r_v(\sigma) = \psi_{v}(\sigma) + (\epsilon \psi_{v})(\sigma)^{-1}$ 
    and
    \[ (r_v(\sigma) - \psi_{v}(\sigma))(r_v(\tau) - (\epsilon \psi_{v})(\tau)^{-1}) = 0. \]
    \item If $v \in T$, the quotient $R_{b, v}^{\textnormal{St}, 1}$ of $R_{b, v}^\square$ representing the functor of all lifts that are Steinberg (not just up to quadratic twist); more precisely, the quotient defined in \cite[\S 5.4.4]{Tho16} with $\alpha_v = 1$, $\beta_v = -1$. %
    \item If $v \in S_{\textnormal{aux}}$, the quotient $R_{b, v}^{\textnormal{St}}$ representing the functor of all Steinberg lifts (i.e.\ in the notation of \cite{Sho16}, the quotient of the ring $R^{\square}(\overline{\rho}_b|_{G_{F_v}}, \tau_{1, \textnormal{ns}})$ by the relation $\det = \epsilon^{-1}$).
    \item If $v \not\in S_p \cup T \cup S_{\textnormal{aux}}$, then $R_{b, v}^{\textnormal{ur}}$ is the quotient of $R_{b, v}^\square$ classifying unramified liftings.
\end{itemize}
\begin{lemma}\label{lem_local_and_global_lifting_rings}
    Let $A \in \CNL_{\cO}$ be given a $\widetilde{\Lambda}$-algebra structure, and let $r : G_{F, S_p \cup T \cup S_{\textnormal{aux}} \cup Q} \to \GL_2(A)$ be a continuous lifting of $\overline{\rho}_b$ of determinant $\epsilon^{-1}$. Then the following are equivalent:
    \begin{enumerate}
        \item The induced map $R_b(G_{F, S_p \cup T \cup S_{\textnormal{aux}} \cup Q}) \widehat{\otimes}_\cO \widetilde{\Lambda} \to A$ factors through $R_{b, Q}^{\textnormal{ord}}$.
        \item For each $v \in S_p$ (resp.\ $ v \in T $, resp.\ $v \in S_{\textnormal{aux}}$), the map $R_{b, v}^\square \widehat{\otimes}_\cO \widetilde{\Lambda}_v \to A$ induced by $r|_{G_{F_v}}$ factors through $R_{b, v}^{\textnormal{ord}}$ (resp.\ $R_{b, v}^{\textnormal{St}, 1}$, resp.\ $R_{b, v}^{\textnormal{St}}$).
    \end{enumerate}
\end{lemma}
\begin{proof}
    It is clear that the conditions in (2) imply (1). We need to show that the converse holds. We can do this place-by-place. If $v \in S_p$, it is true by definition. We next consider the case of $v \in T$. The restriction of the universal lifting over $R_{b, Q}^{\textnormal{ord}}$ to $G_{F_v}$ is conjugate to a representation $r$ with $r(\Frob_v) = \diag(q_v, 1)$ and $r(t_v) = 1 + \left( \begin{smallmatrix} x & y \\ z & -x \end{smallmatrix}\right)$, where $t_v \in I_{F_v}$ is a generator of tame inertia, and $x^2 + yz = 0$. The relation $\tr r(\Frob_v) = \tr r(\Frob_v t_v)$ implies that $(q_v-1)x = 0$, hence $x = 0$, and $yz = 0$. Then the relation $(r(\Frob_v)-1)(r(t_v) - 1) = 0$ implies $y  = 0$, showing that $r$ is of type $R_{b, v}^{\textnormal{St}, 1}$. 
    
    Finally we consider the case $v \in S_{\textnormal{aux}}$, giving the details just in the most challenging case that $\overline{\rho}_b|_{G_{F_v}}$ is trivial. Inspecting the definition of $R^{\textnormal{ord}}_{b, Q}$, we see that we again need to address the purely local question of whether the quotient $R_{b, v}^{\square} \to R$ defined by the equations $\tr r(\sigma) = 1 + \epsilon(\sigma)^{-1}$ and $(r(\sigma) - 1)(r(\tau) - \epsilon(\tau)^{-1}) = 0$ is equal to $R_{b, v}^{\textnormal{St}}$. We have already used in the first sentence of the proof that there is a map $R \to R_{b, v}^{\textnormal{St}}$, and we need to check that this is an isomorphism. 
    
    We claim that in this case the relations $\tr r(\sigma) = 1 + \epsilon(\sigma)^{-1}$ already define $R_{b, v}^{\textnormal{St}}$. Inspecting the proof of \cite[Prop.\ 5.8]{Sho16}, we see that the ring $R$ defined by these relations for all $\sigma \in G_{F_v}$ agrees with quotient of the ring $R'$ appearing there by the relation fixing the determinant to be $\epsilon^{-1}$. It is shown in \emph{loc.\ cit.} that $R'$ is a Cohen--Macaulay $\cO$-flat domain, and that $R' / (\varpi)$ is a domain. Since there is an isomorphism $R' \cong R\llbracket x\rrbracket $, we deduce that $R$ is likewise a Cohen--Macaulay $\cO$-flat domain, 
    and that $R / (\varpi)$ is a domain.
    We moreover see that $R / (\varpi)$ is normal, since one can easily check from the defining equations that $R / (\varpi)$ is regular in codimension 1. Since $R$ and $R_{b, v}^{\textnormal{St}}$ have the same $\overline{E}$-points, it follows that $R \cong R_{b, v}^{\textnormal{St}}$, as required.    
\end{proof}

\begin{rem}\label{rem:structure of Steinberg lifting rings}
We record explicitly that if $v \in S_{\textnormal{aux}}$, then $R_{b, v}^{\textnormal{St}}$ is a Cohen--Macaulay normal $\cO$-flat domain, 
and $R^{\St}_{b, v}/\varpi$ is a Cohen--Macaulay normal domain. 
If $v \in T$, then \cite[Prop.\ 5.16]{Tho16} states that $R_{b, v}^{\textnormal{St}, 1}$ is even formally smooth over $\cO$.
\end{rem}
If $T_0$ is a finite set of places of $F$, then we write $\operatorname{Def}^{\textnormal{ord}, \square_{T_0}}_{Q}$ for the functor of $\widehat{\GL}_2$-conjugacy classes of pairs $(r, (X_v)_{v \in T_0})$, where $r$ is a lifting of type $\operatorname{Def}^{\textnormal{ord}}_{b,Q}$, and $X_v \in \widehat{\GL}_2$. (The action of an element $g \in \widehat{\GL}_2$ on such a pair is given by the formula $g \cdot (r, (X_v)) = (g r g^{-1}, g X_v)$.) 
This functor is represented by an object $R^{\textnormal{ord}, \square_{T_0}}_{b, Q} \in \CNL_\cO$. 
If $v \in T_0$, then $R^{\textnormal{ord}, \square_{T_0}}_{ Q}$ has a natural structure of $R_{b, v}^\square$-algebra, determined by the natural transformation that sends the orbit of $(r, (X_v)_{v \in T_0})$ to $X_v^{-1} r X_v$. 
An important special case will be when $T_0 = S_p \cup S_{\textnormal{aux}} \cup T_1$, where $T_1$ is a finite set of finite places disjoint from $S_p \cup S_{\textnormal{aux}} \cup T \cup Q$. In this case we define %
\[ R_{\textnormal{loc}}^{\square_{T_0}} \coloneqq  \left( \widehat{\bigotimes}_{v \in S_p} R_{b, v}^{\textnormal{ord}} \right) \widehat{\otimes}_\cO \left( \widehat{\bigotimes}_{v \in S_{\textnormal{aux}}} R_{b, v}^{\textnormal{St}} \right) \widehat{\otimes}_\cO \left( \widehat{\bigotimes}_{v \in T_1} R_{b, v}^{\textnormal{ur}} \right). \]
\begin{rem}\label{rem:tangent spaces and Selmer groups}
The rings $R_b^{\textnormal{ord}}$ and $R^{\textnormal{ord}}_{b, Q}$ may be identified with the rings $R_\cS$ and $R_{\cS_Q}$ associated to certain deformation problems $\mathcal{S}$ and $\mathcal{S}_Q$ as in \cite[Def.\ 5.6]{Tho16}. In particular, their tangent spaces, and the tangent spaces of their framed variants are controlled by the Selmer groups $H^1_{\mathcal{S}, T_0}(\ad^0 \overline{\rho}_b)$ and $H^1_{\mathcal{S}_Q, T_0}(\ad^0 \overline{\rho}_b)$ defined in \emph{op.\ cit.}---see e.g.\ \cite[Prop.\ 5.10]{Tho16}, and note that this standard result does not use the standing hypothesis of \cite[\S 5.3]{Tho16} that $\overline{\rho}_b$ is absolutely irreducible: indecomposability is enough. 
\end{rem}
\begin{pp}\label{prop_normality_of_ordinary_lifting_ring}
    If $v \in S_p$, then the reduced quotient $(R_{b, v}^{\textnormal{ord}})^{\textnormal{red}}$ is a normal, $\cO$-flat domain of Krull dimension $4 + 2 [ F_v : \mathbb{Q}_p]$. Moreover, $(R_{b, v}^{\textnormal{ord}})^{\textnormal{red}} / (\varpi)$ is also normal. 
\end{pp}
\begin{proof}
    In the case that $F_v$ contains no non-trivial $p$th roots of unity, this is \cite[Lem.\ 6.2.2]{Pan22}, so we can assume that $F_v$ contains a non-trivial $p$th root of unity. Let $d_v = [ F_v : \mathbb{Q}_p ]$. Then $d_v$ is even and, by \cite[Thm.\ 7.2.8]{Neu08}, $G_{F_v}(p)$ admits a presentation as the quotient of a free pro-$p$ group on generators $x_1, \dots, x_{d_v+2}$ by the relation 
    \[ x_1^{p^s} [x_1, x_2] [ x_3, x_4 ] \dots [x_{d_v+1}, x_{d_v+2}] = 1 \]
    (where $p^s$ is the number of $p$-power roots of unity in $F_v$). We will prove the result in the case that $\overline{\rho}_b|_{G_{F_v}}$ is trivial, the case where it is a non-split extension of the trivial character by itself being simpler. 

    Following \cite{Ger19}, let $A \coloneqq R_{b, v}^{\square} \widehat{\otimes}_\cO \widetilde{\Lambda}_v$, let $r : G_{F_v} \to \GL_2(A)$ denote the universal lifting, and let $\mathcal{G}_v \leq \mathbb{P}^1_A$ be the closed subscheme which for a morphism $A \to B$ classifies direct summands $L_B \leq B^2$ such that for all $\sigma \in G_{F_v}$, $(r(\sigma) - \psi_v(\sigma)) L_B = 0$. (We remind the reader that here $\psi_v$ is the pushforward of the universal character of $G_{F_v}$ from $\widetilde{\Lambda}$.) Equivalently, $\mathcal{G}_v$ is the subscheme defined by the vanishing of the sections $r(x_i) - \psi_v(x_i) \in \Hom_{\cO_{\mathbb{P}^1_A}}(\cO(-1), \cO^2)$ ($i = 1, \dots, d_v+2$). It follows that there is a short exact sequence
    \[ 0 \to \mathcal{K} \to \cO(-1)^{2(d_v+2)} \to \mathcal{I} \to 0, \]
    where $\mathcal{I} \leq \cO_{\mathbb{P}^1_A}$ is the ideal sheaf of $\mathcal{G}_v$. Since $H^\ast(\mathbb{P}^1_A, \cO(-1)) = 0$, we see that $H^1(\mathbb{P}^1_A, \mathcal{I}) = 0$, and therefore that the sequence
    \[ 0 \to H^0(\mathbb{P}^1_A, \mathcal{I}) \to A \to H^0(\mathcal{G}_v, \cO_{\mathcal{G}_v}) \to 0 \]
    is also exact. However, the proof of \cite[Lem.\ 3.11]{Tho15} shows that $\mathcal{G}_v$ is normal, integral, and $\cO$-flat (using here our assumption that $[F_v : \mathbb{Q}_p] > 2$), and therefore that $B \coloneqq H^0(\mathcal{G}_v, \cO_{\mathcal{G}_v})$ is also normal, integral, and $\cO$-flat. (Note that the proof situation here is not exactly identical since we fix the determinant and trivialise the action of the $p$-power roots of unity on the ordinary line, but the changes required are minimal.) %
    In fact, we see that $B / (\varpi)$ is also normal: $H^1(\mathcal{G}_v, \cO_{\mathcal{G}_v}) = 0$, hence $B / (\varpi) \cong H^0(\mathcal{G}_v, \cO_{\mathcal{G}_v / (\varpi)})$. The scheme $\mathcal{G}_v / (\varpi)$ is still integral and normal, so the result follows.

    To complete the proof of the proposition, we therefore need to show that $(R_{b, v}^{\textnormal{ord}})^{\textnormal{red}} \cong B$. It is clear that $B$ is a quotient of $R_{b, v}^{\textnormal{ord}}$, so it suffices to check that these two quotients of $A$ have the same prime ideals. Let $\mathfrak{p}$ be a prime ideal of $R_{b, v}^{\textnormal{ord}}$. Taking into account the truth of a slight variation on \cite[Cor.\ 3.12]{Tho15}, it suffices to show that $r \otimes_A k(\mathfrak{p})$ admits an invariant line on which $G_{F_v}$ acts via $\psi_v \text{ mod }\mathfrak{p}$. If the characters $\psi_v \text{ mod }\mathfrak{p}$, $(\epsilon \psi_v)^{-1} \text{ mod }\mathfrak{p}$ are distinct, then this follows immediately from \cite[Lem.\ 6.2.11]{All23}. Suppose instead that these characters are equal. Then either $r \text{ mod }\mathfrak{p}$ equals $(\psi_v \text{ mod }\mathfrak{p}) \cdot 1_2$, in which case we're done, or there exists $\tau \in G_{F_v}$ such that $r(\tau) - \psi_v(\tau) \not\equiv 0 \text{ mod }\mathfrak{p}$. In the latter case, $r(\tau) - \psi_v(\tau)\text{ mod }\mathfrak{p}$ has rank 1, and the relation $(r(\sigma) - \psi_v(\sigma))(r(\tau) - \psi_v(\tau)) \equiv 0 \text{ mod }\mathfrak{p}$ shows that $(r(\tau) - \psi_v(\tau)) k(\mathfrak{p})^2$ is the desired line.  
\end{proof}
\begin{cor}\label{cor_normality_of_local_lifting_ring}
    $R_{\textnormal{loc}}^{\square_{S_p \cup S_{\textnormal{aux}}}}$ has Krull dimension $1 + 3 |S_p \cup S_{\textnormal{aux}} | + 2 [ F : \mathbb{Q} ]$, and its underlying reduced ring is normal and $\cO$-flat. 
\end{cor}
\begin{proof}
    If $A, B \in \CNL_\cO$ are $\cO$-flat and geometrically integral, and both $A / (\varpi)$ and $B / (\varpi)$ are normal, then $A \widehat{\otimes}_\cO B$ is normal. Indeed, \cite[Lem.\ 1.4]{Tho15} shows that $A \widehat{\otimes}_\cO B$ is $\cO$-flat and geometrically integral. If $C \in \CNL_\cO$ is $\cO$-flat, then $C$ is normal if $C / (\varpi)$ is normal. However, \cite[Thm.\ 5.4]{Tab25} shows that $A \widehat{\otimes}_\cO B / (\varpi) \cong (A / (\varpi)) \widehat{\otimes}_k (B / (\varpi))$ is normal, so the result follows. 
    
    The result therefore follows from Prop.\ \ref{prop_normality_of_ordinary_lifting_ring} and
    Rem.\ \ref{rem:structure of Steinberg lifting rings}.
\end{proof}
The universal property of $P^{\textnormal{ord}}$ (resp.\ $P^{\textnormal{ord}}_{Q}$) implies that for any $b$, the ring $R^{\textnormal{ord}}_{b}$ (resp.\ $R^{\textnormal{ord}}_{b, Q}$) has a natural structure of $P^{\textnormal{ord}}$-algebra (resp.\ $P^{\textnormal{ord}}_{Q}$-algebra). 
From now on, we will write 
$\mathcal{R} \leq P^{\textnormal{ord}}$ (resp.\ $\mathcal{R}_Q \leq P^{\textnormal{ord}}_{Q}$) for the reducibility ideal,
as defined in Section~\ref{subsec:pseudo matrix entries}.

\begin{lemma}\label{lem_bounding_reducible_locus}
    Let $h \coloneqq  h^1_{T  \textnormal{-split}}(F_{S_p \cup S_{\textnormal{aux}}}/F, \overline{\chi})$. Then for any Taylor--Wiles set $Q$, there exists a surjective homomorphism of $P^{\textnormal{ord}}_{ Q}$-algebras
    \[ P^{\textnormal{ord}}_{ Q}\llbracket  X_1, \dots, X_{h} \rrbracket  \to R^{\textnormal{ord}}_{b,  Q}. \]
\end{lemma}
\begin{proof}
    It suffices to give an upper bound for the embedding dimension of $R^{\textnormal{ord}}_{b, Q} / ( \mathfrak{m}_{P^{\textnormal{ord}}_{Q}})$. 
    Let $\rho_b^{\textnormal{univ}}$ be a representative of the universal deformation over $R^{\textnormal{ord}}_{b,  Q} / (\mathfrak{m}_{P^{\textnormal{ord}}_{Q}})$ with $\rho_b^{\textnormal{univ}}(c_u) = \diag(-1, 1)$, as usual. 
    Then $\rho_b^{\textnormal{univ}}(\sigma)_{21} = 0$ for all $\sigma \in G_F$, because $\rho_b^{\textnormal{univ}}(\tau)_{12} \rho_b^{\textnormal{univ}}(\sigma)_{21} = 0$ for all $\sigma, \tau \in G_F$, and there exists $\tau$ such that $\rho_b^{\textnormal{univ}}(\tau)_{12}$ is a unit). Moreover we have $(\rho_b^{\textnormal{univ}})_{11}  = \overline{\chi}_1$ and $(\rho_b^{\textnormal{univ}})_{22}  = \overline{\chi}_2$. Thus $\overline{\chi}_2^{-1} \rho_b^{\textnormal{univ}}(\sigma)_{12}$ is a 1-cocycle in
    \[ Z^1(F_{S_p \cup T \cup S_{\textnormal{aux}} \cup Q} / F, \overline{\chi}) \otimes_k (R^{\textnormal{ord}}_{b, Q} / ( \mathfrak{m}_{P^{\textnormal{ord}}_{Q}})) = Z^1(F_{S_p\cup T \cup S_{\textnormal{aux}} } / F, \overline{\chi}\otimes_k (R^{\textnormal{ord}}_{b, Q} / ( \mathfrak{m}_{P^{\textnormal{ord}}_{Q}}))). \]
    In fact, the cohomology class represented by this 1-cocycle is trivial at each place $v \in T$. Indeed, the definition of $R_{b,v }^{\textnormal{St}, 1}$ implies that $\rho_b^{\textnormal{univ}}|_{G_{F_v}}$ admits an invariant line lifting the one given by the equation $\overline{\chi}_2(\Frob_v) = 1$. Since it also visibly admits an invariant line lifting the one given by the equation $\overline{\chi}_1(\Frob_v) = -1$ (because $\rho_b^{\textnormal{univ}}(\sigma)_{21} = 0$ for all~$\sigma \in G_F$), it must be split at $v$. 

    Let $[b_1], \dots, [b_h]$ be a $k$-basis for $H^1_{T\textnormal{-split}}(F_{S_p\cup S_{\textnormal{aux}} } / F, \overline{\chi})$, and let $\sigma_1, \dots, \sigma_h \in G_{F(\overline{\chi}_1, \overline{\chi}_2)}$ be such that the $h$ ``evaluation at $\sigma_i$'' functionals on $H^1_{T\textnormal{-split}}(F_{S_p\cup S_{\textnormal{aux}} } / F, \overline{\chi})$ are linearly independent. Choose $a_{ij} \in \cO$ such that $\sum_{i=1}^h a_{ij} b_l(\sigma_i) = \delta_{jl}$ for each $j, l = 1, \dots, h$, and let $z_j \coloneqq \sum_{i=1}^h a_{ij} \sigma_i \in \cO[G_{F(\overline{\chi}_1, \overline{\chi}_2)}]$. We claim that the ring 
    \[
    A \coloneqq  R^{\textnormal{ord}}_{b, Q} / ( \mathfrak{m}_{P^{\textnormal{ord}}_{Q}}, \rho_b^{\textnormal{univ}}(z_1)_{12}, \dots, \rho_b^{\textnormal{univ}}(z_h)_{12})
    \]
    coincides with~$k$.
    This immediately implies the desired upper bound on the embedding dimension.

    To see this, write $\overline{\chi}_2^{-1} \rho_b^{\textnormal{univ}}(\sigma)_{12} = b + \sum_{j=1}^h x_j b_j$ with $x_j \in \mathfrak{m}_{R^{\textnormal{ord}}_{b, Q}} / ( \mathfrak{m}_{P^{\textnormal{ord}}_{Q}})$. Evaluating at $z_j$ for each $j = 1, \dots, h$, we see that the equality $\overline{\chi}_2^{-1} \rho_b^{\textnormal{univ}}(\sigma)_{12} = b$ holds after projection to $A$, and that the specialisation of $\rho_b^{\textnormal{univ}}$ to $A$ is equal to $\overline{\rho}_b$. By the universal property of $R^{\textnormal{ord}}_{b, Q}$, this is possible only if $A = k$, as required. 
\end{proof}

We say that we are in the dihedral case if the character $\overline{\chi}$ is quadratic. In this case, it follows from the hypothesis of Thm.\ \ref{thm_tame_SW_black_box} that the extension $F(\overline{\chi}) / F$ is a quadratic CM extension in which all places above $p$ split. 
\begin{pp}\label{prop_existence_of_TW_places_first_patching}
    We can find a constant $c \geq 0$ with the following property: for any $N \geq 1$ we can find a Taylor--Wiles datum $Q_N$ of level $N$ satisfying the following conditions:
    \begin{enumerate}
        \item $|Q_N| = |T|$.
        \item Let $M_{Q_N}$ be the maximal abelian pro-$p$ extension of~$F$ unramified outside $S_p \cup S_{\textnormal{aux}} \cup Q_N$, in which every element of $T$ splits. Then there is a homomorphism $\mathbb{Z}_p^{1 + \delta_F} \to \Gal(M_{Q_N} / F)$ with cokernel of size bounded above by $p^c$. 
        \item If we are in the dihedral case, then let $L_{Q_N}$ be the maximal abelian pro-$p$ extension of~$F(\overline{\chi})$ unramified outside $S_p \cup S_{\textnormal{aux}} \cup T \cup Q_N$. Then the natural map $\Gal(L_{Q_N} / F(\overline{\chi})) \to \Gal(L_\varnothing / F(\overline{\chi})$ is an isomorphism, and there is a homomorphism $\mathbb{Z}_p^{[F : \mathbb{Q}]} \to \Gal(L_{Q_N} / F(\overline{\chi})) / (c_u+1)$ with finite cokernel. 
    \end{enumerate}
\end{pp}
Here $\delta_F$ denotes the Leopoldt defect of the totally real field $F$ (conjecturally 0, and bounded above by $([F : \mathbb{Q}]-1) / 2$, by the main result of \cite{Wal84}).
\begin{proof}
    This is an exercise in class field theory. We first establish (2). Enumerate $T = \{ v_1, \dots, v_s \}$, and let $\mathfrak{p}_i \leq \cO_F$ be the prime ideal corresponding to $v_i$, $\omega_i$ its order in $\operatorname{Cl}(\cO_F)$, $\Pi_i \in \cO_F$ a generator for $\mathfrak{p}_i^{\omega_i}$, and $\nu_i \coloneqq v_p(\omega_i)$.
    
    Let $\nu \coloneqq \max_i \nu_i$. We are free to increase $N$, so can assume without loss of generality that $N > \nu$. We claim that we can find a Taylor--Wiles set $Q_N = \{ u_1, \dots, u_s \}$ of level $N$ with the following properties:
    \begin{itemize}
        \item $|Q_N| = |T|$.
        \item The image of $\Pi_i$ in $k(u_j)^\times(p)$ is divisible by $p^{\nu+1}$ if $i \neq j$, and not divisible by $p^{\nu+1}$ if $i = j$. 
    \end{itemize}
    It then follows that the quotient of $\prod_{v \in Q_N} k(v)^\times(p)$ by the subgroup generated by $\Pi_1, \dots, \Pi_s$ is annihilated by $p^\nu$.  To find such $Q_N$, it is enough, by the Chebotarev density theorem, to show that for each $i = 1, \dots, s$, the image of the map
    \[ G_{F(\zeta_{p^N})} \to \overline{\chi}(G_{F(\zeta_{p^N})}) \times \prod_{j=1}^s \Gal(F(\zeta_{p^N}, \Pi_j^{1/p^{\nu+1}})/F(\zeta_{p^N})) \]
    contains an element $(x, (y_j)_{j=1}^s)$ with $x \neq 1$, $y_j = 1$ if $j \neq i$, and $y_i \neq 1$. Since we assume $\overline{\chi}(G_{F(\zeta_p)})$ is non-trivial, and this group has order prime to $p$, we can find elements in the image with $x \neq 1$ and $y_j = 1$ for all $j$. On the other hand, the extension $F(\zeta_{p^N}, \Pi_j^{1/p^{\nu+1}})/F(\zeta_{p^N})$ is ramified above $v_j$ but not above $v_k$ if $k \neq j$ (because $\nu+1 > \nu_j$). Elements of the inertia group at a place above $v_i$ therefore have $x = 1$ and $y_j = 1$ if $j \neq i$, but $y_i \neq 1$. 
    This concludes the proof of the claim.

    By class field theory, there is a short exact sequence
    \begin{multline*}  0 \to  \frac{\prod_{v \in S_p} \cO_{F_v}^\times(p) \times \prod_{v \in S_{\textnormal{aux}} \cup Q_N} k(v)^\times(p)}{\overline{\cO_{F, T}^\times}} \to \Gal(M_{Q_N} / F) \\
    \to \operatorname{Cl}(\cO_F)(p)/\langle \mathfrak{p}_1, \dots, \mathfrak{p}_s \rangle \to 0.  \end{multline*}
    Let $p^{c_1}$ denote the order of the torsion subgroup of the quotient $\prod_{v \in S_p} \cO_{F_v}^\times(p) / \overline{\cO_F^\times} \times \prod_{v \in S_{\textnormal{aux}} }k(v)^\times(p)$, let $c_2 = s \nu$, and let $p^{c_3}$ denote the order of  $\operatorname{Cl}(\cO_F)(p)/\langle \mathfrak{p}_1, \dots, \mathfrak{p}_s \rangle$. We see that property (2) holds on taking $c = c_1 + c_2 + c_3$.

    Now suppose further that we are in the dihedral case. Since each place of $Q_N$ is inert in $F(\overline{\chi})$, we find that the natural map 
    \[ \Gal(L_{Q_N} / F(\overline{\chi})) / (c_u+1) \to \Gal(L_\varnothing / F(\overline{\chi})) / (c_u + 1) \]
    is an isomorphism: its kernel is a quotient of a product of the groups $k({\widetilde{v}})^\times(p) / (c_u + 1)$, for places $\widetilde{v}$ of $F(\overline{\chi})$ lying above $v \in Q_N$, but these are trivial because $(q_{v} - 1, q_{v} + 1) = 2$ and $p$ is odd. The result (after possibly increasing $c$ to account for the $p$-torsion in  $\Gal(L_\varnothing / F(\overline{\chi})) / (c_u + 1)$) follows immediately. 
\end{proof}
\begin{cor}
    We can find a constant $C \geq 0$ with the following property: 
     for any $N \geq 1$ we can find a Taylor--Wiles datum $Q_N$ of level $N$ satisfying the following conditions:
    \begin{enumerate}
        \item $|Q_N| = |T|$.
        \item There are elements $x_1, \dots, x_{1 + \delta_F} \in P^{\textnormal{ord}}_{Q_N}$ such that $P^{\textnormal{ord}}_{Q_N} / (\mathcal{R}_{Q_N}, \varpi, x_1, \dots, x_{1 + \delta_F})$ is an Artinian local $k$-algebra of dimension at most $C$.
        \item If we are in the dihedral case, then the map $P_{Q_N}^{\textnormal{ord}, \overline{\chi}\textnormal{-dih}} \to P^{\textnormal{ord}, \overline{\chi}\textnormal{-dih}}$ is an isomorphism, and $P^{\textnormal{ord}, \overline{\chi}\textnormal{-dih}} / (\varpi)$ has Krull dimension at most $[F : \mathbb{Q}]$. 
    \end{enumerate}
\end{cor}
\begin{proof}
    Choose $Q_N$ as in Prop.\ \ref{prop_existence_of_TW_places_first_patching}; we will show that these additional properties also hold. 
    We recall that, by definition, $P_{Q_N}^{\textnormal{ord}} / (\mathcal{R}_{Q_N}, \varpi)$ is a quotient of $P(G_{F, S_p \cup T \cup S_{\textnormal{aux}} \cup Q_N}) \widehat{\otimes}_\cO \widetilde{\Lambda}$.
    The universal pseudocharacter $t^{\textnormal{univ}}$ over $P^{\textnormal{ord}}_{Q_N} / (\mathcal{R}_{Q_N}, \varpi)$ admits a unique decomposition $\epsilon^{-1} \chi_2^{-1} \oplus \chi_2$, where $\chi_2$ is a character lifting $\overline{\chi}_2$. Moreover, $\chi_2 [\overline{\chi}_2]^{-1}$ (where $[\cdot]$ denotes Teichm\"uller lift) factors through $\Gal(M_{Q_N} / F)$. Indeed, we just need to check that if $v \in T$, then $\chi_2|_{G_{F_v}}$ is unramified and $\chi_2(\Frob_v) = 1$. If $v\in T$ then $t^{\textnormal{univ}}|_{G_{F_v}} = 1 + \epsilon^{-1} = \epsilon^{-1} \chi_2^{-1}|_{G_{F_v}} + \chi_2|_{G_{F_v}} $, where the factors are distinct, and the result follows by the uniqueness of this decomposition in the residually multiplicity-free case. 
    
    Let $X_2$ denote the universal deformation of $\chi_2$ over $k\llbracket M_{Q_N} \rrbracket$. It follows from the discussion in the previous paragraph that we may view $P_{Q_N}^{\textnormal{ord}} / (\mathcal{R}_{Q_N}, \varpi)$ as a quotient of the ring $S_{Q_N}$, defined to be the quotient of $k\llbracket M_{Q_N} \rrbracket \widehat{\otimes}_\cO \widetilde{\Lambda}$ by the relations $\psi_v + (\epsilon \psi_v)^{-1} = (X_2 \oplus \epsilon^{-1} X_2^{-1})|_{G_{F_v}}$ ($v \in S_p$). According to Prop.\ \ref{prop_existence_of_TW_places_first_patching}, we can find an isomorphism $M_{Q_N} \cong \mathbb{Z}_p^{1 + \delta_F} \oplus T_{Q_N}$, where $T_{Q_N}$ is a finite abelian group of order bounded independently of $N$. In particular, we can find $x_1, \dots, x_{1 + \delta_F} \in k \llbracket M_{Q_N} \rrbracket$ such that $k \llbracket M_{Q_N} \rrbracket / (x_1, \dots, x_{1 + \delta_F}) \cong k [ T_{Q_N} ]$. On the other hand, $S_{Q_N} \otimes_{k \llbracket M_{Q_N} \rrbracket} k$ is a quotient of $\widetilde{\Lambda} / (\varpi)$ by the relations $\psi_v + \psi_v^{-1} = 1 + 1$, $v \in S_p$, and this is a finite Artinian local $k$-algebra of dimension bounded independently of $N$. Using the completed version of Nakayama's lemma, we conclude that $S_{Q_N}$ is a finite $k\llbracket M_{Q_N} \rrbracket$-algebra, with minimal number of generators bounded independently of $N$; taking in account what we know about the ring $k\llbracket M_{Q_N} \rrbracket$, this implies the statement given in the corollary.

    Property (3) follows immediately from Prop.\ \ref{prop_universal_dihedral_pseudocharacter} and the third part of Prop.\ \ref{prop_existence_of_TW_places_first_patching}. 
\end{proof}

\subsection{Automorphic forms and Hecke algebras}\label{subsec:automorphic forms and Hecke algebras}

We continue with the set-up and assumptions of Thm.\ \ref{thm_tame_SW_black_box}. Let $D$ be a quaternion algebra over $F$ ramified precisely at $T \cup S_{\textnormal{aux}}$ and the infinite places. (This exists because of our assumption that $[F : \mathbb{Q}] + |T \cup S_{\textnormal{aux}}|$ is even.) Fix a choice of maximal order $\cO_D \leq D$ and a choice of isomorphism $\cO_D \otimes_{\cO_F} {\widehat{\cO}_F^{T \cup S_{\textnormal{aux}}}} \cong M_2(\widehat{\cO}_F^{T \cup S_{\textnormal{aux}}})$. We use this isomorphism to identify $(D \otimes_F F_v)^\times$ with $\GL_2(F_v)$ for every finite place $v\not\in T \cup S_{\textnormal{aux}}$ of $F$. 

If $U = \prod_v U_v \leq (D \otimes_F \mathbb{A}_F^\infty)^\times$ is an open compact subgroup, then we write $H(U)$ for the set of functions $f : D^\times \mathbb{A}_F^{\infty, \times} \backslash (D \otimes_F \mathbb{A}_F^\infty)^\times / U \to \cO$. 

Let $U_0 \coloneqq  \prod_v U_{0, v}$, where $U_{0, v} \coloneqq  (\cO_D \otimes_{\cO_F} \cO_{F_v})^\times$. If $v$ is a finite place of $F$, and $b \geq c \geq 0$, $b \geq 1$, then we define $\Iw_{v}(b) \leq \GL_2(\cO_{F_v})$ to be the subgroup of matrices which are upper-triangular modulo $\varpi_v^b$, and $\Iw_v(b, c)$ to be either $\Iw_v(b)$ (if $c = 0$) or the kernel of the homomorphism $\Iw_v(b) \to \cO_{F_v}^\times / \langle \cO_{F_v}^\times[\textnormal{tors}], 1 + \varpi_v^c \cO_{F_v} \rangle$ defined by the formula
\[ \begin{pmatrix} x & y \\ z & w \end{pmatrix} \mapsto x w^{-1} \text{ mod }\varpi_v^c \]
(if $c \geq 1$).
We then write $U_0(b, c) \leq U_0$ for the open compact subgroup defined as follows:
\begin{itemize}
    \item If $v \not\in S_p$, then $U_0(b, c)_v = U_{0, v}$.
    \item If $v \in S_p$, then $U_0(b, c)_v = \Iw_v(b, c)$. 
\end{itemize}
If $Q$ is a set of Taylor--Wiles places of level $N \geq 1$, then we write $U_{Q(N), 1} \leq U_Q \leq U_0$ for the open compact subgroups defined as follows:
\begin{itemize}
    \item If $v \not\in Q$, then $U_{Q(N), 1, v} = U_{Q, v} = U_{0, v}$. 
    \item If $v \in Q$, then $U_{Q, v} = \Iw_v \coloneqq  \Iw_v(1)$ is the usual upper-triangular Iwahori subgroup (of matrices that are upper-triangular modulo $\varpi_v$) and $U_{Q(N), 1, v}$ is the kernel of the homomorphism $\Iw_v \to k(v)^\times(p) / (p^N)$ defined by the formula 
    \[ \begin{pmatrix} x & y \\ z & w \end{pmatrix} \in \Iw_v \mapsto x w^{-1} \text{ mod }\varpi_v. \]
\end{itemize}
If $Q$ is a set of Taylor--Wiles places and $b, c$ are as above, then we define $U_Q(b, c) \coloneqq  U_Q \cap U_0(b, c)$ and $U_{Q(N), 1}(b, c) \coloneqq  U_{Q(N), 1} \cap U_0(b, c)$. The spaces $H(U_Q(b, c))$ and $H(U_{Q(N), 1}(b, c))$ receive actions of various different kinds of Hecke operators:
\begin{itemize}
    \item Unramified Hecke operators $T_v$ at places $v \not\in S_p \cup T \cup S_{\textnormal{aux}} \cup Q$.
    \item Iwahori-type Hecke operators $U_v$ at places $v \in S_p \cup Q$.
    \item Quaternionic Hecke operators $U_v \coloneqq  [\Pi_v \cO_{D_v}^\times]$ for $v \in T$ (where $\Pi_v \in \cO_{D_v}$ is a choice of uniformizer). 
    \item Diamond operators $\langle u \rangle = [ U_{Q(N), 1}(b, c)_v \diag(u, 1) U_{Q(N), 1}(b, c)_v]$ for $u \in \cO_{F_v}^\times$ and $v \in S_p \cup Q$. 
\end{itemize}
We write $\mathbb{T}^{\textnormal{univ}}$ for the formal polynomial algebra over $\widetilde{\Lambda}$ in the indeterminates $T_v$ ($v \not \in S_p \cup T \cup S_{\textnormal{aux}}$), and $\mathbb{T}^{\textnormal{univ}}_{T \cup Q} \coloneqq  \mathbb{T}^{\textnormal{univ}}[U_v : v \in T \cup Q]$. We define maximal ideals of these rings as follows:
\begin{itemize}
    \item $\mathfrak{m} \leq \mathbb{T}^{\textnormal{univ}}$ is the maximal ideal with residue field $k$ generated by $\mathfrak{m}_{\widetilde{\Lambda}}$ and the elements $T_v - \tr \rhobarss(\Frob_v)$ ($v \not\in S_p \cup S_{\textnormal{aux}} \cup T$). 
    \item If $Q$ is part of the Taylor--Wiles datum $(Q, (\alpha_v^{(1)}, \alpha_v^{(2)})_{v \in Q}, (\lambda_v)_{v \in Q})$, then $\mathfrak{m}_{Q} \leq \mathbb{T}^{\textnormal{univ}}_{T \cup Q}$ is the maximal ideal with residue field $k$ generated by $\mathfrak{m}$ and the elements $U_v - \alpha_v^{(1)}$ ($v \in Q$) and $U_v - 1$ ($v \in T$).
    We allow $Q = \varnothing,$ in which case $\fm_\varnothing \coloneqq  (\fm, U_v-1: v \in T)$.
\end{itemize}
Let $U_p \coloneqq  \prod_{v \in S_p} U_v$, and let $H(U_{Q(N), 1}(b, c))^{\textnormal{ord}} \leq H(U_{Q(N), 1}(b, c))$ denote the $U_p$-ordinary part, i.e.\ the maximal direct summand on which $U_p$ acts invertibly. 
We define $H(U_Q(b, c))^{\textnormal{ord}}, H(U_0(b, c))^{\textnormal{ord}}$ in a similar way.

Note that $H(U_{Q(N), 1}(b, c))^{\textnormal{ord}}$ (resp.\ $H(U_Q(b, c))^{\textnormal{ord}}, H(U_0(b, c))^{\textnormal{ord}}$) is a $\mathbb{T}^{\textnormal{univ}}_{Q \cup T}$-module. 
Only the action of $\widetilde{\Lambda}$ and of $T_v$, $v\in Q$, is not contained in what we've already written. 
The operator $T_v$ ($v \in Q$) acts by the image of $\Art_{F_v}(\varpi_v)$ under the homomorphism from $W_{F_v}^{\textnormal{ab}}$ to the centre of the Hecke algebra at level $U_{Q(N), 1, v}$.
If~$v \in S_p$, we define the action of $\widetilde \Lambda_v$ using the splitting $F_v^\times(p)^{\textnormal{tf}} \cong \cO_{F_v}^\times(p)^{\textnormal{tf}} \times \varpi_v^{\mathbb{Z}_p}$. The action of the group $\cO_{F_v}^\times$ just defined factors through $\cO_{F_v}^\times(p)^{\textnormal{tf}}$; the operator $U_v$ acts invertibly on $H(U_{Q(N), 1}(b, c))^{\textnormal{ord}}$, by definition, so making $\varpi_v \in F_v^\times(p)$ act by $U_v$ defines the structure of $\widetilde{\Lambda}_v$-module.

The Jacquet--Langlands correspondence, together with the assumption in Thm.\ \ref{thm_tame_SW_black_box} about the existence of $\pi$, implies that $\mathfrak{m}_{\varnothing}$ is in the support of $H(U_0(1, 0))^{\textnormal{ord}}$.
(Here \cite[Lem.\ 4.1]{Tho16} shows that we have selected the correct residual eigenvalue if $v \in T$.) The diamond operators make any module $H(U_{Q(N), 1}(b, c))$ into a module for the group algebra of 
\[
\prod_{v \in S_p} (1 + \varpi_v \cO_{F_v})^{\textnormal{tf}} / (1 + \varpi_v^{\max(c, 1)} \cO_{F_v})^{\textnormal{tf}} \times \Delta_Q / p^N.
\]
The action of this algebra commutes with the action of $\mathbb{T}^{\textnormal{univ}}_{T \cup Q}$. 

We now describe a useful way of packaging these actions. Using the ordering of $Q$ that is part of the definition of Taylor--Wiles datum, we can view $H(U_{Q(N), 1}(b, c))^{\textnormal{ord}}$ as an $\cO[(\mathbb{Z}_p / (p^N))^{|Q|}]$-module.
In the next proposition, we write $m_0$ for the exponent of $p$ in the least common multiple of the integers $ | F^\times \backslash (U_0(1, 0) \mathbb{A}_F^{\infty, \times} \cap t^{-1} D^\times t ) |$ for $t \in (D \otimes_F \mathbb{A}_F^\infty)^\times$. 
\begin{pp}\label{prop_freeness_of_modular_forms_at_TW_level}
    Let $N \geq 1$, and let $Q$ be a Taylor--Wiles set of level $N + m_0$. Then:
    \begin{enumerate}
        \item For any $b' \geq b$, $c \leq b$, there are natural isomorphisms
        \[ H(U_{Q(N),1}(b, c))^{\textnormal{ord}}_{\mathfrak{m}_Q} \to H(U_{Q(N), 1}(b', c))^{\textnormal{ord}}_{\mathfrak{m}_Q} \]
        and
        \[ H(U_0(b, c))^{\textnormal{ord}}_{\mathfrak{m}_\varnothing} \to H(U_Q(b, c))^{\textnormal{ord}}_{\mathfrak{m}_Q}. \]
        \item For any $b \geq c$ with $b \geq 1$, $H(U_{Q(N),1}(b, c))^{\textnormal{ord}}_{\mathfrak{m}_Q}$ is a free module over the group algebra 
        \[ \cO \left [\prod\nolimits_{v \in S_p} (1 + \varpi_v \cO_{F_v})^{\textnormal{tf}} / (1 + \varpi_v^{\max(c, 1)} \cO_{F_v})^{\textnormal{tf}} \times (\mathbb{Z}_p / (p^N))^{|Q|} \right ] \]
        and there is an isomorphism
        \[ H(U_{Q(N),1}(b, c))^{\textnormal{ord}}_{\mathfrak{m}_Q} \otimes_{\cO\left [\prod\nolimits_{v \in S_p} (1 + \varpi_v \cO_{F_v})^{\textnormal{tf}} / (1 + \varpi_v^{\max(c, 1)} \cO_{F_v})^{\textnormal{tf}} \times (\mathbb{Z}_p / (p^N))^{|Q|} \right ]} \cO \cong H(U_0(1, 0))^{\textnormal{ord}}_{\mathfrak{m}_\varnothing} \]
        of $\mathbb{T}^{\textnormal{univ}}$-modules.
    \end{enumerate}
\end{pp}
\begin{proof} 
    We omit the proofs of these standard results from Hida theory and the Taylor--Wiles method for $\GL_2$. The observation that the spaces of modular forms here are free, provided $Q$ has level $N + m_0$, appears in Khare--Wintenberger \cite[\S 7]{Kha09}. An abstraction of the underlying principle appears again in \cite[Lem.\ 8.18]{Boc19}. 
\end{proof}
If $Q$ is part of a Taylor--Wiles datum of level~$N$ then $P^{\textnormal{ord}}_{Q}$ and $R^{\textnormal{ord}}_{b, Q}$ have natural structures of $\cO[\Delta_Q]$-algebra, arising from the restriction of the universal pseudocharacter to $G_{F_v}$ (for the places $v \in Q$). The following proposition summarises properties arising from the existence of Galois representations associated to automorphic representations of $(D \otimes_F \mathbb{A}_F)^\times$.
\begin{pp}
    Let $N \geq 1$, and let $Q$ be a Taylor--Wiles set of level $N + m_0$. Then:
    \begin{enumerate}
        \item For any $b \geq \max(c, 1)$, there is a unique structure on $H(U_{Q(N), 1}(b, c))^{\textnormal{ord}}_{\mathfrak{m}_Q}$ of $P^{\textnormal{ord}}_{ Q}$-module such that, for any finite place $v\not\in S_p \cup T \cup S_{\textnormal{aux}} \cup Q$ of $F$, $t^{\textnormal{univ}}(\Frob_v)$ acts by the Hecke operator $T_v$. 
        \item The two induced structures of $\widetilde{\Lambda}[\Delta_Q]$-module (one by diamond operators, the other by the $\widetilde{\Lambda}[\Delta_Q]$-algebra structure on $P^{\textnormal{ord}}_{Q}$) agree. 
        \item The natural map 
        \[ H(U_{Q(N), 1}(b, c))^{\textnormal{ord}}_{\mathfrak{m}_Q} \otimes_{P^{\textnormal{ord}}_{ Q}} P^{\textnormal{ord}} \to H(U_{0}(b, c))^{\textnormal{ord}}_{\mathfrak{m}_\varnothing} \]
        is an isomorphism of $P^{\textnormal{ord}}$-modules. 
    \end{enumerate}
\end{pp}
\begin{proof}
    The existence of a structure of $P(G_{F, S_p \cup T \cup S_{\textnormal{aux}} \cup Q}) \widehat{\otimes}_\cO \widetilde{\Lambda}$-module on $H(U_{Q(N), 1}(b, c))^{\textnormal{ord}}_{\mathfrak{m}_Q}$, and the fact that this action factors through the quotient $P_Q^{\textnormal{ord}}$, is a standard consequence of the existence of, and local-global compatibility for, the Galois representations attached to automorphic representations on $D$. 

    This shows (1) and (2); and (3) then follows from Prop.\ \ref{prop_freeness_of_modular_forms_at_TW_level}. 
\end{proof}
We define 
\[ H_0: = \varprojlim_{c \geq 1} H(U_0(c, c))^{\textnormal{ord}}_{\mathfrak{m}_\varnothing} \]
and
\[ H_{Q(N)} \coloneqq  \varprojlim_{c \geq 1} H(U_{Q(N), 1}(c, c))^{\textnormal{ord}}_{\mathfrak{m}_Q}.\]
\begin{cor}
    $H_0$ is a finitely generated $P^{\textnormal{ord}}$-module, finite free as $\Lambda$-module. Writing $\mathbb{T}_0$ for the quotient of $P^{\textnormal{ord}}$ that acts faithfully on $H_0$, $\mathbb{T}_0$ is a finite, faithful, reduced $\Lambda$-algebra, equidimensional of dimension $1 + [ F : \mathbb{Q} ]$.

    More generally, if $Q$ is a part of a Taylor--Wiles datum of level $N + m_0$, then $H_{Q(N)}$ is a finitely generated $P^{\textnormal{ord}}_{Q}$-module, finite free as $\Lambda[ \Delta_Q / (p^N)]$-module, and there is an isomorphism $H_{Q(N)} \otimes_{\cO[\Delta_Q]} \cO \cong H_0$ of $P^{\textnormal{ord}}_{Q} \otimes_{\cO[\Delta_Q]} \cO \cong P^{\textnormal{ord}}$-modules. 
\end{cor}
\begin{proof}
    We just consider the case $Q = \varnothing$. In this case, we need only justify the reducedness of $\mathbb{T}_0$. If $c \geq 1$, let $\mathbb{T}(H(U_0(c, c))^{\textnormal{ord}}_{\mathfrak{m}_\varnothing})$ denote the quotient of $\mathbb{T}^{\textnormal{univ}}$ that acts faithfully on $H(U_0(c, c))^{\textnormal{ord}}_{\mathfrak{m}_\varnothing}$. Then $\mathbb{T}(H(U_0(c, c))^{\textnormal{ord}}_{\mathfrak{m}_\varnothing})$ is reduced (it is $\cO$-flat and its $E$-fibre is \'etale) and we can identify $\mathbb{T}_0$ with the inverse limit 
    \[ \mathbb{T}_0 = \varprojlim_c\mathbb{T}(H(U_0(c, c))^{\textnormal{ord}}_{\mathfrak{m}_\varnothing}). \qedhere \]
\end{proof}

\subsection{First patching argument}\label{subsec:first patching argument}
Recall that we have fixed a model $\rho'$ over $\cO$ for $r_{\pi, \iota}$. Since we are assuming that $\pi_v$ is the Steinberg representation for all~$v \in T,$
the representation $\rho' \text{ mod } \varpi$ has a $k[G_{F_v}]$-stable line on which the $G_{F_v}$-action is trivial. We conclude that $\rho' \text{ mod }\varpi = \overline{\rho}_{b'}$ for some non-zero $[b'] \in H^1_{T\textnormal{-split}}(F_{S_p \cup S_{\textnormal{aux}}} / F, \overline{\chi})$.

Choose, for each $N \geq 1$, a Taylor--Wiles datum with set $Q_N= (v_1^{(N)}, \dots, v_{|T|}^{(N)} )$ of level $N + m_0$ satisfying the conditions of Prop.\ \ref{prop_existence_of_TW_places_first_patching}. 
Let $S_\infty \coloneqq  \cO \llbracket  \mathbb{Z}_p^{|T|} \rrbracket $, and let $\mathfrak{a} \leq S_\infty$ be the augmentation ideal. 
We also define 
\[
\mathfrak{a}_N \coloneqq  \ker(S_\infty \to \cO [ (\mathbb{Z}_p / (p^N))^{|T|}]) 
\]
and 
\[
\mathfrak{b}_N \coloneqq  (\mathfrak{m}_\Lambda^N, \mathfrak{a}_N) \leq \Lambda \widehat{\otimes}_\cO S_\infty. 
\]
Then $\mathfrak{b}_N$ is an open ideal of $\Lambda \widehat{\otimes}_\cO S_\infty$, which is a complete Noetherian local $\cO$-algebra. 

Let $d_0 \coloneqq  \dim_k H_0 / (\mathfrak{m}_\Lambda)$, and if $N \geq 1$, let $r_N \coloneqq  N d_0 |T| p^N$. 
We define 
\[
\mathcal{D}^s G_{F, S_p \cup T \cup S_{\textnormal{aux}} \cup Q_N} \coloneqq  P_s(\Gal(F_{S_p \cup T \cup S_{\textnormal{aux}} \cup Q_N} / F(\overline{\chi}_1, \overline{\chi}_2))), 
\]
with notation as in the statement of Lem.\ \ref{lem_quantitative_continuity_of_pseudocharacters}. Then 
\[ \mathcal{D}^s G_{F, S_p \cup T \cup S_{\textnormal{aux}} \cup Q_N} \leq G_{F, S_p \cup T \cup S_{\textnormal{aux}} \cup Q_N} \]
is an open subgroup, and the group
\[ \Gamma_{S_p \cup T \cup S_{\textnormal{aux}} \cup Q_N} \coloneqq  \varprojlim_s G_{F, S_p \cup T \cup S_{\textnormal{aux}} \cup Q_N} / \mathcal{D}^s G_{F, S_p \cup T \cup S_{\textnormal{aux}} \cup Q_N} \]
may be identified with the Galois group, over $F$, of the maximal pro-$p$ extension of $F(\overline{\chi}_1, \overline{\chi}_2)$ contained in $F_{S_p \cup T \cup S_{\textnormal{aux}} \cup Q_N}$. 

We define a patching datum of level $N \geq 1$ to consist of the following data:
\begin{enumerate}
    \item A complete Noetherian local $\Lambda \widehat{\otimes}_\cO S_\infty$-algebra $\mathbf{P}_N$, together with a homomorphism $\mathbf{P}_N \to P^{\textnormal{ord}} / \mathfrak{m}_{P^{\textnormal{ord}}}^{r_N}$ of $\Lambda \widehat{\otimes}_\cO S_\infty$-algebras.
    \item For each non-zero $[b] \in H^1_{T \textnormal{-split}}(F_{S_p \cup S_{\textnormal{aux}}} / F, \overline{\chi})$, a complete Noetherian local $\mathbf{P}_N$-algebra $\mathbf{R}_{b, N}$, together with a homomorphism $\mathbf{R}_{b, N} \to R^{\textnormal{ord}}_{b} / \mathfrak{m}_{R^{\textnormal{ord}}_{b}}^{r_N}$, making the diagram
    \[ \xymatrix{ \mathbf{P}_N \ar[r] \ar[d] & P^{\textnormal{ord}} / \mathfrak{m}_{P^{\textnormal{ord}}}^{r_N} \ar[d] \\ \mathbf{R}_{b, N} \ar[r] & R^{\textnormal{ord}}_{b} / \mathfrak{m}_{R^{\textnormal{ord}}_{b}}^{r_N} }
    \]
    commute.
    \item A $\mathbf{P}_N$-module $\mathbf{H}_N$, finite free as $\Lambda \widehat{\otimes}_\cO S_\infty / (\mathfrak{b}_N)$-module, together with an isomorphism $\mathbf{H}_N / (\mathfrak{b}_N, \mathfrak{a}) \cong H_0 / (\mathfrak{b}_N)$ of $\mathbf{P}_N$-modules (where $\mathbf{P}_N$ acts via its homomorphism to $P^{\textnormal{ord}} / \mathfrak{m}_{P^{\textnormal{ord}}}^{r_N}$).
    \item A finite group $\Gamma_N$, together with a pseudorepresentation $\mathbf{t}_N : \Gamma_N \to \mathbf{P}_N$ and, for each non-zero $[b] \in H^1_{T \textnormal{-split}}(F_{S_p \cup S_{\textnormal{aux}}} / F, \overline{\chi})$, a representation $\mathbf{r}_{b, N} : \Gamma_N \to \GL_2(\mathbf{R}_{b, N})$.
    \item A homomorphism $\Gamma_N \to G_{F, S_p \cup T \cup S_{\textnormal{aux}}} / \mathcal{D}^{3 r_N} G_{F, S_p \cup T \cup S_{\textnormal{aux}}}$ such that the pushforward of $\mathbf{t}_N$ (resp.\ $\mathbf{r}_{b, N}$) through $\mathbf {P}_N \to P^{\textnormal{ord}} / \mathfrak{m}_{P^{\textnormal{ord}}}^{r_N}$ (resp.\ $\mathbf{R}_{b, N} \to R^{\textnormal{ord}}_{b} / \mathfrak{m}_{R^{\textnormal{ord}}_{b}}^{r_N}$) agrees with the pseudocharacter (resp.\ representation) of $\Gamma_N$ pulled back from the universal pseudocharacter (resp.\ representation) 
    of $G_{F, S_p \cup T \cup S_{\textnormal{aux}}} / \mathcal{D}^{3 r_N} G_{F, S_p \cup T \cup S_{\textnormal{aux}}}$
    over the latter ring. 
    Here we use Lem.\ \ref{lem_quantitative_continuity_of_pseudocharacters} 
    to ensure that $G_{F, S_p \cup T \cup S_{\textnormal{aux}}}$ acts through its quotient $G_{F, S_p \cup T \cup S_{\textnormal{aux}}} / \mathcal{D}^{3 r_N} G_{F, S_p \cup T \cup S_{\textnormal{aux}}}$.
    \item An element $\mathbf{c}_u \in \Gamma_N$ lifting the image of $c_u$ in $G_{F, S_p \cup T \cup S_{\textnormal{aux}}} / \mathcal{D}^{3 r_N} G_{F, S_p \cup T \cup S_{\textnormal{aux}}}$.
    \item For each $v \in S_p \cup T \cup S_{\textnormal{aux}}$, a homomorphism $\gamma_v : G_{F_v} \to \Gamma_N$ lifting the tautological homomorphism $G_{F_v} \to G_{F, T \cup S_p\cup S_{\aux}} / \mathcal{D}^{3 r_N} G_{F, T \cup S_p\cup S_{\aux}}$. 
    \item For each $i = 1, \dots, |T|$, a homomorphism $\gamma_i : \mathbb{Z}_p \to \Gamma_N$.
    \item For each non-zero $[b] \in H^1_{T \textnormal{-split}}(F_{S_p \cup S_{\textnormal{aux}}} / F, \overline{\chi})$, a surjection 
    \begin{multline*} R_{\textnormal{loc}}^{\square_{S_p \cup S_{\textnormal{aux}}}} \widehat{\otimes}_\cO S_\infty\llbracket x_1, \dots x_g\rrbracket  / (\mathfrak{m}_{R_{\textnormal{loc}}^{\square_{S_p \cup S_{\textnormal{aux}}}} \widehat{\otimes}_\cO S_\infty\llbracket x_1, \dots x_g\rrbracket }^{r_N}) 
    \\ \to \mathbf{R}_{b, N} \llbracket  y_1, \dots, y_{4|S_p \cup S_{\textnormal{aux}} \cup T| - 1} \rrbracket  / \mathfrak{m}_{\mathbf{R}_{b, N} \llbracket  y_1, \dots, y_{4|S_p \cup S_{\textnormal{aux}} \cup T| - 1} \rrbracket }^{r_N} 
    \end{multline*}
    whose kernel may be generated by $r$ elements, where 
    \[ g = h^1_{S_p \cup S_{\textnormal{aux}}\textnormal{-split}}(F_{S_p \cup T} / F, \ad^0 \overline{\rho}_b) + 4(|S_p \cup S_{\textnormal{aux}}| + 2 |T|) \]
    and $g - r = - [ F : \mathbb{Q}] - 1 + |S_p \cup S_{\textnormal{aux}}| + 4|T|$.
\end{enumerate}
The patching data of level $N$ form a groupoid $\operatorname{Patch}_N$, where the morphisms 
\[(\mathbf{P}_N, \{ \mathbf{R}_{b, N} \}_b, \mathbf{H}_N, \Gamma_N, \mathbf{t}_N, \{ \mathbf{r}_{b, N} \}_b )  \to  (\mathbf{P}'_N, \{ \mathbf{R}'_{b, N} \}_b, \mathbf{H}'_N, \Gamma'_N, \mathbf{t}'_N, \{ \mathbf{r}'_{b, N} \}_b ) \]
are given by isomorphisms of the objects appearing in the given tuple making all the relevant diagrams commute. In particular, the isomorphism $\mathbf{P}_N \cong \mathbf{P}'_N$ should carry the pseudocharacter $\mathbf{t}_N$ of $\Gamma_N$ to the pseudocharacter $\mathbf{t}'_N$ of $\Gamma'_N$. 

For each $1 \leq N \leq M$ with $M \geq r_N + 1$, we can define a patching datum $D(N, M) \in \operatorname{Patch}_N$ as follows:
\begin{enumerate}
    \item $\mathbf{P}_N \coloneqq  P^{\textnormal{ord}}_{Q_M} / \mathfrak{m}_{P^{\textnormal{ord}}_{Q_M}}^{r_N}$.
    \item For each non-zero $[b] \in H^1_{T \textnormal{-split}}(F_{S_p \cup S_{\textnormal{aux}}} / F, \overline{\chi})$, $\mathbf{R}_{b, N} \coloneqq  R^{\textnormal{ord}}_{b,Q_M} / \mathfrak{m}_{R^{\textnormal{ord}}_{b, Q_M}}^{r_N}$.
    \item $\mathbf{H}_N \coloneqq  H_{Q_M(N)} / (\mathfrak{b}_N)$. 
    \item $\Gamma_N \coloneqq  G_{F, S_p \cup T \cup S_{\textnormal{aux}} \cup Q_M} / \mathcal{D}^{3 r_N} G_{F, S_p \cup T \cup S_{\textnormal{aux}} \cup Q_M}$, with $\mathbf{t}_N$, $\mathbf{r}_{b, N}$ the specialisations of the universal (pseudo-)characters.
    \item The homomorphism $\Gamma_N \to G_{F, S_p \cup T \cup S_{\textnormal{aux}}} / \mathcal{D}^{3 r_N} G_{F, S_p \cup T \cup S_{\textnormal{aux}}}$ is the tautological one.
    \item $\mathbf{c}_u$ is the image of the fixed complex conjugation $c_u$ in $\Gamma_N$.
    \item For each $v \in S_p \cup S_{\textnormal{aux}} \cup T$, the map $G_{F_v} \to \Gamma_N$ is the pushforward of the tautological one $G_{F_v} \to G_{F, S_p  \cup T \cup S_{\textnormal{aux}} \cup Q_M}$. 
    \item The homomorphism $\gamma_i : \mathbb{Z}_p \to \Gamma_N$ is the projection of the composite $\mathbb{Z}_p \to I_{F_{v_i^{(M)}}}(p) \to G_{F, S_p \cup T \cup S_{\textnormal{aux}} \cup Q_M}$, where $\mathbb{Z}_p \to I_{F_{v_i^{(M)}}}(p)$ is an arbitrary lift of the map $\mathbb{Z}_p \to k(v_i^{(M)})^\times(p)$ associated to the Taylor--Wiles datum. 
    \item The map 
    \[ R_{\textnormal{loc}}^{\square_{S_p \cup S_{\textnormal{aux}}}} \widehat{\otimes}_\cO S_\infty\llbracket x_1, \dots x_g\rrbracket  \to \mathbf{R}_{b, N} \llbracket  y_1, \dots, y_{4|S_p \cup T \cup S_{\textnormal{aux}}| - 1} \rrbracket  / \mathfrak{m}_{\mathbf{R}_{b, N} \llbracket  y_1, \dots, y_{4|S_p \cup T \cup S_{\textnormal{aux}}| - 1} \rrbracket }^{r_N} \]
    comes by quotienting out the map $R_{\textnormal{loc}}^{\square_{S_p \cup S_{\textnormal{aux}}}} \widehat{\otimes} \mathcal{O}[\Delta_{Q_M}] \llbracket x_1, \dots, x_g\rrbracket  \to R_{b, Q_M}^{\textnormal{ord}, \square_{S_p \cup S_{\textnormal{aux}} \cup Q_M}}$, obtained using the formalism of e.g.\ \cite[\S 5]{Tho16}, with $g = \dim_k H^1_{\mathcal{S}_{Q_M}, S_p \cup S_{\textnormal{aux}} \cup Q_M}(\ad^0 \overline{\rho}_b) + 3|Q_M|$, bounded above by 
    \begin{multline}\label{eqn:bound on g in Euler characteristic formula}
    h^1_{S_p \cup S_{\textnormal{aux}} \cup Q_M\textnormal{-split}}(F_{S_p \cup T} / F, \ad^0 \overline{\rho}_b) + 4(|S_p \cup S_{\textnormal{aux}} \cup T \cup Q_M|) \\ \leq  h^1_{S_p \cup S_{\textnormal{aux}}\textnormal{-split}}(F_{S_p \cup T} / F, \ad^0 \overline{\rho}_b) + 4(|S_p \cup S_{\textnormal{aux}}| + 2|T|) ,
    \end{multline}
    and making the identifications 
    \[
    R_v^\square = \cO[\Delta_v]\llbracket x_1, x_2, x_3\rrbracket 
    \]
    for~$v \in Q_M$ (justified as $R_v^\square$ is a formally smooth $\cO[\Delta_{v}]$-algebra), and 
    \[
    R_{b, Q_M}^{\textnormal{ord}, \square_{S_p \cup S_{\textnormal{aux}} \cup Q_M}} \cong R_{b, Q_M}^{\textnormal{ord}}\llbracket y_1, \dots, y_{4 |S_p \cup S_{\textnormal{aux}} \cup Q_M|-1}\rrbracket  
    \]
    (depending on a choice of representative for the universal deformation). Note that the condition $M \geq r_N + 1$ implies that the natural map
    \begin{equation}\label{eqn:kernel containment on Iwasawa algebra} 
    S_\infty \llbracket x_1, x_2, x_3\rrbracket  / (\mathfrak{m}_{S_\infty}^{r_N}) \to \cO[\Delta_{Q_M}]\llbracket x_1, x_2, x_3\rrbracket  / (\mathfrak{m}_{S_\infty}^{r_N}) \end{equation}
    is an isomorphism. The formula for $g-r$ uses Tate's Euler characteristic formula (and, in particular, the vanishing of the terms $h^0(F, \ad^0 \overline{\rho}_b)$ and $h^0(F, \ad^0 \overline{\rho}_b(1))$). Note finally that we can augment $g$ and $r$ in order to saturate the upper bound in \eqref{eqn:bound on g in Euler characteristic formula}.
\end{enumerate}
\begin{lemma}\label{lem_finiteness_of_patching_data}
    \leavevmode
    \begin{enumerate}
        \item With notation as above, $\mathbf{t}_N$ and $\mathbf{r}_{b, N}$ do factor through $\Gamma_N$, and $D(N, M)$ is a well-defined patching datum of level $N$.
        \item For fixed $N$, the patching data $D(N, M)$ ($M \geq r_N+1$) fall into finitely many isomorphism classes. 
    \end{enumerate}
\end{lemma}
\begin{proof}
    (1) The claim for $\mathbf{t}_N$ follows from Lem.\ \ref{lem_quantitative_continuity_of_pseudocharacters}. The claim for $\mathbf{r}_{b, N}$ is simpler (note that, if $A \in \CNL_\cO$ and $\mathfrak{m}_A^N = 0$, then $P_N(\widehat{\GL}_2(A)) = 1$). We also need to check the claim that the action of $P^{\textnormal{ord}}_{ Q_M}$ on $H_{Q_M(N)} / (\mathfrak{b}_N)$ factors through the quotient $P^{\textnormal{ord}}_{Q_M} / \mathfrak{m}_{P^{\textnormal{ord}}_{Q_M}}^{r_N}$. If $f_1, \dots, f_{d_0} \in \mathfrak{m}_{P^{\textnormal{ord}}_{Q_M}}$, then they induce commuting nilpotent endomorphisms of $H_{Q_M(N)} / (\mathfrak{a}, \mathfrak{m}_\Lambda)$, so the product $f_1 \dots f_{d_0}$ is 0, and we have 
    \[ f_1 \dots f_{d_0} \cdot H_{Q_M(N)} / (\mathfrak{b}_N) \leq (\mathfrak{a}, \mathfrak{m}_\Lambda) H_{Q_M(N)} / (\mathfrak{b}_N). \]
    Since $(\mathfrak{a}, \mathfrak{m}_\Lambda)^{|T| p^N} \leq (\mathfrak{a}_N, \mathfrak{m}_\Lambda)$, hence $(\mathfrak{a}, \mathfrak{m}_\Lambda)^{|T| N p^N} \leq (\mathfrak{a}_N, \mathfrak{m}_\Lambda^N) = \mathfrak{b}_N$, it follows that if $f_1, \dots, f_{d_0 |T| N p^N} \in \mathfrak{m}_{P^{\textnormal{ord}}_{Q_M}}$, then $f_1 \dots f_{d_0 |T| N p^N}$ acts as 0 on $H_{Q_M(N)} / (\mathfrak{b}_N)$. Since $r_N = d_0 |T| N p^N$, the claim follows.

    Finally, we also need to explain why the map
    \begin{multline*}  R_{\textnormal{loc}}^{\square_{S_p \cup S_{\textnormal{aux}}}} \widehat{\otimes}_\cO S_\infty\llbracket x_1, \dots x_g\rrbracket  / \mathfrak{m}_{R_{\textnormal{loc}}^{\square_{S_p \cup S_{\textnormal{aux}}}} \widehat{\otimes}_\cO S_\infty\llbracket x_1, \dots x_g\rrbracket }^{r_N} \\ \to R^{\textnormal{ord}}_{b, Q_M} \llbracket  y_1, \dots, y_{4|S_p \cup S_{\textnormal{aux}} \cup T| - 1} \rrbracket  / \mathfrak{m}_{R^{\textnormal{ord}}_{b, Q_M} \llbracket  y_1, \dots, y_{4|S_p \cup S_{\textnormal{aux}} \cup T| - 1} \rrbracket }^{r_N} 
    \end{multline*}
    determined by our choices has a kernel generated by $r$ elements, whereas the Euler characteristic formalism would imply only that a map 
    \begin{multline*} R_{\textnormal{loc}}^{\square_{S_p \cup S_{\textnormal{aux}}}} \widehat{\otimes}_\cO \cO[\Delta_{Q_M}]\llbracket x_1, \dots x_g\rrbracket  / \mathfrak{m}_{R_{\textnormal{loc}}^{\square_{S_p \cup S_{\textnormal{aux}}}} \widehat{\otimes}_\cO \cO[\Delta_{Q_M}]\llbracket x_1, \dots x_g\rrbracket }^{r_N} \\ \to R^{\textnormal{ord}}_{b,  Q_M} \llbracket  y_1, \dots, y_{4|S_p \cup S_{\textnormal{aux}} \cup T| - 1} \rrbracket  / \mathfrak{m}_{R^{\textnormal{ord}}_{b, Q_M} \llbracket  y_1, \dots, y_{4|S_p \cup S_{\textnormal{aux}} \cup T| - 1} \rrbracket }^{r_N} 
    \end{multline*}
    (thus with $S_\infty$ replaced by $\cO[\Delta_{Q_M}]$) has a kernel generated by $r$ elements. This replacement is justified provided that $\ker(S_\infty \to \cO[\Delta_{Q_M}]) \leq \mathfrak{a}_M$ is contained in $\mathfrak{m}_{S_\infty}^{r_N}$. As observed in~\eqref{eqn:kernel containment on Iwasawa algebra}, this is true because we impose the condition $M \geq r_N + 1$. 

    (2) It suffices to check for fixed $N$, there is a uniform bound on the cardinality of the objects $\mathbf{P}_N$, $\mathbf{R}_{b, N}$, $\Gamma_N$, etc.\ appearing in the definition of the patching data $D(N, M)$. This is true: an argument for the pseudodeformation ring is given in \cite[Lem.\ 3.28]{Tho15}, and this argument can be adapted essentially without change to the other objects appearing. 
\end{proof}
If $N \geq 1$, then there is a functor $\Pi_{N+1} : \operatorname{Patch}_{N+1} \to \operatorname{Patch}_N$ given by ``dividing out by $\mathfrak{m}^{r_N}$  and $\mathfrak{b}_N$''. By the second part of Lem.\ \ref{lem_finiteness_of_patching_data}, we can find a sequence of integers $1 \leq M_1 < M_2 < M_3 < \dots$ and, for every $N \geq 1$, an isomorphism $\Pi_{N+1}( D(N+1, M_{N+1}) ) \cong D(N, M_N)$ in $\operatorname{Patch}_{N}$. We therefore obtain inverse systems of rings $P^{\textnormal{ord}}_{ Q_{M_N}} / \mathfrak{m}_{P^{\textnormal{ord}}_{Q_{M_N}}}^{r_N}$, modules $H_{Q_{M_N}(N)} / (\mathfrak{b}_N)$, and groups $G_{F, S_p \cup T \cup S_{\textnormal{aux}} \cup Q_{M_N}} / \mathcal{D}^{3 r_N} G_{F, S_p \cup T \cup S_{\textnormal{aux}} \cup Q_{M_N}}$. Passing to the inverse limit, we obtain the following objects:
\begin{enumerate}
    \item A complete Noetherian local $\Lambda \widehat{\otimes}_\cO S_\infty$-algebra $P_\infty$, together with a homomorphism $P_\infty \to P^{\textnormal{ord}}$ of $\Lambda \widehat{\otimes}_\cO S_\infty$-algebras.
    \item For each non-zero $[b] \in H^1_{T\textnormal{-split}}(F_{S_p \cup S_{\textnormal{aux}}} / F, \overline{\chi})$, a complete Noetherian local $P_\infty$-algebra $R_{b, \infty}$, together with a homomorphism $R_{b, \infty} \to R^{\textnormal{ord}}_{b}$ making the diagram
    \[ \xymatrix{ P_\infty \ar[r] \ar[d] & P^{\textnormal{ord}} \ar[d] \\ 
    R_{b, \infty} \ar[r] & R^{\textnormal{ord}}_{b} }
    \]
    commute.
    \item A $P_\infty$-module $H_\infty$, finite free as $\Lambda \widehat{\otimes}_\cO S_\infty$-module, together with an isomorphism $H_\infty / (\mathfrak{a}) \cong H_0$ of $P_\infty$-modules. 
    \item A profinite group $\Gamma_\infty$, together with a pseudorepresentation $t_\infty : \Gamma_\infty \to P_\infty$, and for each non-zero $b \in H^1_{T \textnormal{-split}}(F_{S_p \cup S_{\textnormal{aux}}} / F, \overline{\chi})$, a representation $r_{b, \infty} : \Gamma_\infty \to \GL_2(R_{b, \infty})$. 
    \item A homomorphism $\Gamma_\infty \to \Gamma_{F, S_p \cup T \cup S_{\textnormal{aux}}}$ such that the pushforward of $t_\infty$ (resp.\ $r_{b, \infty}$) to $P^{\textnormal{ord}}$ (resp.\ $R^{\textnormal{ord}}_{b}$) agrees with the pseudocharacter (resp.\ representation) of $\Gamma_\infty$ pulled back from the universal pseudocharacter (resp.\ representation) over these rings. 
    \item A  lift $c_\infty \in \Gamma_\infty$ of $c_u$ of order 2.
    \item For each place $v \in S_p \cup T \cup S_{\textnormal{aux}}$, a homomorphism $G_{F_v} \to \Gamma_\infty$.
    \item For each $i = 1, \dots, |T|$, a homomorphism $\gamma_i : \mathbb{Z}_p \to \Gamma_\infty$, such that $t_\infty \circ \gamma_i = \mu_{i, 1} + \mu_{i, 2}$ is a direct sum of characters of~$\mathbb Z_p$, with the characters $\mu_{1, 1}, \dots, \mu_{|T|, 1}$ collectively determining the $S_\infty$-algebra structure on $P_\infty$. %
\end{enumerate}
\begin{lemma}\label{lem_patched_ring_represents_a_functor}
    \leavevmode
    \begin{enumerate}
        \item $P_\infty$ represents the functor of pairs of pseudodeformations of $t_\infty \text{ mod }\mathfrak{m}_{P_\infty}$ and characters $\psi_v$ ($v \in S_p$), of determinant $\epsilon^{-1}$, and satisfying the analogue of the local conditions defining $P^{\textnormal{ord}}$ at the start of \S \ref{subsec_deformation_and_pseudodeformation_rings} (well-defined because of the given homomorphisms $G_{F_v} \to \Gamma_\infty$, $v \in S_p \cup T \cup S_{\textnormal{aux}}$). 
        \item If the class $[b] \in H^1_{T\textnormal{-split}}(F_{S_p \cup S_{\textnormal{aux}}} / F, \overline{\chi})$ is non-zero, then $R_{b, \infty}$ represents the functor of deformations of $r_{b, \infty} \text{ mod } \mathfrak{m}_{R_{b, \infty}}$, of determinant $\epsilon^{-1}$, and satisfying the analogous local conditions at places $v \in S_p \cup T \cup S_{\textnormal{aux}}$. 
    \end{enumerate}
\end{lemma}
\begin{proof}
    We just show (1), which is the trickier of the two. Let $P(\Gamma_\infty) \widehat{\otimes}_\cO \widetilde{\Lambda} \to P'_\infty$ denote the representing object of the given functor. Then there's a map $P'_\infty \to P_\infty$ arising from the universal property of $P'_\infty$. We need to show that it's an isomorphism. To this end, consider a quotient $P'_\infty / \mathfrak{m}_{P'_\infty}^N$ for some $N \geq 1$. Then Lem.\ \ref{lem_quantitative_continuity_of_pseudocharacters} shows that the universal pseudocharacter $t(\Gamma_\infty) \text{ mod } \mathfrak{m}_{P'_\infty}^N$ factors through $\Gamma_\infty / \mathcal{D}^{3 N} \Gamma_\infty$, which may be identified with a quotient of the group
    \[ G_{F, S_p \cup T \cup S_{\textnormal{aux}} \cup Q_{M_N}}/\mathcal{D}^{3 N} G_{F, S_p \cup T \cup S_{\textnormal{aux}} \cup Q_{M_N}} \]
    (using here that $r_N \geq N$). We therefore find that there is a map 
    \[ P_\infty \to P^{\textnormal{ord}}_{Q_{M_N}} / \mathfrak{m}_{P^{\textnormal{ord}}_{ Q_{M_N}}}^{r_N} \to P'_\infty / \mathfrak{m}_{P'_\infty}^N. \]
    These maps are compatible as $N$ varies, and so passing to the limit gives a map $P_\infty \to P'_\infty$, which one checks is an inverse to the map arising from the universal property. 
\end{proof}
\begin{cor}\label{cor_comparison_of_patched_objects}
    Let $\mathfrak{q}_b \leq R_{b, \infty}$ be a prime ideal of dimension 1 and characteristic $p$ such that the associated representation is irreducible, and let $\mathfrak{q}^{\textnormal{ps}}$ denote its pullback to $P_\infty$. Then the induced map $\widehat{(P_{\infty, \mathfrak{q}^{\textnormal{ps}}})} \to \widehat{(R_{b, \infty, \mathfrak{q}_b})}$ is an isomorphism.
\end{cor}
\begin{proof}
    This follows from Lem.\ \ref{lem_patched_ring_represents_a_functor} and Lem.\ \ref{lem_comparison_of_P_and_R_b}. 
\end{proof}
Let $\mathcal{R}_\infty \leq P_\infty$ be the reducibility ideal of the pseudocharacter $t_\infty$ (by definition, generated by the elements $x(\sigma, \tau)$ for $\sigma, \tau \in \Gamma_\infty$---note that the definition of these elements in terms of $t_\infty$ uses our patched complex conjugation $c_\infty \in \Gamma_\infty$). If we are in the dihedral case, let $\mathcal{D}_\infty \leq P_\infty$ denote the $\overline{\chi}$-dihedral ideal, defined by the equation $t_\infty = t_\infty \otimes [\overline{\chi}]$ on the universal pseudocharacter. 
\begin{lemma}\label{lem_zariski_density}
    Let $\mathfrak{q}^{\textnormal{ps}}$ be a prime ideal of characteristic $p$, let $K \coloneqq  k(\mathfrak{q}^{\textnormal{ps}})$, and let $\rho_\mathfrak{q} : \Gamma_\infty \to \GL_2(K)$ be a semisimple representation with associated pseudocharacter $t_\infty \text{ mod }\mathfrak{q}^{\textnormal{ps}}$. Suppose that the following conditions are satisfied:
    \begin{enumerate}
        \item $\mathcal{R}_\infty \not\leq \mathfrak{q}^{\textnormal{ps}}$.
        \item There exists $g \in \Gamma_\infty$ such that the ratio of the eigenvalues of $\rho_{\mathfrak{q}}(g)$ is an element of $K^\times$ of infinite order.
        \item If we are in the dihedral case, then $\mathcal{D}_\infty \not\leq \mathfrak{q}^{\textnormal{ps}}$.
    \end{enumerate}
    Then $\rho_{\mathfrak{q}}$ is absolutely irreducible and $\ad^0 \rho_{\mathfrak{q}}$ is absolutely irreducible, even after restriction to the subgroup $\ker(\Gamma_\infty \to \Gal(F(\zeta_{p^\infty}) / F))$.
\end{lemma}
\begin{proof}
    We claim that the image of $\rho_{\mathfrak{q}}(\Gamma_\infty)$ in $\operatorname{PGL}_2(K)$ is Zariski dense. This in turn implies that the commutator subgroup $[\rho_{\mathfrak{q}}(\Gamma_\infty), \rho_{\mathfrak{q}}(\Gamma_\infty)]$ is Zariski dense in $\operatorname{SL}_2(K)$, implying the truth of the lemma. 

    To establish the claim, observe that the absolute irreducibility of $\rho_{\mathfrak{q}}$ follows immediately from the fact that $\mathcal{R}_\infty \not\leq \mathfrak{q}^{\textnormal{ps}}$. Let $G$ denote the Zariski closure of the image of $\rho_{\mathfrak{q}}(\Gamma_\infty)$ in $\operatorname{PGL}_2(K)$. The existence of an element in the image with eigenvalue ratio of infinite order implies that $G$ must contain a split torus $T$. By irreducibility, we must have either $G = \operatorname{PGL}_2$, or $G = N_{\operatorname{PGL}_2}(T)$. In the latter case, there is a non-trivial quadratic character $\delta : \Gamma_\infty \to k^\times$ such that $t = t \otimes \delta$. Reducing modulo the maximal ideal, we find $\overline{t} = \overline{t} \otimes \delta$, hence $\overline{\chi}_2 = \overline{\chi}_1 \delta$. This is only possible if $\overline{\chi}$ is quadratic and $\delta = \overline{\chi}$, i.e.\ we are in the dihedral case. However, this possibility is explicitly ruled out by our hypotheses. 
\end{proof}
\begin{pp}\label{prop_p_infty_red_dim}
    We have 
    \[ \dim P_\infty / (\mathcal{R}_\infty, \varpi) \leq 1 + \delta_F \]
    and, if we are in the dihedral case, 
    \[ \dim P_\infty / (\mathcal{D}_\infty, \varpi) \leq [F : \mathbb{Q}]. \]
\end{pp}
\begin{proof}
    This follows immediately from Prop.\ \ref{prop_existence_of_TW_places_first_patching} and its corollary. 
\end{proof}
\begin{pp}\label{prop_c_dim}
    Let $[b] \in H^1_{T \textnormal{-split}}(F_{S_p \cup S_{\textnormal{aux}}} / F, \overline{\chi})$ be non-zero. Then:
    \begin{enumerate}
     \item We have $\dim R_{b, \infty} / (\mathcal{R}_\infty, \varpi) \leq \delta_F + h + 1$. If we are in the dihedral case, then $\dim R_{b, \infty} / (\mathcal{D}_\infty, \varpi) \leq \max(\delta_F + h + 1, [F : \mathbb{Q}])$.
    \item The connectedness dimension of $R_{b, \infty}$ (see \cite{Bro86}) satisfies $\operatorname{cdim}(R_{b, \infty}) \geq [ F : \mathbb{Q} ] + |T|$.
    \item Each irreducible component of $\Spec R_{b, \infty}$ has dimension at least $1 + [F : \mathbb{Q}] + |T|$. More generally, if~$I_b \subset R_{b, \infty}$ is an ideal which can be generated by~$t$ elements, then each irreducible component of $\Spec R_{b, \infty}/I_b$ has dimension at least $1+[F:\mathbb Q]+|T|-t$.
    \end{enumerate}
\end{pp}
\begin{proof}
    (1) Lem.\ \ref{lem_bounding_reducible_locus} and Prop.\ \ref{prop_existence_of_TW_places_first_patching} together imply that for any $N \geq 1$, there is a surjection
    \[ A_{Q_N} \llbracket  x_1, \dots , x_{h + \delta_F + 1} \rrbracket  \to R^{\textnormal{ord}}_{b, Q_N} / (\mathcal{R}_{Q_N}, \varpi), \]
    where $A_{Q_N}$ is an Artinian $k$-algebra of uniformly bounded dimension. By patching, we conclude that there is a surjection 
    \[ A_{\infty} \llbracket  x_1, \dots , x_{h + \delta_F +1} \rrbracket  \to R_{b, \infty} / (\mathcal{R}_{\infty}, \varpi), \]
    where $A_\infty$ is again an Artinian $k$-algebra. This immediately implies the desired upper bound on the dimension.

    Suppose we are in the dihedral case, and let $Q \leq R_{b, \infty} / (\mathcal{D}_\infty, \varpi)$ be a minimal prime. If the specialisation of the universal lifting to $k(Q)$ is reducible, then $R_{b, \infty} / (\mathcal{D}_\infty, \varpi, Q)$ is a quotient of $R_{b, \infty} / (\mathcal{R}_\infty, \varpi)$, and we're done by the previous paragraph. Otherwise, the specialisation is irreducible, and Lem.\ \ref{lem_comparison_of_P_and_R_b} implies that we have
    \[ \dim R_{b, \infty} / (\mathcal{D}_\infty, \varpi, Q) \leq \dim P_\infty / (\mathcal{D}_\infty, \varpi) \leq [ F : \mathbb{Q} ], \]
    by Prop.\ \ref{prop_p_infty_red_dim}. 

    (2) By patching, and by Lem.\ \ref{lem_generation_of_kernel} below,  there is a presentation
    \[ R_{b, \infty} \llbracket y_1, ..., y_{4|S_p \cup S_{\textnormal{aux}} \cup T| - 1}\rrbracket  \cong \left( R_{\textnormal{loc}}^{\square_{S_p \cup S_{\textnormal{aux}}}} \widehat{\otimes}_\cO S_\infty \right) \llbracket x_1, \dots x_g \rrbracket  / (f_1, \dots, f_r), \]
    where $g-r = -[F : \mathbb{Q}] - 1 + |S_p \cup S_{\textnormal{aux}}| + 4 |T|$. By \cite[Thm.\ 2.4]{Bro86}, it follows that the connectedness dimension of $R_{b, \infty}$ is at least
    \begin{multline*} \dim R_{\textnormal{loc}}^{\square_{S_p \cup S_{\textnormal{aux}}}} + |T| + g - r - 1 - 4(|S_p \cup S_{\textnormal{aux}} \cup T|) + 1 \\= 1 + 3|S_p \cup S_{\textnormal{aux}}| + 2[F:\mathbb{Q}] + |T| - [F : \mathbb{Q}] - 1 + |S_p \cup S_{\textnormal{aux}}| + 4 |T| - 1 - 4(|S_p \cup S_{\textnormal{aux}} \cup T|) + 1 \\
    =   [F : \mathbb{Q}] + |T|. 
    \end{multline*}
    This concludes the proof of (2); and then the bounds in (3) also follows from the existence of this presentation.  
\end{proof}
\begin{lemma}\label{lem_generation_of_kernel}
    Let $R$ be a Noetherian local ring, and let $f : R \to S$ be a surjective ring homomorphism. Suppose that for all $n \geq 1$, the kernel of the induced map $R / \mathfrak{m}_R^n \to S / \mathfrak{m}_S^n$ can be generated by $r$ elements. Then the kernel of $f$ can be generated by $r$ elements.
\end{lemma}
\begin{proof}
    Let $I = \ker f$, and let $f_n : R / \mathfrak{m}_R^n \to S / \mathfrak{m}_S^n$ be the induced homomorphism. Then $I_n = \ker f_n = (I + \mathfrak{m}_R^n) / \mathfrak{m}_R^n$, and our hypothesis is that $\ker f_n$ can be generated by $r$ elements for all $n \geq 1$.

    There is a natural isomorphism $I_n \cong I / (I \cap \mathfrak{m}_R^n)$. By the Artin--Rees lemma, there exists $k \geq 1$ such that for all $n \geq k$, we have $I \cap \mathfrak{m}_R^n = \mathfrak{m}_R^{n-k} (I \cap \mathfrak{m}_R^k)$. In particular, we have 
    \[ I \cap \mathfrak{m}_R^{k+1} = \mathfrak{m}_R ( I \cap \mathfrak{m}_R^k) \leq \mathfrak{m}_R I. \]
    It follows that
    \[ I_{k+1} / \mathfrak{m}_R I_{k+1} \cong I / (I \cap \mathfrak{m}_R^{k+1}, \mathfrak{m}_R I) = I / \mathfrak{m}_R I. \]
    Our hypothesis states that $I_{k+1} / \mathfrak{m}_R I_{k+1}$ has dimension at most $r$ as a vector space over $R / \mathfrak{m}_R$; Nakayama's lemma then implies that $I$ may be generated by at most $r$ elements. 
\end{proof}
We set $\mathcal{M}_\infty \coloneqq  \operatorname{Ann}_{P_\infty} H_\infty$, and $\mathbb{T}_\infty \coloneqq  P_\infty / \mathcal{M}_\infty$. Thus $\mathbb{T}_\infty$ is a finite $\Lambda \widehat{\otimes}_{\cO} S_\infty$-algebra, and the kernel of the natural homomorphism $\mathbb{T}_\infty \otimes_{S_\infty} \cO \to \mathbb{T}_0$ is nilpotent (using e.g.\ \cite[Lem.\ 2.2]{Tay08}).
\begin{defn}\label{defn:nice prime ideal}
    Let $\mathfrak{q} \leq \mathbb{T}_\infty$ be a prime ideal of dimension 1 and characteristic $p$. We say that $\mathfrak{q}$ is nice if it satisfies the following conditions:
    \begin{enumerate}
        \item $\mathcal{R}_\infty \mathbb{T}_\infty \not\leq \mathfrak{q}$. If we are in the dihedral case, then $\mathcal{D}_\infty  \mathbb{T}_\infty \not\leq \mathfrak{q}$. 
        \item Let $A$ denote the normalisation of $\mathbb{T}_\infty / \mathfrak{q}$ in its field of fractions. Then the map $P_\infty \to k(\mathfrak{q})$ extends to a homomorphism $R_{b, \infty} \to A$ for some non-zero $[b] \in H^1_{T \textnormal{-split}}(F_{S_p \cup S_{\textnormal{aux}}} / F, \overline{\chi})$. 
    (In this case, we say that $\mathfrak{q}$ arises from $R_{b, \infty}$.)
    \end{enumerate}
\end{defn}
\begin{lemma}
    Let $\mathfrak{q}$ be a nice prime, let $K \coloneqq  k(\mathfrak{q})$, and let $\rho_{\mathfrak{q}} : \Gamma_\infty \to \GL_2(K)$ be a representation with corresponding pseudocharacter $t_\infty \text{ mod }\mathfrak{q}^{\textnormal{ps}}$. Then $\ad^0 \rho_{\mathfrak{q}}$ is absolutely irreducible, even after restriction to the subgroup $\ker(\Gamma_\infty \to \Gal(F(\zeta_{p^\infty}) / F))$. 
\end{lemma}
\begin{proof}
    By Lem.\ \ref{lem_zariski_density}, it suffices to show there is $g \in \Gamma_\infty$ such that the ratio of the eigenvalues of $\rho_{\mathfrak{q}}(g)$ lies in $K$ and has infinite order. Suppose not. Then, in particular, for each $v \in S_p$, the character $\psi_v \text{ mod }\mathfrak{q}$ is trivial, showing that the homomorphism $\Lambda \to K$ factors through $\Lambda \to k$; and the characters $\mu_{i, 1} \text{ mod }\mathfrak{q}$ must be trivial, showing that the homomorphism $S_\infty \to K$ must factor through $S_\infty \to k$. Therefore the map $\Lambda \widehat{\otimes}_\cO S_\infty \to \mathbb{T}_\infty / \mathfrak{q}$ is finite, yet factors through the quotient $\Lambda \widehat{\otimes}_\cO S_\infty \to k$, implying that $\mathbb{T}_\infty / \mathfrak{q}$ must be an Artinian $k$-algebra. Since this ring is, by hypothesis, of Krull dimension 1, this is a contradiction.
\end{proof}
\begin{pp}\label{prop_existence_of_nice_prime}
    There exists a nice prime $\mathfrak{q} \leq \mathbb{T}_\infty$. 
\end{pp}
\begin{proof}
    Recall from the start of Section~\ref{subsec:first patching argument} the lift $\rho'$ of $\overline{\rho}_{b'}$ associated to $\pi$. This determines a map 
    \[
    R_{b', \infty} \to R_{b'}^{\ord} \to \cO 
    \]
    whose pullback to $P_\infty$ factors through $\mathbb{T}_\infty$. Let $\mathfrak{p}_{b'} \leq R_{b', \infty}$, $\mathfrak{p}^{\textnormal{ps}} \leq P_\infty$, $\mathfrak{p} \leq \mathbb{T}_\infty$ be the respective kernels. The map $\widehat{(P_{\infty, \mathfrak{p}^{\textnormal{ps}}})} \to \widehat{(R_{b', \infty, \mathfrak{p}_{b'}})}$ is an isomorphism (same argument as Cor.\ \ref{cor_comparison_of_patched_objects}, using \cite[Prop.\ 2.2.1]{Pan22}). On the other hand, we have $\dim_E \mathfrak{p}^{\textnormal{ps}} / ((\mathfrak{p}^{\textnormal{ps}})^2, \mathfrak{a}) \otimes_\cO E = [F : \mathbb{Q}]$, using the main theorem of \cite{New23}. This forces the map $\widehat{(P_{\infty, \mathfrak{p}^{\textnormal{ps}}})} \to \widehat{(\mathbb{T}_{\infty, \mathfrak{p}})}$ to be an isomorphism (for dimension reasons). We conclude that if $Q_{b'}$ is a minimal prime of $R_{b', \infty}$ contained in $\mathfrak{p}_{b'}$, and $Q^{\textnormal{ps}}$ is its pullback to $P_\infty$, then $\dim P_\infty / Q^{\textnormal{ps}} = \dim R_{b', \infty} / Q_{b'}$ and $\mathcal{M}_\infty \leq Q^{\textnormal{ps}}$. 

    Since 
    \[ \dim R_{b' ,\infty} / (Q_{b'}, \mathcal{R}_\infty, \varpi) \leq 1 + \delta_F + h < [F : \mathbb{Q}] + |T|, \] and (if we are in the dihedral case), 
    \[ \dim R_{b' ,\infty} / (Q_{b'}, \mathcal{D}_\infty, \varpi) \leq \max(1 + \delta_F + h, [F : \mathbb{Q}]) < [F : \mathbb{Q}] + |T|, \]
    we can find a prime ideal $\mathfrak{q}_{b'}$ of dimension 1 and characteristic $p$ containing $Q_{b'}$, and in particular containing $\mathcal{M}_\infty R_{b', \infty}$, that avoids the reducible and/or dihedral locus. Let $\mathfrak{q}^{\textnormal{ps}}$ denote the pullback of $\mathfrak{q}_{b'}$ to $P_\infty$, and $\mathfrak{q}$ its image in $\mathbb{T}_\infty$. Lem.\ \ref{lem_comparison_of_P_and_R_b} shows that $k(\mathfrak{q}) = k(\mathfrak{q}_{b'})$. Thus $\mathfrak{q}$ is the desired nice prime. 
\end{proof}
The key technical ingredient in the rest of the proof is the following proposition, whose proof will be given in Section \ref{subsec_second_patching} below:
\begin{pp}\label{prop_local_R_equals_T}
    Let $\mathfrak{q} \leq \mathbb{T}_\infty$ be a nice prime, and let $\mathfrak{q}^{\textnormal{ps}}$ be its pullback to $P_\infty$. Then the induced map $\widehat{(P_{\infty, \mathfrak{q}^{\textnormal{ps}}})} \to \widehat{(\mathbb{T}_{\infty, \mathfrak{q}})}$ has nilpotent kernel. 
\end{pp}
In the remainder of Section \ref{subsec:first patching argument}, we show how Prop.\ \ref{prop_local_R_equals_T} implies Thm.\ \ref{thm_tame_SW_black_box}.
\begin{pp}\label{prop_propagating_pro_modularity}
    Let $b \in H^1_{T\textnormal{-split}}(F_{S_p \cup S_{\textnormal{aux}}} / F, \overline{\chi})$ be non-zero, and suppose that there exists a nice prime $\mathfrak{q}$ arising from $R_{b, \infty}$. Then the ideal $\mathcal{M}_\infty R_{b, \infty}$ is nilpotent. 
\end{pp}
\begin{proof}
    It suffices to show that $\mathcal{M}_\infty R_{b, \infty}$ is contained in every minimal prime of $R_{b, \infty}$. We first show that it is contained in some minimal prime. More precisely, let $A$ denote the normalisation of $\mathbb{T}_\infty / \mathfrak{q}$ in its field of fractions, and let $R_{b, \infty} \to A$ denote the extension of $P_\infty \to A$ implicit in the statement of the proposition. Let $\mathfrak{q}_b \leq R_{b, \infty}$ be its kernel, and let $Q \leq R_{b, \infty}$ be a minimal prime contained in $\mathfrak{q}_b$. We first claim that $\mathcal{M}_\infty R_{b, \infty} \leq Q$.

    To see this, note that we have, by construction, a diagram of complete local rings
    \[ \xymatrix{ \widehat{(R_{b, \infty, \mathfrak{q}_b})} & \\ \widehat{(P_{\infty, \mathfrak{q}^{\textnormal{ps}}})} \ar[u] \ar[r] & \widehat{(\mathbb{T}_{\infty, \mathfrak{q}})} } \]
    and that, by Prop.\ \ref{prop_local_R_equals_T}, $\mathcal{M}_\infty P_{\infty, \mathfrak{q}^{\textnormal{ps}}}$ is nilpotent. It follows that $\mathcal{M}_\infty R_{b, \infty, \mathfrak{q}_b}$ is nilpotent, and therefore that $\mathcal{M}_\infty R_{b, \infty}$ is contained in $Q$. 

    Suppose next that there is a minimal prime of $R_{b, \infty}$ that does not contain $\mathcal{M}_\infty R_{b, \infty}$. Then, by Prop.\ \ref{prop_c_dim} (and by the definition of connectedness dimension), we can find minimal primes $Q_1, Q_2 \leq R_{b, \infty}$ with the following properties:
    \begin{itemize}
        \item $\mathcal{M}_\infty R_{b, \infty} \leq Q_1$.
        \item $\mathcal{M}_\infty R_{b, \infty} \not\leq Q_2$.
        \item The dimension of $R_{b, \infty} / (Q_1, Q_2)$ is at least $[F : \mathbb{Q}] + |T|$.
    \end{itemize}
    We claim that we can find a prime ideal $\mathfrak{p}_b \leq R_{b, \infty}$ of dimension 1 and characteristic $p$ with the following properties:
    \begin{itemize}
        \item $\mathfrak{p}_b$ contains $(Q_1, Q_2)$, and in particular $\mathcal{M}_\infty R_{b, \infty}$.
        \item We have $\mathcal{R}_\infty R_{b, \infty} \not\leq \mathfrak{p}_b$. If we are in the dihedral case, then $\mathcal{D}_\infty R_{b, \infty} \not\leq \mathfrak{p}_b$.
    \end{itemize}
    This is possible provided that the inequality
    \[ \dim R_{b, \infty} / (Q_1, Q_2, \varpi) >  \dim R_{b, \infty} / (\mathcal{R}_\infty, \varpi) \]
    (and, if we are in the dihedral case, the inequality
    \[ \dim R_{b, \infty} / (Q_1, Q_2, \varpi) >  \dim R_{b, \infty} / (\mathcal{D}_\infty, \varpi)) \]
     holds. Using Prop.\ \ref{prop_c_dim}, we see that this would be true if 
    \[ [F : \mathbb{Q}] + |T| - 1 > \max(\delta_F + h + 1, [F : \mathbb{Q}]), \]
    which follows from our assumption that $|T| > h + 1$.
    
    Let $\mathfrak{p}^{\textnormal{ps}}$ denote the pullback of $\mathfrak{p}_b$ to $P_\infty$. Then $\mathcal{M}_\infty \leq \mathfrak{p}^{\textnormal{ps}}$, and so the latter prime is pulled back from a prime ideal $\mathfrak{p} \leq \mathbb{T}_\infty$. To check that $\mathfrak{p}$ is nice, we just need to check that $R_{b, \infty} / \mathfrak{p}_b$ and $P_\infty / \mathfrak{p}^{\textnormal{ps}}$ have the same fraction field (so that $R_{b, \infty} / \mathfrak{p}_b$ is in particular contained in the normalisation of $P_\infty / \mathfrak{p}^{\textnormal{ps}}$ in its field of fractions). We can choose $\sigma_0 \in \Gamma_\infty$ and $\alpha \in R_{b, \infty}$ such that $\diag(\alpha, 1) r_{b, \infty}(\sigma_0) \diag(\alpha, 1)^{-1}$ has upper right entry equal to 1. Then $R_{b, \infty} / \mathfrak{p}_b$, and hence its fraction field, is topologically generated by the matrix entries of $\diag(\alpha, 1) r_{b, \infty} \diag(\alpha)^{-1} \text{ mod }\mathfrak{p}_b$. However, formula \eqref{eq_rep_of_irreducible_ps} shows that these entries are already contained in $\operatorname{Frac}(P_\infty / \mathfrak{p}^{\textnormal{ps}})$. 
    
    We can now apply Prop.\ \ref{prop_local_R_equals_T} once more to conclude that $\mathcal{M}_\infty P_{\infty, \mathfrak{p}}$ is nilpotent, and hence that $\mathcal{M}_\infty R_{b, \infty} \leq Q_2$, a contradiction. This contradiction concludes the proof. 
\end{proof}
We can now complete the proof of Thm.\ \ref{thm_tame_SW_black_box}. 
\begin{proof}[Proof of Thm.\ \ref{thm_tame_SW_black_box}]
    After conjugation, we may assume that the representation that we must prove is modular factors through a continuous homomorphism $\rho : G_{F, S_p \cup T \cup S_{\textnormal{aux}}} \to \GL_2(\cO)$.
    A standard classicality result in Hida theory implies that it suffices to show that the morphism $P^{\ord} \to \cO$ classifying~$\rho$
    factors through~$\mathbb{T}_0$.
    Since $\mathbb{T}_\infty \otimes_{S_\infty} \cO \to \mathbb{T}_0$ is surjective with nilpotent kernel, it suffices
    to show that the pulled back morphism $P_\infty \to \cO$ factors through $\mathbb{T}_\infty$.

    After further conjugation, we can assume that $\rho  \text{ mod }\varpi$ is equal to $\overline{\rho}_{b_\rho}$, for a cocycle $b_\rho$ representing a non-zero class in $H^1_{T\textnormal{-split}}(F_{S_p \cup S_{\textnormal{aux}}} / F, \overline{\chi})$. By Prop.\ \ref{prop_existence_of_nice_prime} and Prop.\ \ref{prop_propagating_pro_modularity}, we can assume the existence of another cocycle $b$ representing a non-zero class such that there is a nice prime of $\mathbb{T}_\infty$ arising from $R_{b, \infty}$ and that $\mathcal{M}_\infty R_{b, \infty}$ is nilpotent. 
    We can assume that the cohomology classes of $b$, $b_\rho$ are independent (otherwise, $\rho$ determines an $\cO$-point of $R_{b, \infty}$, and then of $\mathbb{T}_\infty$, as desired).
    We are going to prove that there exists a nice prime $\mathfrak q$ of~$\mathbb{T}_\infty$ arising from $R_{b_{\rho}, \infty}$.
    Prop.\ \ref{prop_propagating_pro_modularity} then implies that $\mathcal{M}_\infty R_{b_\rho, \infty}$ is nilpotent, and therefore that the restriction of the homomorphism $R_{b_\rho, \infty} \to \cO$ associated to $\rho$ to $P_\infty$ factors through $\mathbb{T}_\infty$.

    We begin by completing $b_1 \coloneqq  b$, $b_2 \coloneqq  b_\rho$ to a set $b_1, b_2, \dots, b_h$ whose image in $H^1_{T\textnormal{-split}}(F_{ S_p \cup S_{\textnormal{aux}}} / F, \overline{\chi})$ forms a basis for this $k$-vector space. We choose elements $\sigma_1, \dots, \sigma_h$ of $\Gamma_\infty$ projecting to elements in $G_{F(\overline{\chi}_1, \overline{\chi}_2)}$ such that the functionals on $H^1_{T\textnormal{-split}}(F_{ S_p \cup S_{\textnormal{aux}}} / F, \overline{\chi})$ induced by evaluation at $\sigma_1, \dots, \sigma_h$ are $k$-linearly independent, and elements $a_{ij} \in \cO$ such that $\sum_{i=1}^h a_{ij} b_l(\sigma_i) = \delta_{jl}$ for each $j, l = 1, \dots, h$. Let $z_j \coloneqq \sum_{i=1}^h a_{ij} \sigma_i \in \cO[\Gamma_\infty]$. 

    We consider a lift $\widetilde{\rho}$ of $\overline{\rho}_b$ over $A \coloneqq  k\llbracket X\rrbracket $, given by the formula
    \[ \widetilde{\rho}(\sigma) \coloneqq  \begin{pmatrix} \overline{\chi}_1(\sigma) & \overline{\chi}_2 (\sigma)(b_1(\sigma) + X b_2(\sigma)) \\ 0 & \overline{\chi}_2(\sigma) \end{pmatrix}. \]
    Then $\widetilde{\rho}(z_1) = \left( \begin{smallmatrix} 1 & 1 \\ 0 & 1 \end{smallmatrix} \right)$,  $\widetilde{\rho}(z_2) = \left( \begin{smallmatrix} 1 & X \\ 0 & 1 \end{smallmatrix} \right)$, and $\widetilde{\rho}(z_j) = 1_2$ if $j = 3, \dots, h$.
    
    Let $\mathfrak{p}_b \leq R_{b, \infty}$ denote the kernel of the homomorphism $R_{b, \infty} \to A$ induced by~$\widetilde{\rho}$.
    Then $r_{b, \infty} \otimes_{R_{b, \infty}} R_{b, \infty}/\mathfrak{p}_b$ is $\widehat{\GL}_2$-conjugate to~$\widetilde{\rho}$. Let $\alpha \coloneqq r_{b, \infty}(z_1)_{12} \in  1 +\fm_{R_{b, \infty}}$, and let 
    \[ \widetilde{r}_{b, \infty} \coloneqq  \diag(\alpha^{-1}, 1) r_{b, \infty} \diag(\alpha, 1) = \begin{pmatrix} \widetilde{a}(\sigma) & \widetilde{b}(\sigma) \\ \widetilde{c}(\sigma) & \widetilde{d}(\sigma) \end{pmatrix}. \]
    Then $\widetilde{r}_{b, \infty}(z_1)_{12} \text{ mod } \mathfrak{p}_b = 1$ and $\widetilde{r}_{b, \infty}(c_\infty) = \operatorname{diag}(-1, 1)$, hence $\widetilde{r}_{b, \infty} \text{ mod }\mathfrak{p}_b = \widetilde{\rho}$ (equality, not just conjugacy). In particular, we have $\widetilde{r}_{b, \infty}(z_j)_{12} \equiv 0 \text{ mod }\mathfrak{p}_b$ for each $j = 3, \dots, h$. 
    
    Let $Q_b$ denote a minimal prime of $R_{b, \infty} / (\widetilde{r}_{b, \infty}(z_3)_{12}, \dots, \widetilde{r}_{b, \infty}(z_h)_{12})$ contained in $\mathfrak{p}_b$. Then 
    \[
    \dim R_{b, \infty} / Q_b \geq 1 + [F : \mathbb{Q}] + |T| - (h-2) = [F : \mathbb{Q}] + |T| - h + 3, 
    \]
    where the inequality uses Prop.\ \ref{prop_c_dim}~(3). 
    On the other hand, we have
    \[ \dim R_{b, \infty} / (\mathcal{R}_\infty, Q_b, \varpi) \leq \dim P_{\infty} / (\mathcal{R}_\infty, \varpi) + \dim R_{b, \infty} / (\mathfrak{m}_{P_\infty}, Q_b).  \]
    Prop.\ \ref{prop_p_infty_red_dim} shows that $\dim P_{\infty} / (\mathcal{R}_\infty, \varpi) \leq 1 + \delta_F$. We have
    \[ R_{b, Q}^{\textnormal{ord}} / (\mathfrak{m}_{P^{\textnormal{ord}}_{ Q}}) = R_{b}^{\textnormal{ord}} / (\mathfrak{m}_{P^{\textnormal{ord}}}) \]
    for any Taylor--Wiles datum $Q$. 
    By patching, we deduce that $R_{b, \infty}/(\fm_{P_\infty}) = R_b^{\textnormal{ord}}/(\fm_{P^{\textnormal{ord}}})$,
    and therefore $R_{b, \infty} / (\mathfrak{m}_{P_\infty}, Q_b)$ is a quotient of 
    \[ R_{b}^{\textnormal{ord}} / (\mathfrak{m}_{P^{\textnormal{ord}}}, \widetilde{r}_{b, \infty}(z_3)_{12}, \dots, \widetilde{r}_{b, \infty}(z_h)_{12}), \]
    which has Krull dimension at most 2, by Lem.\ \ref{lem_bounding_reducible_locus}. We conclude that 
    \[ \dim R_{b, \infty} / (\mathcal{R}_\infty, Q_b, \varpi) \leq 1 + \delta_F + 2 = 3 + \delta_F. \]
    Our hypotheses imply that
    \[ 3 + \delta_F < [F : \mathbb{Q}] + |T| - h + 3, \]
    so we further conclude that $\mathcal{R}_{b, \infty}$ is not contained in~$Q_b$, hence $\widetilde{r}_{b, \infty} \otimes k(Q_b)$ is irreducible.

    Let $Q$ denote the pullback of $Q_b$ to $P_\infty$. Then $x(z_j, \tau) \in Q$ for each $j = 3, \dots, h$ and $\tau \in \Gamma_\infty$, since its image in $R_{b, \infty}$ is $\widetilde{r}_{b, \infty}(z_j)_{12}\widetilde{r}_{b, \infty}(\tau)_{21}$. 
    Since $\mathcal{M}_\infty R_{b, \infty}$ is nilpotent, it is contained in~$Q_b$, and so $\mathcal{M}_\infty \leq Q$. 
    By Lem.\ \ref{lem_comparison_of_P_and_R_b}, we have 
    \[ \dim P_\infty / Q \geq \dim R_{b, \infty} / Q_b \geq [F : \mathbb{Q}] + |T| - h + 3. \]
    Let $I_c \leq P_\infty / Q$ be the ideal generated by the elements $x(z_1, \tau)$ for $\tau \in \Gamma_\infty$, as in the statement of Prop.\ \ref{prop_height_of_ideal_in_P}. The ring $R_{b, \infty} / (Q_b, \mathfrak{m}_{P_\infty})$ is not Artinian as, by construction, it has $R_{b, \infty} / \mathfrak{p}_b$ as a quotient. We can therefore apply Prop.\ \ref{prop_height_of_ideal_in_P} to deduce that $\operatorname{ht}(I_c) \leq 1$, and therefore that
    \[ \dim P_\infty / (Q, \varpi, I_c) \geq [F : \mathbb{Q}] + |T| - h + 1. \]
    Let $\mathfrak{q}$ be a dimension 1, characteristic $p$ prime ideal of $P_\infty$ containing $(Q, \varpi, I_c)$ such that $\mathcal{R}_\infty \not\leq \mathfrak{q}$ and (if we are in the dihedral case) $\mathcal{D}_\infty \not\leq \mathfrak{q}$. Such an ideal exists 
    by~\cite[Lem.\ 1.9]{Tho15},
    since we have 
    \[ \dim P_\infty / (\mathcal{R}_\infty, \varpi) \leq 1 + \delta_F < \dim P_\infty / (Q, \varpi, I_c) \]
    and, if we are in the dihedral case,
    \[ \dim P_\infty / (\mathcal{D}_\infty, \varpi) \leq [ F : \mathbb{Q}] < \dim P_\infty / (Q, \varpi, I_c). \]
    We are going to show that~$\mathfrak q$ is a nice prime ideal arising from $R_{b_{\rho}, \infty}$.
    Part~(1) of Def.\ \ref{defn:nice prime ideal} holds by construction (using that $\mathcal M_\infty \leq Q$), so it remains to check part~(2).
    
    Let $A \cong k'\llbracket Y\rrbracket $ be the normalisation of $P_\infty / \mathfrak{q}$ in its fraction field, and let $\rho_{\mathfrak{q}} : \Gamma_\infty \to \GL_2(A)$ be a representation with pseudocharacter $t_\infty \text{ mod }\mathfrak{q}$, chosen (as described in \S \ref{sec_pseudocharacters}) to have $\rho_{\mathfrak{q}}(c_\infty) = \operatorname{diag}(-1, 1)$ and $\rho_{\mathfrak{q}} \text{ mod }Y$ upper triangular but \emph{not} diagonal. In particular, this forces the extension class appearing in the reducible residual representation $\rho_{\mathfrak{q}} \text{ mod }Y$ to be non-trivial. 

    The defining relations of the rings $P^{\textnormal{ord}}_{Q}$, and therefore of $P_\infty$, show that if $v \in T$, and $\sigma, \tau \in G_{F_v}$, then we have the equation
    \[ (\rho_{\mathfrak{q}}(\sigma) - 1)(\rho_{\mathfrak{q}}(\tau) - \epsilon(\tau)^{-1}) = 0. \]
    In particular, the cohomology class defining the extension $\rho_{\mathfrak{q}} \text{ mod }Y$ must become split after restriction to any place in $T$. %
    
    Since $\rho_{\mathfrak{q}} \otimes \operatorname{Frac}(A)$ is irreducible, we can find $\tau_0 \in \Gamma_\infty$ such that $\rho_{\mathfrak{q}}(\tau_0)_{21}\neq 0$. By construction, we have $x(z_j, \tau_0) \equiv 0 \text{ mod }\mathfrak{q}$ for $j = 1$ and for each $j = 3, \dots, h$. 
    Since $x(z_j, \tau_0) = \rho_{\mathfrak{q}}(z_j)_{12}\rho_{\mathfrak{q}}(\tau_0)_{21}$ in~$\operatorname{Frac}(A)$,
    it follows that $\rho_{\mathfrak{q}}(z_j)_{12} = 0$ for each such value of $j$. Reducing modulo $Y$, it follows that the extension class appearing in $\rho_{\mathfrak{q}}$ is a scalar multiple of $b_\rho$. It follows immediately that $\mathfrak{q}$ is a nice prime of $P_\infty$ that arises from $R_{b_\rho, \infty}$. This concludes the proof.
\end{proof}

\subsection{Second patching argument}\label{subsec_second_patching}

We now turn to the proof of Prop.\ \ref{prop_local_R_equals_T}, which rests upon the introduction of further Taylor--Wiles places.
\begin{pp}\label{prop_existence_of_TW_places}
    Let $m \geq 1$, let $A \coloneqq  k\llbracket X\rrbracket $, let $Q$ be a Taylor--Wiles set, let $T_a$ be a finite set of finite places disjoint from $S_p \cup S_{\textnormal{aux}} \cup T \cup Q$, and suppose given a homomorphism $f : R_{b, Q}^{\textnormal{ord},\square_{S_p \cup S_{\textnormal{aux}} \cup Q \cup T_a}} \to A / (X^m)$, corresponding to a lifting $\rho_m$ of $\overline{\rho}_b$, such that the composite 
    \[ R_{\textnormal{loc}}^{\square_{S_p \cup S_{\textnormal{aux}}\cup Q \cup T_a}} \to R_{b, Q}^{\textnormal{ord},\square_{S_p \cup S_{\textnormal{aux}}\cup Q \cup T_a}} \to  A / (X^m) \]
    surjects onto the image $A'$ of $f$. 
    Suppose given $C_0, C_1 \geq 0$ such that 
    \[
    X^{C_0} \left (A / (X^m) \right)\leq A' 
    \]
    and 
    \[
    X^{C_1} \operatorname{End}_A( \ad^0 \rho_m ) \leq (\ad^0 \rho_m)(A[G_{F(\zeta_{p^\infty})}]). 
    \]
    Let $C \coloneqq  \max(C_0, C_1)$. Then, for any
    \[
    q_2 \geq h^1(F_{S_p \cup S_{\textnormal{aux}} \cup T \cup Q \cup T_a}/F, \ad^0 \rhobar_b(1)) 
    \]
    and $N \geq 1$, we can find a set $Q'$ of Taylor--Wiles places of level $N$, with $q_2$ elements and disjoint from $S_p \cup S_{\textnormal{aux}} \cup T \cup Q \cup T_a$, and with the following property: 
    \begin{itemize}
    \item Let $\widetilde{Q} \coloneqq  Q \cup Q'$, and let $\widetilde{I}$ denote the kernel of the induced homomorphism $R_{b, \widetilde{Q}}^{\textnormal{ord},\square_{S_p \cup S_{\textnormal{aux}}\cup Q \cup T_a}} \to A'$, $I_{\textnormal{loc}}$ the kernel of the induced homomorphism $R^{\square_{S_p \cup S_{\textnormal{aux}} \cup Q \cup T_a}}_{\textnormal{loc}} \to A'$. Then there exists a homomorphism
    \[ (A')^g \to \widetilde{I} / (\widetilde{I}^2, I_{\textnormal{loc}}) \]
    of $A'$-modules, with cokernel annihilated by $X^{2 C}$, where 
    \[ g \coloneqq  q_2 - [ F: \mathbb{Q}] - 1 + |S_p \cup S_{\textnormal{aux}} \cup Q \cup T_a|. \]
    \end{itemize}
\end{pp}
\begin{proof}
    We can assume that $m > C$, as otherwise the conclusion is vacuously true because $X^{2C}\widetilde{I} / (\widetilde{I}^2, I_{\textnormal{loc}}) = 0$. 
     Let $Q'$ be an additional set of Taylor--Wiles places, to be specified. If $0 \leq n \leq m$, let 
     \[
     M_n \coloneqq  \Hom_{A'}(\widetilde{I} / (\widetilde{I}^2, I_{\textnormal{loc}}), A / (X^n)), 
     \]
     and let
     \[
     \rho_n : G_{F, S_p \cup S_{\textnormal{aux}} \cup T \cup \widetilde{Q}} \to \GL_2(A / (X^n))
     \]
     be the lifting associated to $f \text{ mod }X^n$, and $W_n \coloneqq  \ad^0 \rho_n$ be the adjoint module of~$\rho_n$. Applying \cite[Proposition 4.10]{All23a} shows that there is a $A$-module isomorphism 
     \[ M_n \cong h^1_{\mathcal{S}, S_p \cup S_{\textnormal{aux}} \cup Q \cup T_a}(\ad \rho_n ), \]
     where $\mathcal{S}$ is the deformation problem, in the sense of \emph{op.\ cit.}, determined by our choice of local lifting rings. 
     Hence, a standard calculation in Galois cohomology, using the Greenberg--Wiles Euler characteristic formula (cf. \cite[Lem.\ 4.3.2]{Pan22}), shows that  the length $\ell_A(M_n)$ of $M_n$ as an $A$-module is equal to 
     \[ h^1_{Q'\textnormal{-split}, \mathcal{L}_{n, T}^\perp}(F_{S_p \cup S_{\textnormal{aux}} \cup T \cup T_a \cup Q} / F, W_n(1)) + n(|Q'| - [ F: \mathbb{Q}] - 1 + |S_p \cup S_{\textnormal{aux}}\cup Q \cup T_a|), \]
     where $\mathcal{L}_{n, T}^\perp = \{ \mathcal{L}_{n, v}^\perp \}_{v \in T}$ is the collection of dual local conditions, to the conditions $\mathcal{L}_{n, v} \leq H^1(F_v, \ad \rho_n)$ determined by the local lifting ring $R_v^{\operatorname{St}, 1}$ ($v \in T$). (These local conditions will reappear in \S \ref{sec: construction of T-new lifts} below.) This computation uses that $H^0(F, W_1) = H^0(F, W_1(1)) =  0$.
     
     We note that there is a natural isomorphism
     \begin{multline*}  H^1_{Q'\textnormal{-split},  \mathcal{L}_{n, T}^\perp}(F_{S_p\cup S_{\textnormal{aux}}\cup T \cup T_a \cup Q} / F, W_n(1)) \\ \cong  H^1_{Q'\textnormal{-split},  \mathcal{L}_{m, T}^\perp}(F_{S_p \cup S_{\textnormal{aux}} \cup T \cup T_a \cup Q} / F, W_m(1))[X^n]. 
     \end{multline*}
     Suppose therefore that we can choose $Q'$ so that $|Q'| = q_2$ and \[ X^{C_1} H^1_{Q'\textnormal{-split}}(F_{S_p \cup S_{\textnormal{aux}} \cup T \cup T_a \cup Q} / F, W_m(1)) = 0. \]
     We find that $\length_A(M_{n+1}) = \length_A(M_{n}) + g$ for each $n = C_1, \dots, m-1$. Since in  any case there is an isomorphism $M_{n} = M_{n+1}[X^n]$, this implies that there is an isomorphism $M_m \cong (A / (X^m))^g \oplus D$, where $D$ is an $A / (X^m)$-module killed by $X^{C_1}$. Since $\Hom_k(A / (X^m), k) \cong A / (X^m)$, it follows that $\widetilde{I} / (\widetilde{I}^2, I_{\textnormal{loc}}) \otimes_{A'} A / (X^m) \cong (A / X^m)^g \oplus D'$, where $X^C D' = 0$. The map $(A')^g \to (A / X^m)^g \oplus D'$ induced by inclusion in the first factor is injective, with cokernel annihilated by $X^{C}$. The map 
     \[ \widetilde{I} / (\widetilde{I}^2, I_{\textnormal{loc}}) \otimes_{A'} A / (X^m) \to \widetilde{I} / (\widetilde{I}^2, I_{\textnormal{loc}}) \]
     induced by $x \otimes a \mapsto (X^C a) x$ has kernel and  cokernel annihilated by $X^C$. Taking the composite of these two maps, it follows that there is a map $(A')^g \to \widetilde{I} / (\widetilde{I}^2, I_{\textnormal{loc}})$ with kernel and cokernel annihilated by $X^{2C}$.

     To complete the proof of the proposition, it is therefore enough to show that for any $q_2 \geq h^1(F_{S_p\cup S_{\textnormal{aux}} \cup T\cup Q \cup T_a}/F, W_1(1))$, we can choose $Q'$ of cardinality $q_2$ so that $H^1_{Q'\textnormal{-split}}(F_{S_p \cup S_{\textnormal{aux}} \cup T \cup T_a \cup Q} / F, W_m(1))$ is annihilated by $X^{C_1}$. By \cite[Lem.\ 6.5]{Kis04}, it is enough to show that the kernel of the map
     \[ H^1(F_{S_p \cup S_{\textnormal{aux}} \cup T \cup Q \cup T_a} / F, W_m(1)) \to \prod_{v} H^1(F_v, W_m(1)) \]
     is annihilated by $X^{ C_1}$, where the product ranges over the set of all Taylor--Wiles places of level $N$. Let $[\phi]$ be a class in the kernel, and suppose for contradiction that $X^{ C_1} [\phi] \neq 0$. Let $L / F$ be the extension of $F(\zeta_{p^N})$ cut out by $\rho_m$. Then $H^1(L / F, W_m(1)) = 0$ (cf. Lem.\ \ref{lem:definition of M}), so $X^{C_1} \operatorname{Res}_{L / F} \phi \neq 0$, where now $\operatorname{Res}_{L / F} \phi \in \Hom(G_L, W_m(1))^{G_F}$. It follows from the definition of $C_1$ that $A \cdot \phi(G_L)$ contains $X^{m-1} W_m(1)$. (Indeed, the image of the map $A[G_F] \to \operatorname{End}_A(W_m(1))$ contains $X^{C_1} \operatorname{End}_A(W_m(1))$, so for any $z \in W_m(1)$ such that $X^{C_1} z \neq 0$, $X^{m-1} W_m(1) \leq A[G_F] \cdot z$.)

     Choose any $\sigma \in G_{F(\zeta_{p^\infty})}$ such that $\overline{\chi}(\sigma) \neq 1$. If $\phi(\sigma) \not\in (\sigma-1) W_m(1)$, then we can choose, by Chebotarev density theorem, a Taylor--Wiles place $v$ of level $N$ such that $\rho_m(\Frob_v) = \rho_m(\sigma)$ and $\phi(\Frob_v) \not\in (\Frob_v-1)W_m(1)$. In particular, the image of $[\phi]$ in $H^1(F_v, W_m(1))$ is non-zero, contradicting the definition of $[\phi]$. Therefore $\phi(\sigma) \in (\sigma-1) W_m(1)$.

     In this case, we can choose $\tau \in G_L$ such that $\phi(\tau) \not\in (\sigma-1) W_m(1)$ (as otherwise it would not be the case that $A \cdot \phi(G_L)$ contains $X^{m-1} W_m(1)$). We can then choose a Taylor--Wiles place $v$ of level $N$ such that $\rho_m(\Frob_v) = \rho_m(\tau \sigma) = \rho_m(\sigma)$ and $\phi(\Frob_v) = \phi(\tau \sigma) = \phi(\tau) + \phi(\sigma) \not\in (\Frob_v-1) W_m(1)$. It follows that the image of $[\phi] \in H^1(F_v, W_m(1))$ is again non-zero. This contradiction concludes the proof. 
\end{proof}
Coming back to the situation of Prop.\ \ref{prop_local_R_equals_T}, recall that we are given a normalised cocycle $b$ representing a non-zero class in $H^1_{T \textnormal{-split}}(F_{S_p\cup S_{\textnormal{aux}}} / F, \overline{\chi})$, and a nice prime $\mathfrak{q} \leq \mathbb{T}_\infty$ that arises from $R_{b, \infty}$, and we must show that the natural map $\widehat{(P_{\infty, \mathfrak{q}^{\textnormal{ps}}})} \to \widehat{(\mathbb{T}_{\infty, \mathfrak{q}})}$ has nilpotent kernel. Let $A$ denote the normalisation of $\mathbb{T}_\infty / \mathfrak{q}$ in its field of fractions. By hypothesis, the map $P_\infty \to A$ extends to a homomorphism $R_{b, \infty} \to A$, corresponding to the equivalence class of the lifting $r_{b, \infty} : \Gamma_\infty \to \GL_2(A)$ of $\overline{\rho}_b$. We also write $\mathfrak{q}_b \coloneqq  \ker(R_{b, \infty} \to A)$ and $B \coloneqq  R_{b, \infty} / \mathfrak{q}_b \leq A$, and we keep the notation~$Q_M$ for the Taylor--Wiles sets constructed in Section~\ref{subsec:first patching argument}.

After possibly enlarging $\cO$, we can assume that the residue field of $A$ equals $k$, and then fix an isomorphism $A \cong k\llbracket X\rrbracket $. (More precisely: let $k' / k$ be the residue field of $A$, and let $\cO' / \cO$ be the unramified extension of $\cO$ with residue field $k'$. Let $\mathfrak{q}'_b$ be a prime of $R_{b, \infty} \otimes_{\cO} \cO'$ lying above $\mathfrak{q}_b$, and let $(\mathfrak{q}')^{\textnormal{ps}}$ be its pullback to $P_\infty \otimes_\cO \cO'$. Then the normalisation of $(P_\infty \otimes_\cO \cO') / (\mathfrak{q}')^{\textnormal{ps}}$ in its fraction field may be identified with $A$, and the natural map $P_{\infty, \mathfrak{q}^{\textnormal{ps}}} \to (P_\infty \otimes_\cO \cO')_{(\mathfrak{q}')^{\textnormal{ps}}}$ is faithfully flat. Therefore, if we show that the ideal $\mathcal{M}_\infty (P_\infty \otimes_\cO \cO')_{(\mathfrak{q}')^{\textnormal{ps}}}$ is nilpotent, it will follow that the ideal $\mathcal{M}_\infty P_{\infty, \mathfrak{q}^{\textnormal{ps}}}$ is also nilpotent. This reduction behaves well with respect to patching because e.g.\ the natural map
\[ \left(\varprojlim_N P_\infty / \mathfrak{m}_{P_\infty}^{r_N} \right) \otimes_\cO \cO' \to \left(\varprojlim_N P_\infty / \mathfrak{m}_{P_\infty}^{r_N}\otimes_\cO \cO'  \right) \]
is an isomorphism. Base extension by $k' / k$ does introduce new directions in the group $H^1(F_{S_p \cup S_{\textnormal{aux}}}, k(\overline{\chi}))$, but this is not a problem here as we are now using a fixed $b$.)

\begin{lemma}\label{lem_good_unramified_places}
    There exists $d_1 \geq 0$ with the following property: for any $M \geq 1$, there are places $v_1, \dots, v_{d_1} \not\in S_p \cup S_{\textnormal{aux}} \cup T\cup Q_M$ such that $R_{b, Q_M}^{\textnormal{ord}}$ may be topologically generated as a $\widetilde{\Lambda}$-algebra by the entries of $\rho_b^{\textnormal{univ}}(\Frob_{v_i})$ ($i = 1, \dots, d_1$), for any representative of the universal deformation.  
\end{lemma}
\begin{proof}
    This is a consequence of the Chebotarev density theorem, together with the existence of a uniform bound on the dimension of the Zariski tangent space of $R_{b,  Q_M}^{\textnormal{ord}}$ (which was proved in the course of proving Lem.\ \ref{lem_finiteness_of_patching_data}~(2)). 
\end{proof}
Let 
\[
q_2 \coloneqq  h^1(F_{S_p \cup S_{\textnormal{aux}} \cup T}/F, \ad^0 \rhobar_b(1)) + |T| + 3 d_1, 
\]
and let 
\[
\widetilde{S}_\infty \coloneqq  \cO \llbracket  \mathbb{Z}_p^{|T|} \oplus \mathbb{Z}_p^{q_2}\rrbracket . 
\]
There is a natural homomorphism $\widetilde{S}_\infty \to S_\infty$. 
We write 
\[
\widetilde{\mathfrak{a}}_N \coloneqq  \ker(\widetilde{S}_\infty \to \cO[(\mathbb{Z}_p / p^N)^{|T| + q_2}]) 
\]
and 
\[
\widetilde{\mathfrak{b}}_N \coloneqq  (\mathfrak{m}_\Lambda^N, \widetilde{\mathfrak{a}}_N) \leq \Lambda \widehat{\otimes}_\cO \widetilde{S}_\infty.
\]
Then $\widetilde{\mathfrak{b}}_N$ is an open ideal of $\Lambda \widehat{\otimes}_\cO \widetilde{S}_\infty$, and its image in $\Lambda \widehat{\otimes}_\cO S_\infty$ is equal to $\mathfrak{b}_N$. We set $\widetilde{r}_N \coloneqq  N d_0 (|T| + q_2) p^N$. 

We now define, for each $N \geq 1$, a groupoid $\widetilde{\operatorname{Patch}}_N$ of ``augmented patching data of level $N$''. Such a datum consists of the following objects (here we continue to use the notation~$M_N$ to denote the integers produced in the course of the first patching argument):
\begin{itemize}
    \item A complete Noetherian local $\Lambda \widehat{\otimes}_\cO \widetilde{S}_\infty$-algebra $\widetilde{\mathbf{P}}_N$, together with a homomorphism $\widetilde{\mathbf{P}}_N \to P^{\textnormal{ord}}_{Q_{M_N}} / \mathfrak{m}_{P^{\textnormal{ord}}_{Q_{M_N}}}^{r_N}$ of $\Lambda \widehat{\otimes}_\cO \widetilde{S}_\infty$-algebras.
    \item A complete Noetherian local $\widetilde{\mathbf{P}}_N$-algebra $\widetilde{\mathbf{R}}_{b, N}$, together with a homomorphism $\widetilde{\mathbf{R}}_{b, N} \to R^{\textnormal{ord}}_{b, Q_{M_N}} / \mathfrak{m}_{R^{\textnormal{ord}}_{b, Q_{M_N}}}^{r_N}$, making the diagram
    \[ \xymatrix{ \widetilde{\mathbf{P}}_N \ar[r] \ar[d] & P^{\textnormal{ord}}_{ Q_{M_N}} / \mathfrak{m}_{P^{\textnormal{ord}}_{ Q_{M_N}}}^{r_N} \ar[d] \\
    \widetilde{\mathbf{R}}_{b, N} \ar[r] & R^{\textnormal{ord}}_{b,  Q_{M_N}} / \mathfrak{m}_{R^{\textnormal{ord}}_{b, Q_{M_N}}}^{r_N}} \]
    commute.
    \item A $\widetilde{\mathbf{P}}_N$-module $\widetilde{\mathbf{H}}_N$, finite free as $\Lambda \widehat{\otimes}_\cO \widetilde{S}_\infty / \widetilde{\mathfrak{b}}_N$-module, together with an isomorphism $\widetilde{\mathbf{H}}_N \otimes_{\widetilde{S}_\infty} S_\infty \cong H_{Q_{M_N}(N)} / (\mathfrak{b}_N)$ of $\widetilde{\mathbf{P}}_N$-modules (where $\widetilde{\mathbf{P}}_N$ acts on $H_{Q_{M_N}(N)} / (\mathfrak{b}_N)$ via the given homomorphism $\widetilde{\mathbf{P}}_N \to P^{\textnormal{ord}}_{ Q_{M_N}} / \mathfrak{m}_{P^{\textnormal{ord}}_{ Q_{M_N}}}^{r_N}$).
    \item A homomorphism in $\CNL_\cO$
       \begin{multline*} R_{\textnormal{loc}}^{\square_{S_p \cup S_{\textnormal{aux}}}} \widehat{\otimes}_\cO S_\infty\llbracket x_1, \ldots, x_{q_2 + 4 |T| + 4d_1 + |S_p \cup S_{\textnormal{aux}}| - [F : \mathbb{Q}] - 1}\rrbracket  / (\mathfrak{m}_{R_{\textnormal{loc}}^{\square_{S_p \cup S_{\textnormal{aux}}}} \widehat{\otimes}_\cO S_\infty\llbracket x_1, \ldots x_{q_2 + 4 |T| + 4d_1 + |S_p \cup S_{\textnormal{aux}}| - [F : \mathbb{Q}] - 1}\rrbracket }^{\widetilde{r}_N}) 
    \\ \to \widetilde{\mathbf{R}}_{b, N} \llbracket  y_1, \ldots, y_{4(|S_p \cup S_{\textnormal{aux}} \cup T| + d_1) - 1} \rrbracket  / \mathfrak{m}_{\widetilde{\mathbf{R}}_{b, N} \llbracket  y_1, \ldots, y_{4(|S_p \cup S_{\textnormal{aux}} \cup T| +d_1) - 1} \rrbracket }^{\widetilde{r}_N} 
    \end{multline*}
    \item A finite group $\widetilde{\Gamma}_N$, together with a pseudorepresentation $\widetilde{\mathbf{t}}_N : \widetilde{\Gamma}_N \to \widetilde{\mathbf{P}}_N$ and a representation $\widetilde{\mathbf{r}}_{b, N} : \widetilde{\Gamma}_N \to \GL_2(\widetilde{\mathbf{R}}_{b, N})$.
    \item A homomorphism $\widetilde{\Gamma}_N \to G_{F, S_p \cup T \cup S_{\textnormal{aux}} \cup Q_{M_N}} / \mathcal{D}^{3 r_N} G_{F, S_p \cup T \cup S_{\textnormal{aux}} \cup Q_{M_N}} $ such that the pushforwards of $\widetilde{\mathbf{t}}_N$ (resp.\ $\widetilde{\mathbf{r}}_{b, N}$) to $P^{\textnormal{ord}}_{ Q_{M_N}} / \mathfrak{m}_{P^{\textnormal{ord}}_{Q_{M_N}}}^{r_N}$ (resp.\ $R^{\textnormal{ord}}_{b, Q_{M_N}} / \mathfrak{m}_{R^{\textnormal{ord}}_{b, Q_{M_N}}}^{r_N}$) agree with the pseudocharacter (resp.\ representations) of $\widetilde{\Gamma}_N$ pulled back from the universal pseudocharacter (resp.\ representations) over these latter rings. 
    \item An element $\widetilde{\mathbf{c}}_u \in \widetilde{\Gamma}_N$ lifting the image of $c_u$ in $G_{F, S_p \cup T \cup S_{\textnormal{aux}} \cup Q_{M_N}} / \mathcal{D}^{3 r_N} G_{F, S_p \cup T \cup S_{\textnormal{aux}} \cup Q_{M_N}}$.
    \item For each $v \in S_p \cup T \cup S_{\textnormal{aux}}$, a homomorphism $\widetilde{\gamma}_v : G_{F_v} \to \widetilde{\Gamma}_N$ lifting the tautological homomorphism $G_{F_v} \to G_{F, S_p \cup T \cup S_{\textnormal{aux}} \cup Q_{M_N}} / \mathcal{D}^{3 r_N} G_{F, S_p \cup T \cup S_{\textnormal{aux}} \cup Q_{M_N}}$. 
\end{itemize}
The isomorphisms of augmented patching data are again given by isomorphisms of the relevant objects making all relevant squares commute. 

We may define, for each $N \geq 1$, a functor $\widetilde{\Pi}_{N+1} : \widetilde{\operatorname{Patch}}_{N+1} \to \widetilde{\operatorname{Patch}}_{N}$, given by ``dividing out by $\mathfrak{m}^{\widetilde{r}_N}$ and $\widetilde{\mathfrak{b}}_N$''. The definition of this functor uses the isomorphisms $\Pi_{N+1}(D(N+1, M_{N+1})) \cong D(N, M_N)$ in $\operatorname{Patch}_N$ already constructed as part of the first patching argument. 

For each $N \geq 1$, we apply Lem.\ \ref{lem_good_unramified_places} to $Q_{M_N}$ to obtain an auxiliary set of unramified places $T_N$ with $|T_N| = d_1$. By construction, we are given an isomorphism $R_{b, \infty} / \mathfrak{m}_{R_{b, \infty}}^{r_N} \cong R^{\textnormal{ord}}_{b,  Q_{M_N}} / \mathfrak{m}_{R^{\textnormal{ord}}_{b,  Q_{M_N}}}^{r_N}$, under which $r_{b, \infty}$ specialises to a specialisation of the fixed choice of universal lifting. In particular, the map $R_{b, \infty} / \mathfrak{m}_{R_{b, \infty}}^{r_N} \to A / (X^{r_N})$ gives us a map
\[ R^{\textnormal{ord}}_{b, Q_{M_N}} / \mathfrak{m}_{R^{\textnormal{ord}}_{b, Q_{M_N}}}^{r_N} \to A / (X^{r_N}), \]
which (with the given data) naturally extends to a map
\[ R^{\textnormal{ord}, \square_{S_p\cup S_{\textnormal{aux}} \cup Q_{M_N} \cup T_N}}_{b, Q_{M_N}} \to A / (X^{r_N}) \]
with the same image $A'$ (namely the image of $B \to A / (X^{r_N})$), and such that the induced map
\[ R^{\square_{S_p \cup S_{\textnormal{aux}}\cup Q_{M_N} \cup T_N}}_{\textnormal{loc}} \to A / (X^{r_N}) \]
has the same image (using the defining property of the set $T_N$, given in Lem.\ \ref{lem_good_unramified_places}). 

Fix $C_0 \geq 0$ such that $X^{C_0} A \leq B$. By Lem.\ \ref{lem_zariski_density}, the specialisation of $r_{b, \infty}$ to $k(\mathfrak{q}_b)$ acts absolutely irreducibly on the adjoint representation, even after restriction to $\ker(\Gamma_\infty \to \Gal(F(\zeta_{p^\infty}) / F))$. We can thus find $C_1 \geq 0$ such that 
\[ X^{C_1} \operatorname{End}(\ad^0 r_{b, \infty} \otimes_{R_{b, \infty}} A) \leq A \cdot (\ad^0 r_{b, \infty} \text{ mod }\mathfrak{q}_b)(\ker(\Gamma_\infty \to \Gal(F(\zeta_{p^\infty}) / F))). \]
We can then apply Prop.\ \ref{prop_existence_of_TW_places} for any $N \geq 1$, with $m \coloneqq r_N$ and with these values of $C_0, C_1$.
We obtain an additional set $Q'_N$ of Taylor--Wiles places of level $N + m_0$, with $q_2$ elements, and such that there is a homomorphism $(A')^g \to \widetilde{I} / (\widetilde{I}^2, I_{\textnormal{loc}})$ with cokernel annihilated by $X^{2C}$, where $g = q_2 - [F : \mathbb{Q}] - 1 + |S_p \cup S_{\textnormal{aux}}\cup T \cup T_N|$, and where $A', \widetilde{I}$, $I_{\textnormal{loc}}$ are as defined in the statement of Prop.\ \ref{prop_existence_of_TW_places}, i.e.\ 
\[
\widetilde I \coloneqq \ker \left ( R_{b, Q_{M_N}}^{\textnormal{ord}, \square_{S_p \cup S_{\aux}\cup Q_{M_N}\cup T_N}} \to A/ (X^{r_N}) \right )
\]
and~$I_{\loc}$ is the preimage of~$\widetilde I$ in $R^{\square_{S_p \cup S_{\textnormal{aux}}\cup Q_{M_N} \cup T_N}}_{\textnormal{loc}}$.
We set $\widetilde{Q}_N \coloneqq  Q_{M_N} \cup Q'_N$, and choose arbitrary lifts $\widetilde{x}_{1, N}, \dots, \widetilde{x}_{g, N} \in \widetilde{I}$ of the images of the basis vectors of $(A')^g$. 

We now define, for each $1 \leq N \leq \widetilde{N}$, an augmented patching datum $\widetilde{D}(N, \widetilde{N}) \in \widetilde{\operatorname{Patch}}_N$, as follows:
\begin{itemize}
    \item $\widetilde{\mathbf{P}}_N \coloneqq P^{\textnormal{ord}}_{\widetilde{Q}_{\widetilde{N}}} / \mathfrak{m}_{P^{\textnormal{ord}}_{\widetilde{Q}_{\widetilde{N}}}}^{\widetilde{r}_N}$. The map $\widetilde{\mathbf{P}}_N \to P^{\textnormal{ord}}_{ Q_{M_N}} / \mathfrak{m}_{P^{\textnormal{ord}}_{ Q_{M_N}}}^{r_N}$ is the composite of the natural map
    \[ P^{\textnormal{ord}}_{\widetilde{Q}_{\widetilde{N}}} / \mathfrak{m}_{P^{\textnormal{ord}}_{\widetilde{Q}_{\widetilde{N}}}}^{\widetilde{r}_N} \to P^{\textnormal{ord}}_{ Q_{M_{\widetilde{N}}}} / \mathfrak{m}_{P^{\textnormal{ord}}_{ Q_{M_{\widetilde{N}}}}}^{r_N} \]
    with the isomorphism
    \[P^{\textnormal{ord}}_{ Q_{M_{\widetilde{N}}}} / \mathfrak{m}_{P^{\textnormal{ord}}_{Q_{M_{\widetilde{N}}}}}^{r_N} \cong P^{\textnormal{ord}}_{ Q_{M_{N}}} / \mathfrak{m}_{P^{\textnormal{ord}}_{ Q_{M_{N}}}}^{r_N}  \]
    constructed as part of the first patching argument. 
    \item $\widetilde{\mathbf{R}}_{b, N} \coloneqq R^{\textnormal{ord}}_{b, \widetilde{Q}_{\widetilde{N}}} / \mathfrak{m}_{R^{\textnormal{ord}}_{b, \widetilde{Q}_{\widetilde{N}}}}^{\widetilde{r}_N}$, with map constructed similarly. 
    \item $\widetilde{\mathbf{H}}_N \coloneqq H_{\widetilde{Q}_{\widetilde{N}}(N)} / (\widetilde{\mathfrak{b}}_N)$, with map constructed similarly. 
    \item To write down the homomorphism 
    \begin{multline*}
    R^{\square_{S_p\cup S_{\textnormal{aux}}}}_{\textnormal{loc}} \widehat{\otimes}_\cO S_\infty \llbracket  x_1, \ldots, x_{q_2+4|T|+4d_1+|S_p \cup S_{\aux}|-[F:\mathbb Q]-1} \rrbracket/\mathfrak m^{\widetilde r_N}  
    \to \widetilde{\mathbf R}_{b, N}\llbracket y_1, \ldots, y_{4(|S_p \cup S_{\aux} \cup T|+d_1)-1} \rrbracket  / \mathfrak m^{\widetilde r_N},
    \end{multline*}
    we first take the natural homomorphism
    \[ R^{\square_{S_p \cup S_{\textnormal{aux}}\cup Q_{M_{\widetilde{N}}} \cup T_{\widetilde{N}}}}_{\textnormal{loc}} \to R^{\textnormal{ord},\square_{S_p \cup S_{\textnormal{aux}} \cup Q_{M_{\widetilde{N}}} \cup T_{\widetilde{N}}}}_{b, \widetilde{Q}_{\widetilde{N}}}. \]
    Using the choice of a surjection $S_\infty \llbracket x_1, \dots, x_{3 |T|}\rrbracket  \to \widehat{\otimes}_{v \in Q_{M_{\widetilde{N}}}} R_{b, v}^\square$ already discussed during the course of the first patching argument, 
    and a choice of isomorphism $\cO\llbracket x_1, \dots, x_{3d_1}\rrbracket  \cong \widehat{\otimes}_{v \in T_{\widetilde{N}}} R_{b, v}^{\textnormal{ur}}$, this becomes a map
    \[ R^{\square_{S_p\cup S_{\textnormal{aux}}}}_{\textnormal{loc}} \widehat{\otimes}_\cO S_\infty \llbracket x_1, \dots, x_{3(|T| + d_1)}\rrbracket  \to R^{\textnormal{ord},\square_{S_p \cup S_{\textnormal{aux}}\cup Q_{M_{\widetilde{N}}} \cup T_{\widetilde{N}}}}_{b, \widetilde{Q}_{\widetilde{N}}} \]
    that we extend to a map from 
    \[ R^{\square_{S_p\cup S_{\textnormal{aux}}}}_{\textnormal{loc}} \widehat{\otimes}_\cO S_\infty \llbracket x_1, \dots, x_{3(|T| + d_1)+q_2 - [F : \mathbb{Q}]-1+|S_p \cup S_{\textnormal{aux}}| + |T| + d_1}\rrbracket  \]
    by sending the last few power series variables to the chosen elements 
    \[ \widetilde{x}_{1, \widetilde{N}}, \dots, \widetilde{x}_{q_2 - [F : \mathbb{Q}]-1+|S_p \cup S_{\textnormal{aux}}| + |T| + d_1, \widetilde{N}}. \]
    Finally, we identify the target $R^{\textnormal{ord},\square_{S_p \cup S_{\textnormal{aux}}\cup Q_{M_{\widetilde{N}}} \cup T_{\widetilde{N}}}}_{b, \widetilde{Q}_{\widetilde{N}}}$ of this map with a power series ring over %
    $R^{\textnormal{ord}}_{b, \widetilde{Q}_{\widetilde{N}}}$
    by using a choice of representative of the universal deformation in the usual way, and finally we pass to the quotient by the $\widetilde{r}_N$-th power of the maximal ideal. 
    \item $\widetilde{\Gamma}_N \coloneqq G_{F, S_p \cup T \cup S_{\textnormal{aux}} \cup \widetilde{Q}_{\widetilde{N}}} / \mathcal{D}^{3 \widetilde{r}_N} G_{F, S_p \cup T \cup S_{\textnormal{aux}} \cup \widetilde{Q}_{\widetilde{N}}}$, and the additional data $\widetilde{\mathbf{t}}_N$, $\widetilde{\mathbf{r}}_{b, N}$, $\widetilde{\mathbf{c}}_u$, and $\widetilde{\gamma}_v$ ($v \in S_p \cup T \cup S_{\textnormal{aux}}$) comes via specialisation of the universal (or tautological) data over $P^{\textnormal{ord}}_{\widetilde{Q}_{\widetilde{N}}}$ and $R^{\textnormal{ord}}_{b, \widetilde{Q}_{\widetilde{N}}}$.
\end{itemize}
The analogue of Lem.\ \ref{lem_finiteness_of_patching_data} holds for these augmented patching data. 
Therefore we can find $1 \leq \widetilde{N}_1 < \widetilde{N}_2 < \widetilde{N}_3 < \dots$ and, for each $N \geq 1$, an isomorphism $\widetilde{\Pi}_{N+1}(\widetilde{D}(N+1, \widetilde{N}_{N+1})) \cong \widetilde{D}(N, \widetilde{N}_N)$ in $\widetilde{\operatorname{Patch}}_N$. We thus obtain inverse systems of rings $P^{\textnormal{ord}}_{ \widetilde{Q}_{\widetilde{N}_N}} / \mathfrak{m}_{P^{\textnormal{ord}}_{ \widetilde{Q}_{\widetilde{N}_N}}}^{\widetilde{r}_{N}}$ and modules $\widetilde{H}_{\widetilde{Q}_{\widetilde{N}_N}(N)} / (\widetilde{\mathfrak{b}}_{N})$. Passing to the inverse limit, we obtain the following objects:
\begin{itemize}
    \item A complete Noetherian local $\Lambda \widehat{\otimes}_\cO \widetilde{S}_\infty$-algebra $\widetilde{P}_\infty$, together with a homomorphism $\widetilde{P}_\infty \to P_\infty$ of $\Lambda \widehat{\otimes}_\cO \widetilde{S}_\infty$-algebras. 
    \item A complete Noetherian local $\widetilde{P}_\infty$-algebra $\widetilde{R}_{b, \infty}$, together with a homomorphism $\widetilde{R}_{b, \infty} \to R_{b, \infty}$ making the diagram
    \[ \xymatrix{ \widetilde{P}_\infty \ar[r] \ar[d] & P_\infty \ar[d] \\
    \widetilde{R}_{b, \infty} \ar[r] & R_{b, \infty}} \]
    commute.
    \item A $\widetilde{P}_\infty$-module $\widetilde{H}_\infty$, finite free as $\Lambda \widehat{\otimes}_\cO \widetilde{S}_\infty$-module, together with an isomorphism $\widetilde{H}_\infty \otimes_{\widetilde{S}_\infty} S_\infty \cong H_\infty$ of $\widetilde{P}_\infty$-modules. 
    \item A homomorphism 
    \begin{multline*} R^{\square_{S_p\cup S_{\textnormal{aux}}}}_{\textnormal{loc}} \widehat{\otimes}_\cO S_\infty\llbracket x_1, \dots, x_{4(|T| + d_1) + |S_p\cup S_{\textnormal{aux}}| + q_2 - [F: \mathbb{Q}] - 1}\rrbracket  \\ \to \widetilde{R}_{b, \infty}\llbracket t_1, \dots, t_{4(|S_p \cup S_{\textnormal{aux}}\cup T| + d_1) - 1}\rrbracket .
    \end{multline*}
\end{itemize}
Let $\widetilde{\mathfrak{q}}^{\textnormal{ps}}$ denote the pullback of $\mathfrak{q}^{\textnormal{ps}}$ to $\widetilde{P}_\infty$, let $\widetilde{\mathfrak{q}}_b$ denote the pullback of $\mathfrak{q}_b$ to $\widetilde{R}_{b, \infty}$,  let $\widetilde{\mathfrak{q}}_b^\square$ denote the pullback of $\widetilde{\mathfrak{q}}_b$ to $\widetilde{R}_{b, \infty}\llbracket t_1, \dots, t_{4(|S_p \cup S_{\textnormal{aux}}\cup T| + d_1) - 1}\rrbracket$, and let $\widetilde{\mathfrak{q}}_{\textnormal{loc}}$ denote the pullback of $\widetilde{\mathfrak{q}}_b^\square$ to $R^{\square_{S_p\cup S_{\textnormal{aux}}}}_{\textnormal{loc}} \widehat{\otimes}_\cO S_\infty\llbracket x_1, \dots, x_{3(|T| + d_1)}\rrbracket $. Then, by construction, the residue rings of $\mathfrak{q}_b$, $\widetilde{\mathfrak{q}}_b$, and $\widetilde{\mathfrak{q}}_{\textnormal{loc}}$ are all isomorphic, and the induced map
\begin{multline}\label{eq_patched_presentation} \widehat{(R^{\square_{S_p\cup S_{\textnormal{aux}}}}_{\textnormal{loc}} \widehat{\otimes}_\cO S_\infty\llbracket x_1, \dots, x_{3(|T| + d_1)}\rrbracket _{\widetilde{\mathfrak{q}}_{\textnormal{loc}}})}\llbracket x_{3(|T| + d_1)+1},  \dots, x_{4(|T| + d_1) + |S_p\cup S_{\textnormal{aux}}| + q_2 - [F: \mathbb{Q}] - 1}\rrbracket  \\ \to \widehat{( \widetilde{R}_{b, \infty}  \llbracket t_1, \dots, t_{4(|S_p \cup S_{\textnormal{aux}}\cup T| + d_1) - 1}\rrbracket_{\widetilde{\mathfrak{q}}_b^\square})}  \cong \widehat{( \widetilde{R}_{b, \infty, \widetilde{\mathfrak{q}}_b} )} \llbracket t_1, \dots, t_{4(|S_p \cup S_{\textnormal{aux}}\cup T| + d_1) - 1}\rrbracket 
\end{multline}
is surjective (using at this last step the conclusion of Prop.\ \ref{prop_existence_of_TW_places}).
\begin{pp}\label{prop_presentation_of_tilde_ring}
    The underlying reduced ring of 
    \[ \widehat{(\widetilde{P}_{\infty, \widetilde{\mathfrak{q}}^{\textnormal{ps}}})}\llbracket t_1, \dots, t_{4(|S_p \cup S_{\textnormal{aux}}\cup T| + d_1) - 1}\rrbracket  \]
    is a quotient of a normal Noetherian local domain of Krull dimension $4(|S_p\cup S_{\textnormal{aux}} \cup T| + d_1) + [F : \mathbb{Q}] + |T| + q_2 - 1$. 
\end{pp}
\begin{proof}
    The same argument as in Cor.\ \ref{cor_comparison_of_patched_objects} shows that the natural map 
    \[ \widehat{(\widetilde{P}_{\infty, \widetilde{\mathfrak{q}}^{\textnormal{ps}}})} \to \widehat{(\widetilde{R}_{b, \infty, \widetilde{\mathfrak{q}}_b})} \]
    is an isomorphism.The given statement now follows from \eqref{eq_patched_presentation} and Cor.\ \ref{cor_normality_of_local_lifting_ring} and the fact that normality and reducedness 
    of complete Noetherian local rings
    is preserved under localisation and completion (by e.g.\ \cite[\href{https://stacks.math.columbia.edu/tag/0C21}{Tag 0C21}, \href{https://stacks.math.columbia.edu/tag/0C22}{Tag 0C22}]{stacks-project} and the standard fact that complete Noetherian local rings are excellent~\cite[\href{https://stacks.math.columbia.edu/tag/07QW}{Tag 07QW}]{stacks-project}).    
\end{proof}
We can now complete the proof of Prop.\ \ref{prop_local_R_equals_T}. 
The ring $\widehat{(\widetilde{P}_{\infty, \widetilde{\mathfrak{q}}^{\textnormal{ps}}})}$ acts on the non-zero module $\widehat{(\widetilde{H}_{\infty, \widetilde{\mathfrak{q}}^{\textnormal{ps}}})}$, which has depth at least $[F : \mathbb{Q}] + |T| + q_2$ (by \cite[\href{https://stacks.math.columbia.edu/tag/0FCC}{Tag 0FCC}]{stacks-project}, and since $\widetilde{H}_\infty$ has depth $1 + [F : \mathbb{Q}] + |T| + q_2$, being free over $\Lambda \widehat{\otimes}_\cO \widetilde{S}_\infty$).  By Prop.\ \ref{prop_presentation_of_tilde_ring}, $\widehat{(\widetilde{P}_{\infty, \widetilde{\mathfrak{q}}^{\textnormal{ps}}})}$ has Krull dimension at most $[F : \mathbb{Q}] + |T| + q_2$, with equality only if its spectrum is irreducible. Therefore its spectrum is irreducible and it acts nearly faithfully (in the sense of \cite[Def.\ 2.1]{Tay08}) on $\widehat{(\widetilde{H}_{\infty, \widetilde{\mathfrak{q}}^{\textnormal{ps}}})}$. Therefore, $\widetilde{H}_{\infty, \widetilde{\mathfrak{q}}^{\textnormal{ps}}} \otimes_{\widetilde{S}_\infty} S_\infty \cong H_{\infty, \mathfrak{q}}$ is a nearly faithful $\widetilde{P}_{\infty, \widetilde{\mathfrak{q}}^{\textnormal{ps}}} \otimes_{\widetilde{S}_\infty} S_\infty \cong P_{\infty, \mathfrak{q}^{\textnormal{ps}}}$-module, implying that the kernel of the map $P_{\infty, \mathfrak{q}^{\textnormal{ps}}} \to \mathbb{T}_{\infty, \mathfrak{q}}$ is nilpotent. This is precisely what we needed to prove.

\section{Construction of $T$-new lifts}\label{sec: construction of T-new lifts}
In this section we implement the lifting technique introduced in \cite{FKPI, FKPII} in order to construct lifts of Galois representations modulo $p^N$ with enough auxiliary ramification that we can prove their modularity using the automorphy lifting theorem established in \S \ref{sec_an_ALT}. This application will be taken up in the next section. 

We follow the methods of \cite{FKPI,FKPII} closely, with three important differences. First, since we can choose the initial value of $N$ as large as we like (and it will eventually go to infinity), we can skip the ``doubling'' argument used in the cited works to go from a mod $p$ representation to mod~$p^N$ for some sufficiently large~$N$.

Second, we need control over the number of auxiliary primes that are introduced in order to actually construct the $p$-adic lift. (We remind the reader that there are in fact two kinds of auxiliary primes; those in the set $T$ are used in Section \ref{sec_an_ALT} to control the size of the reducible locus, while those in the set $S_{\textnormal{aux}}$ allow us to construct the requisite $p$-adic lift.) We formulate the required result as Prop.\ \ref{prop:T-new lift} below: its key feature (relative to \cite{FKPI, FKPII}) is that, fixing the size of the set $T$ in advance, we can choose $T$ in such a way that the size of $|S_{\textnormal{aux}}|$ is bounded. 

Third, our eventual application forces us to work with representations that may not have image containing an open subgroup of $\GL_2(\mathbb{Z}_p)$. This hypothesis is used in \cite{FKPI, FKPII} to construct auxiliary Ramakrishna-type deformation problems using elements in the image contained in a split maximal torus of $\GL_{2, \mathbb{Z}_p}$ (in particular, defined over $\mathbb{Z}_p$). Our representations may only have image containing a rational conjugate of $\GL_2(\mathbb{Z}_p)$, which forces us to generalise the definition of these deformation problems to allow tori that are only defined rationally, not integrally. We analyse these deformation problems by direct calculation; it would be interesting to understand how to define them for a general reductive group, along the lines considered in~\cite{FKPI}. 
\subsection{Preliminaries}\label{subsec:preliminaries for lifting argument}

\subsubsection{Definitions and assumptions}\label{subsubsec:assumptions on Galois representation}
Let $p$ be an odd prime, let $E / \mathbb{Q}_p$ be a coefficient field, let $F$ be a totally real number field, and let 
\[ \rhobar = \left (\begin{smallmatrix}
\chibar_1 & *\\0 & \chibar_2
\end{smallmatrix} \right ) : G_F \to \GL_2(k) \]
be  a continuous indecomposable representation with determinant $\det \rhobar = \epsilon^{-1}.$
We assume the existence of a complex conjugation~$c_u \in G_F$ such that $\rhobar(c_u) = \diag(-1, 1)$. %
We write~$\chibar \coloneqq  \chibar_1 \chibar_2^{-1}$, and we assume that $\chibar|_{G_{F(\zeta_p)}} \ne 1$.
We let $\lambda \in (\mathbb{Z}^2_+)^{\Hom(F, E)}$, and we let $\rho : G_{F} \to \GL_2(\cO)$ be a lift of~$\rhobar$ such that: 
\begin{enumerate}
    \item $\rho$ is unramified at all but finitely many places of~$F$.
    \item $\det \rho = \epsilon^{-1}.$ %
    \item If~$v \in S_p,$ then 
    $\rho|_{G_{F_v}}$ is de Rham of regular weight~$\lambda_v$ and inertial type~$\tau_v$. 
    \item There exists $h \in \GL_2(E)$ such that the image of~$\rho$ contains an open subgroup of $h \GL_2(\bZ_p) h^{-1}$.
\end{enumerate}

\begin{defn}\label{defn:good prime}
If~$v \nmid p$ is a finite place of~$F$,
we say that~$v$ is \emph{good} for~$\rhobar$ if $q_v \equiv -1 \text{ mod } p$, 
$\rhobar$ is unramified at~$v$, and $\rhobar(\Frob_v)$ is conjugate to $\diag(-1, 1)$ within $\rhobar(G_F)$.
\end{defn}

For each~$N \geq 1$, we write:
\begin{itemize}
    \item $\rho_N \coloneqq  \rho \otimes_\cO \cO/\varpi^N$.
    \item $K_N \coloneqq  F(\rho_N, \mu_{p})$. We also write~$K \coloneqq  K_1$.
    \item $F_N^* \coloneqq  F(\rhobar, \ad^0 \rho_N, \mu_{p^N})$.
\end{itemize}
We choose~$C_\rho > 0$ minimally so that $\rho(\cO[G_F])$ contains $\varpi^{C_\rho}M_2(\cO)$.
(This is possible, by Prop.\ \ref{prop_pseudocharacters_with_large_image}(1).)

\subsubsection{The set~$S_{\Sha}$}
By~\cite[Lem.\ 7.1]{FKPII}, our assumption that $\chibar|_{G_{F(\zeta_p)}} \ne 1$ implies that
\begin{equation}\label{eqn:vanishing of Galois cohomology of K}
H^1(\Gal(K/F), \ad^0 \rhobar(1)) = 0.
\end{equation}
Whenever~$S$ is a finite set of places of~$F,$ and~$M$ a continuous $\cO[G_{F, S}]$-module, we define $\Sha^i_S(G_{F, S}, M)$ by the exact sequence
\[
0 \to \Sha^i_S(G_{F, S}, M) \to H^i(G_{F, S}, M) \to \prod_{v \in S}H^i(G_{F_v}, M).
\]
Write~$S_0$ for the set of finite, prime-to-$p$, ramified places of~$\rho$.
Because of~\eqref{eqn:vanishing of Galois cohomology of K}, we can find 
a finite set~$S_{\Sha}$ of places of~$F$, disjoint from $S_0 \cup S_p$, and such that:
\begin{enumerate}
    \item For all~$v \in S_{\Sha}:$
    $q_v \equiv 1 \mod \varpi$ and $\rhobar |_{G_{F_v}}$ is trivial.
    \item Writing $S\coloneqq  S_p \cup S_0 \cup S_{\Sha}$, then %
\begin{equation}\label{eqn:vanishing of Sha1}
\Sha^1_S(G_{F, S}, \ad^0 \rhobar(1)) = 0.
\end{equation}
\end{enumerate}
We fix such a set~$S_{\Sha}$ for the duration of Section~\ref{sec: construction of T-new lifts}, and we continue writing $S\coloneqq  S_p \cup S_0 \cup S_{\Sha}$.
Note that~\eqref{eqn:vanishing of Sha1} implies that
\[
\Sha^1_{S \cup X}(G_{F, S \cup X}, \ad^0 \rhobar(1)) = 0
\]
for every~$X$ disjoint from~$S$, since this group coincides with $\Sha^1_{S \cup X}(G_{F, S}, \ad^0 \rhobar(1)),$ which is a subgroup of $\Sha^1_S(G_{F, S}, \ad^0 \rhobar(1)).$
It then follows from Poitou--Tate duality that $\Sha^2_{S \cup X}(G_{F, S \cup X}, \ad^0 \rhobar) = 0$.

\subsubsection{Local lifting rings}
We introduce local lifting rings at places of~$S$ 
in the following way (the determinant is always fixed to be~$\epsilon^{-1}$):
\begin{enumerate}
    \item For each $v \in S_p$,
    we choose a component~$R_v$ of~$R_{\rhobar|_{G_{F_v}}}^{\lambda_v, \tau_v}$ containing~$\rho|_{G_{F_v}}$; 
    \item For each $v \in S_{\Sha},$ we let~$R_v$ be the unramified component of~$R_{\rhobar|_{G_{F_v}}}^\square$ containing~$\rho|_{G_{F_v}}$; 
    \item For each $v \in S_0,$ we choose a component~$R_v$ of $R_{\rhobar_{G_{F_v}}}^\square$ containing~$\rho|_{G_{F_v}}$. 
\end{enumerate}
In all cases, $\Spec R_v[1/\varpi]$ is formally smooth at every closed point.
(When $v \mid p$, this is a consequence of~\cite[Thm.\ A.2]{Kis09} and~\cite[Rem.\ 16]{Boo19}, asserting that the irreducible components of $R_{\rhobar|_{G_{F_v}}}^{\lambda_v, \tau_v}[1/\varpi]$ are regular.
When $v \nmid p$, the result follows from a similar argument, making use of the moduli space of Weil--Deligne representation rather than the moduli space of $(\varphi, N)$-modules; or see~\cite[Lem.\ 2.6.6]{Bel19}.)

\begin{rem}\label{rem:intended choice of components}
In our intended application, which is the proof of Prop.\ \ref{prop:level-lowering}, the following will be true:
\begin{enumerate}
\item If~$v \in S_p$, then $\rho$ is ordinary and semistable of weight~$\lambda_v$.
The inertial type~$\tau_v$ is chosen to be trivial, so $R_{\rhobar|_{G_{F_v}}}^{\lambda_v, \tau_v}$ is the semistable lifting ring
$R_{\rhobar|_{G_{F_v}}}^{\lambda_v, \mathrm{st}}$ of weight~$\lambda_v$.
The ring~$R_v$ is chosen 
to be an ordinary component of $R_{\rhobar|_{G_{F_v}}}^{\lambda_v, \mathrm{st}}$ containing $\rho |_{G_{F_v}}$,
such that
for every $\overline{\mathbb Q}_p$-point~$\rho'_v$ of~$\Spec R_v$, writing $\rho'_v \sim \left( \begin{smallmatrix} \psi'_{v, 1} & * \\ 0 & \psi'_{v, 2} \end{smallmatrix} \right)$ for characters $\psi'_{v, i} : G_{F_v} \to \overline{\mathbb{Q}}_p^\times$, as in the definition of $\lambda_v$-ordinary, then $\psi'_{v, 1} \circ \Art_{F_v}$ is trivial on the subgroup of $p$-power roots of unity of $\cO_{F_v}^\times$.
\item If~$v \in S_0$, then 
$q_v \equiv 1 \text{ mod } p$, $\rhobar^{\textnormal{ss}}|_{G_{F_v}}$ is trivial, and
$(\rho|_{G_{F_v}})^{\textnormal{ss}} \cong 1 \oplus \epsilon^{-1}$;
the ring~$R_v$ is chosen to be the Steinberg lifting ring of~$\rhobar|_{G_{F_v}}$.
\item If~$v \in S_{\Sha}$, then 
$q_v \equiv 1 \text{ mod } p$, $\rhobar|_{G_{F_v}}$ is trivial, and
$\rho|_{G_{F_v}}$ is unramified; the ring~$R_v$ is chosen to be the unramified component of the lifting ring of~$\rhobar|_{G_{F_v}}$.
\end{enumerate}
\end{rem}

\subsubsection{Statement of main result}
Our target in this section is the following result. 
In its statement, we fix integers $M_\rho$ and $P_\rho$ (produced by Lem.\ \ref{lem_rationally_defined_elements_in_image}), and integers $M_v$ for~$v \in S$ produced by the construction of Section~\ref{subsubsec:local condition at S}; and we fix lifting rings~$R_v$ for~$v \in S$, as in Section~\ref{subsec:preliminaries for lifting argument}.
We also employ the notation $C_\rho$ from Section~\ref{subsec:preliminaries for lifting argument}, and $h^1_\varnothing(\ad^0 \rhobar)$ from Section~\ref{subsec:Selmer groups II}.
These are quantities depending only on~$\rho$ and on the set~$S$. %

\begin{pp}\label{prop:T-new lift}
Make the same assumptions as in Section~\ref{subsubsec:assumptions on Galois representation}.
Fix~$N > \max \{P_\rho + 2 M_\rho, P_\rho + \max_{v \in S} M_v, P_\rho + C_\rho\}$, and an integer $a>0$. 
Then there exist finite sets~$T\coloneqq T_{a, N}$ and~$S_{\aux}$ of finite places of~$F$, disjoint from~$S$, and a continuous lift 
\[
\rho' : G_{F, S \cup S_{\aux} \cup T} \to \GL_2(\cO)
\]
of~$\rho \otimes_\cO \cO/\varpi^{N-P_\rho}$, such that:
\begin{enumerate}
    \item $\det \rho' = \epsilon^{-1}$;
    \item For all~$v \in S$: 
    $\rho'|_{G_{F_v}}$ defines an $\cO$-point of $\Spec R_v$. 
    \item For all $v \in T$: $q_v \equiv -1 \text{ mod }p$,
    $\rhobar(\Frob_v)$ is conjugate to~$\diag(-1, 1)$ in~$\rhobar(G_F)$, and
    $(\rho'|_{G_{F_v}} \otimes_\cO E)^{\textnormal{ss}} \cong 1 \oplus \epsilon^{-1}$. 
    \item For all $v \in S_{\textnormal{aux}}$: $q_v \equiv 1\text{ mod } \varpi^{M_\rho}$, $q_v \not \equiv 1\text{ mod } \varpi^{M_\rho+1}$, $\overline{\rho}|_{G_{F_v}}$ is trivial, and $(\rho'|_{G_{F_v}} \otimes_\cO E)^{\textnormal{ss}} \cong 1 \oplus \epsilon^{-1}$.
    \item $|S_{\aux}| \leq P_\rho(h^1_\varnothing(\ad^0 \rhobar)+1)+1$; 
    \item $|T| = a$;
    \item $\rho'(\cO[G_F])$ contains $\varpi^{C_\rho}M_2(\cO)$.
\end{enumerate}
Furthermore, if~$L/F$ is a finite Galois extension linearly disjoint from $F(\rho, \epsilon)$, and such that $p \nmid [L:F]$, then the sets~$T$ and~$S_{\aux}$ can be chosen to split completely in~$L$. 
\end{pp}

\begin{rem}\label{rem:independence of w}
    The sets~$T$ and~$S_{\aux}$ in Prop.\ \ref{prop:T-new lift} will be constructed by an application of the Chebotarev density theorem.
    Given any finite place~$\mathfrak{w}$ of~$F$, we can therefore choose them in such a way that $\mathfrak{w} \not \in T \cup S_{\aux}$.
\end{rem}

Note that Prop.\ \ref{prop:T-new lift}~(2) implies that~$\rho'$ is unramified at~$S_{\Sha}$. 
We immediately record the following simplification.

\begin{lemma}\label{lem:deducing bound on image of rho}
    If~$\rho' : G_F \to \GL_2(\cO)$ is congruent to~$\rho$ modulo~$\varpi^{N-P_\rho}$, then $\rho'(\cO[G_F])$ contains $\varpi^{C_\rho}M_2(\cO)$.
    Hence, condition~(7) in Prop.\ \ref{prop:T-new lift} is implied by the other conditions.
\end{lemma}
\begin{proof}
Let $X_\rho \coloneqq  \rho(\cO[G_F]), X_{\rho'} \coloneqq  \rho'(\cO[G_F])$, and let~$X \coloneqq  M_2(\cO)$.
Since $\rho' \equiv \rho \text{ mod } \varpi^{N - P_\rho} \cO,$
the images of $X_\rho$ and~$X_{\rho'}$ in $X/\varpi^{N-P_\rho} X$ coincide.
By definition of~$C_\rho$, they contain $\varpi^{C_\rho} X/\varpi^{N-P_\rho}X$.
Hence $X_{\rho'} + \varpi^{N-P_\rho}X \supset \varpi^{C_\rho}X.$
Since $N-P_\rho > C_\rho$, Nakayama's lemma applied to the inclusion $X_{\rho'} \subset X_{\rho'} + \varpi^{C_\rho}X$ (of finite type $\cO$-modules) shows that $X_{\rho'} = X_{\rho'} + \varpi^{C_\rho} X,$ hence $\varpi^{C_\rho} X \subset X_{\rho'}$, as desired.
\end{proof}

\subsubsection{Calculations in Galois cohomology}
For any continuous character $\overline\psi : G_F \to k^\times$, we will write $k(\overline\psi)$ for the corresponding one-dimensional $k[G_F]$-module, and 
$\bF_p(\overline\psi)$ for a representative of the unique $\bF_p[G_F]$-isomorphism class of simple $\bF_p[G_F]$-subquotients of~$k(\overline\psi)$.
(The uniqueness follows from the fact that $k/\bF_p$ is finite and separable. Similarly, $\bF_p(\overline \psi) \cong \bF_p(\overline \psi')$ if and only if $\overline \psi, \overline \psi'$ are $\Gal(k/\bF_p)$-conjugate.)

\begin{lemma}\label{lem:uniserial adjoint}
The $k[G_F]$-module $\ad^0 \rhobar$ has socle~$k(\chibar)$ and cosocle~$k(\chibar^{-1}).$
Furthermore,
\[
\operatorname{soc}_{k[G_F]} \ad^0 \rhobar = \{\psi \in \operatorname{End}_k(\rhobar): \operatorname{soc}_{k[G_F]}\rhobar \subset \ker \psi, \operatorname{im}\psi \subset \operatorname{soc}_{k[G_F]} \rhobar\}.
\]
\end{lemma}
\begin{proof}
    Since $\ad^0 \rhobar$ is self-dual it suffices to prove the statement about the socle.
    This is a consequence of the fact that, for any character $\overline\psi : G_F \to k^\times$, we have
    \[
    \dim_k \Hom(\rhobar \otimes \overline\psi, \rhobar) \leq 1
    \]
    with equality if and only if~$\overline\psi = 1, \chibar$.
\end{proof}

\begin{lemma}\label{lem:computing composition factors}
    Fix~$N \geq 1$.
    \begin{enumerate}
    \item The group $\Gal(K_N/K)$ has a filtration by $G_F$-equivariant normal subgroups,
    each of whose subquotients is $\bF_p[G_F]$-linearly isomorphic to $\bF_p(\chibar), \bF_p$ or~$\bF_p(\chibar^{-1})$.
    \item The group $\Gal(K(\mu_{p^N})/K(\mu_p))$ has a filtration by $G_F$-equivariant normal subgroups, each of those subquotients is $\bF_p[G_F]$-linearly isomorphic to $\bF_p$.
    \end{enumerate}
\end{lemma}
\begin{proof}
    Since~$K_N \coloneqq  F(\rho_N, \mu_{p})$ and~$K \coloneqq  F(\rhobar, \mu_p)$,
    the group $\Gal(K_N/K)$ injects into $\Gal(F(\rho_N)/F(\rhobar))$ in a $G_F$-equivariant manner, by restriction of automorphisms.
    On the other hand, there is a $G_F$-equivariant embedding
    \[
    \rho_N : \Gal(F(\rho_N)/F(\rhobar)) \to 1 + \varpi M_2(\cO/\varpi^N \cO),
    \]
    and the target of this embedding admits a filtration as in the statement of the lemma.
    We can intersect each term of this filtration with $\Gal(K_N/K)$.
    The result is a filtration of $\Gal(K_N/K)$ by $G_F$-equivariant normal subgroups, each of whose subquotients is $\bF_p[G_F]$-linearly isomorphic to $\bF_p(\chibar), \bF_p$ or $\bF_p(\chibar^{-1})$ (or zero). 
    This concludes the proof of~(1).

    The proof of~(2) is immediate from the fact that $\Gal(F(\mu_{p^N})/F)$ is abelian,
    and $\Gal(K(\mu_{p^N})/K(\mu_p))$ injects $G_F$-equivariantly into it (by restriction of automorphisms).
\end{proof}

\begin{lemma}\label{lem:orthogonality of adjoint and dual}
Let~$\overline\psi_1 \in \{\chibar, 1, \chibar^{-1}\}$ and $\overline\psi_2 \in \{\chibar\omega, \omega, \chibar^{-1} \omega\}.$
Then 
\[
\Hom_{\bF_p[G_F]}(\bF_p(\overline{\psi}_1), \bF_p(\overline{\psi}_2)) = 0.
\]
\end{lemma}
\begin{proof}
    Since~$\chibar$ and~$\omega$ are totally odd, parity considerations reduce to proving that $\Hom_{\bF_p[G_F]}(\bF_p, \bF_p(\chibar^{\pm 1}\omega)) = 0$ and $\Hom_{\bF_p[G_F]}(\bF_p(\chibar^{\pm 1}), \bF_p(\omega)) = 0$.
    This is because, by assumption, no $\Gal(k/\bF_p)$-conjugate of~$\chibar$ is equal to a power of~$\omega$.
\end{proof}

\begin{lemma}\label{lem:extra cohomology vanishing}
    For all~$N \geq 1$, $H^1(\Gal(K_N/F), \ad^0 \rhobar(1)) = 0.$
\end{lemma}
\begin{proof}
By inflation-restriction, and the fact that $H^1(K/F, \ad^0 \rhobar(1)) = 0,$ 
it suffices to prove that
\[
H^1(\Gal(K_N/K), \ad^0 \rhobar(1))^{\Gal(K/F)} = 0.
\]
The group $\Gal(K_N/K) = \Gal(F(\rho_N, \zeta_p)/F(\rhobar, \zeta_p))$ 
acts trivially on $\ad^0 \rhobar(1)$,
and admits a filtration as in Lem.\ \ref{lem:computing composition factors}. 
On the other hand, the simple $\bF_p[G_F]$-subquotients of~$\ad^0\rhobar(1)$
are $\bF_p(\chibar\omega), \bF_p(\omega), \bF_p(\chibar^{-1}\omega).$
We can thus conclude by an application of Lem.\ \ref{lem:orthogonality of adjoint and dual}.
\end{proof}

The following result is essentially a restatement of~\cite[Lem.\ 6.4]{FKPI} and~\cite[Lem.\ 5.3]{FKPII}, but we provide some details since we work in a slightly different context.

\begin{lemma}\label{lem:definition of M}
\leavevmode
    \begin{enumerate}
        \item For all~$1 \leq i \leq N$, we have $H^1(\Gal(F_N^*/F), \ad^0 \rho_i(1)) = 0$.
        \item 
        There exists an integer~$N_\rho$ such that for all~$N> N_\rho$,
        the map 
        \[
        H^1(\Gal(F_N^*/F), \ad^0 \rho_{N_\rho}) \to H^1(\Gal(F_N^*/F), \ad^0 \rhobar)
        \]
        is zero.
    \end{enumerate}
\end{lemma}
\begin{proof}
Proof of~(1): Fix~$N \geq 1$. 
It suffices to prove that $H^1(\Gal(F_N^*/F), \ad^0\rhobar(1)) = 0,$
and by inflation-restriction and the fact that $H^1(\Gal(K/F), \ad^0 \rhobar(1)) = 0,$
it suffices to prove that
\begin{equation}\label{eqn:vanishing to prove}
\Hom_{G_F}(\Gal(F_N^*/K), \ad^0 \rhobar(1)) = 0.
\end{equation}
Since $F_N^* \coloneqq  F(\rhobar, \ad^0 \rho_N, \mu_{p^N})$, there is a $G_F$-equivariant injection
\[
\Gal(F_N^*/K) \to \Gal(K(\ad^0 \rho_N)/K) \times \Gal(K(\mu_{p^N})/K).
\]
Since $K_N \coloneqq  F(\rho_N, \mu_p)$, there is a $G_F$-equivariant surjection 
\[ \Gal(K_N/K) \to \Gal(K(\ad^0 \rho_N)/K). \]
Hence, by Lem.\ \ref{lem:computing composition factors}, $\Gal(F_N^*/K)$ can be filtered by $G_F$-equivariant normal subgroups, with successive quotients isomorphic to $\bF_p(\chibar), \bF_p, \bF_p(\chibar^{-1})$ as $\bF_p[G_F]$-modules.
Hence~\eqref{eqn:vanishing to prove} follows from an application of Lem.\ \ref{lem:orthogonality of adjoint and dual}.

Proof of~(2): Let $F_N \coloneqq  F(\rhobar, \ad^0 \rho_N)$.
For each~$1 \leq i \leq N$ we have the inflation-restriction sequence
\[
0 \to H^1(\Gal(F_N/F), \ad^0 \rho_i) \to H^1(\Gal(F_N^*/F), \ad^0 \rho_i) \to
\Hom_{G_F}(\Gal(F_N^*/F_N), \ad^0 \rho_i),
\]
and the last term vanishes, because~$G_F$ acts trivially on~$\Gal(F_N^*/F_N)$ but
$\ad^0 \rho_i$ does not have any nonzero trivial $\bZ_p[G_F]$-submodule.
(To see this last claim, it suffices to show that $\ad^0 \rhobar$ 
does not contain any trivial $\bF_p[G_F]$-submodule,
which is a consequence of Lem.\ \ref{lem:uniserial adjoint}.)

So it suffices to prove that there exists~$N_\rho$ such that, for all $N>N_\rho$, the map
\begin{equation}\label{eqn:map to prove vanishing}
H^1(\Gal(F_N/F), \ad^0 \rho_{N_\rho}) \to H^1(\Gal(F_N/F), \ad^0 \rhobar)
\end{equation}
is zero.
Note that the order of $\Gal(F_N/F(\ad^0 \rho_N))$ is coprime to~$p$: in fact, this group embeds via~$\rhobar$ in the centre of $\GL_2(k)$.
Hence, by inflation-restriction, we can replace~$F_N$ with~$F(\ad^0 \rho_N)$ in~\eqref{eqn:map to prove vanishing}.

Let~$\Gamma \coloneqq  (\ad^0 \rho)(G_F)$, a compact analytic group over~$\bQ_p$ which surjects onto $\Gal(F(\ad^0 \rho_N)/F)$.
Since $H^1(F(\ad^0 \rho_N)/F, \ad^0 \rho_i) \to H^1(\Gamma, \ad^0 \rho_i)$ is injective for all values of~$i$,
it suffices to find~$N_\rho$ such that, for all~$N>N_\rho$, the map
\[
H^1(\Gamma, \ad^0 \rho_{N_\rho}) \to H^1(\Gamma, \ad^0 \rhobar)
\]
is zero.
Since $H^0(\Gamma, \ad^0 \rhobar) = 0$,
this is a consequence of~\cite[Cor.\ B.2]{FKPI}, applied to the $\cO[\Gamma]$-module~$\ad^0 \rho$.
(The assumption in \emph{loc.\ cit.} that $\Lie \Gamma$ is semisimple follows from our assumption that $\rho(G_F)$ contains a conjugate of an open subgroup of $\GL_2(\bZ_p)$. In fact, under this assumption \cite[Thm.\ 0.7]{Pin98} implies that there exist an intermediate extension $\bQ_p \subset E_0 \subset E$, a central simple $E_0$-algebra $D_0$, and an isomorphism $D_0 \otimes_{E_0} E \isom M_2(E)$ such that $\rho(G_F) \subset E^\times D_0^\times$ and $\rho(G_F)$ has open image in $E^\times D_0^\times/E^\times$. Thus the Lie algebra of $\ad^0 \rho(G_F)$ coincides with $\Lie (D_0^\times/E_0^\times)$, which is semisimple.) 
\end{proof}
If $g \in \GL_2(\mathbb{Z}_p)$, then we define $E_g \coloneqq h g \left ( \begin{smallmatrix} 0 & 0 \\ 0 & 1  \end{smallmatrix} \right ) g^{-1} h^{-1}$, $N_g \coloneqq h g \left ( \begin{smallmatrix} 0 & 1 \\ 0 & 0 \end{smallmatrix} \right )g^{-1}h^{-1}$. These are elements of~$M_2(E)$. Let $\lambda_g, \mu_g \in \mathbb{Z}$ be minimal such that $\varpi^{\lambda_g} E_g$, $\varpi^{\mu_g} N_g$ lie in $M_2(\cO)$. Note that $\lambda_g \geq 0$ for all $g$, since $E_g$ is an idempotent. 
\begin{lemma}\label{lem_existence_of_well_placed_elements}
    We can find $g_1, \dots, g_s \in \GL_2(\mathbb{Z}_p)$ and $C_0 \geq 1$ such that for any $X, Y \in \ad^0 \rho \setminus \varpi \ad^0 \rho$, there is some $1\leq i \leq s$ such that 
    \[ \ord_\varpi \tr \varpi^{\lambda_{g_i}} E_{g_i} X  \leq C_0 \]
    and
    \[ \ord_\varpi \tr \varpi^{\mu_{g_i}} N_{g_i} Y \leq C_0. \]
\end{lemma}
\begin{proof}
    Consider the function $g \mapsto (\tr E_g X) (\tr N_g Y ) \in E[\GL_2]$.
    For any fixed $(X, Y)$ it is non-zero. The subset $\GL_2(\mathbb{Z}_p)$ is Zariski dense in $\GL_{2, E}$, so we can find $g \in \GL_2(\mathbb{Z}_p)$ such that $\tr E_g X$ and $\tr N_g Y$ are both non-zero. By compactness, we can find $g_1, \dots, g_r \in \GL_2(\mathbb{Z}_p)$ such that for any $X, Y \in \ad^0 \rho \setminus \varpi \ad^0 \rho$, there is some $i = 1, \dots, s$ such that $\tr E_{g_i} X$ and $\tr N_{g_i} Y$ are both non-zero. Since $\ad^0 \rho \setminus \varpi \ad^0 \rho$ is compact, the existence of a uniform bound on the maximum valuation is automatic. 
\end{proof}
We henceforth fix $g_1, \dots, g_s$ and $C_0$ as in the statement of Lem.\ \ref{lem_existence_of_well_placed_elements}, and set $P_\rho \coloneqq N_\rho + 2 C_0$. 
\begin{lemma}\label{lem_rationally_defined_elements_in_image}
    We can find an integer $M_\rho \geq P_\rho + \max_j \lambda_{g_j}$, divisible by $e$, with the following property: for any $1 \leq i \leq s$, there exists $\gamma \in G_F$ satisfying the following conditions:
    \begin{enumerate}
        \item Writing $q \coloneqq \det \rho(\gamma)$, we have $\rho(\gamma) = 1 + (q-1) E_{g_i}$ and $\rho_{P_\rho}(\gamma) = 1$. 
        \item $q \equiv 1 \text{ mod }\varpi^{M_\rho}$, but $q \not\equiv 1 \text{ mod }\varpi^{M_\rho+1}$. 
    \end{enumerate}
\end{lemma}
\begin{proof}
    According to our standing hypotheses, $\rho(G_F)$ contains an open subgroup of $h \GL_2(\mathbb{Z}_p) h^{-1}$. In particular, there is $b \geq 1$ such that $\rho(G_F)$ contains $1 + (x-1) E_{g_j}$ for any $x \in 1 + p^b \mathbb{Z}_p$ and $1 \leq j \leq s$. Since $\det(1+(x-1)E_{g_j}) = x$, the lemma thus holds for any $M_\rho$ such that $e | M_\rho$ and $M_\rho \geq \max( b e , P_\rho + \max_j \lambda_{g_j})$.
\end{proof}

\subsection{Selmer groups}
We henceforth fix integers $M_\rho$ and $P_\rho$ satisfying the above conditions. 

\subsubsection{Local condition at~$S$}\label{subsubsec:local condition at S}

Recall our choice of lifting rings~$(R_v: v \in S)$ 
from Section~\ref{subsec:preliminaries for lifting argument}.
For all~$v \in S$, we write $\rho_v \coloneqq  \rho |_{G_{F_v}}$.
By construction, $\rho_v$ defines an $\cO$-point of~$\Spec R_v$.
Following~\cite[Lem.\ 4.4]{FKPI} we introduce, for each~$r \geq 1$, 
the notation~$Z^1_{r, v}$ for the $\cO$-submodule of
\[
\Hom_{R_v}(\Omega_{R_v/\cO} \otimes_{R_v, \rho_v} \cO, \cO/\varpi^r)
\]
consisting of those morphisms that annihilate the $\varpi$-power torsion.
If we regard it as a submodule of
\[
Z^1(G_{F_v}, \ad^0 \rho_r),
\]
then it contains $B^1(G_{F_v}, \ad^0 \rho_{r})$, by~\cite[Lem.\ 4.5]{FKPI}, and we denote by~$\mathcal{L}_{r, v}$ its image in
\[
H^1(G_{F_v}, \ad^0 \rho_{r}).
\]
When~$r = 1$, we omit~$r$ from the notation and write~$\mathcal{L}_v \coloneqq  \mathcal{L}_{1, v}$.
We also write~$\mathcal{L}_{r, v}^\perp$ for the annihilator of~$\mathcal{L}_{r, v}$ under the local duality pairing
\[
H^1(G_{F_v}, \ad^0 \rho_r) \otimes_{\cO/\varpi^r} H^1(G_{F_v}, \ad^0 \rho_r(1)) \to \cO/\varpi^r.
\]
By the same argument as in the proof of~\cite[Prop.\ 4.7(3)]{FKPI}, the local conditions~$\mathcal{L}_{r, v}$ are automatically balanced, in the sense that
\begin{align*}
|\mathcal{L}_{r, v}| &=  |H^0(G_{F_v}, \ad^0 \rho_r)| \text{ if } v \nmid p\\
 &= |H^0(G_{F_v}, \ad^0 \rho_r)|\cdot |\cO/\varpi^r|^{[F_v:\bQ_p]} \text{ if } v \mid p.
\end{align*}
We now apply~\cite[Prop.\ 4.7]{FKPI} to each~$v \in S$, with the integer there denoted $r_0$ being equal to~$P_\rho$.
We obtain: 
\begin{itemize}
\item an integer~$M_v$, denoted~$n_0$ in \emph{loc.\ cit.}, and
\item for each $n \geq M_v$,
a subset $D_{n, v}$ of the set of $\cO/\varpi^n$-points of $\Spec(R_v)$, containing $\rho_n |_{G_{F_v}}$ and denoted~$Y_n$ in \emph{loc.\ cit.},
\end{itemize}
such that, for all~$1 \leq r \leq P_\rho$ and~$n \geq M_v$, reduction modulo~$\varpi^n$ defines a surjection
$D_{n+r, v} \to D_{n, v}$
whose fibres are principal homogeneous spaces under the natural action of~$Z^1_{r, v}$.

\subsubsection{Local condition at good primes}\label{subsubsec:local condition at T}
If~$v$ is a good prime for~$\rhobar$, in the sense of Def.\ \ref{defn:good prime}, we write~$R_v \coloneqq  R^{\St, 1}_{\rhobar|_{G_{F_v}}}$ 
for the irreducible component of $R^\square_{\rhobar|_{G_{F_v}}}$ whose $\overline \bQ_p$-points 
classify the Steinberg lifts of~$\rhobar|_{G_{F_v}}$ that admit a $G_{F_v}$-stable line whose $\Frob_v$-eigenvalue is congruent to~$1$ modulo~$\varpi$. 
Equivalently, these are lifts of~$\rhobar$ whose semisimplification has the form $1 \oplus \epsilon^{-1}$
(as opposed to its nontrivial unramified quadratic twist).
Note that $R_v$ is formally smooth over~$\cO$: in the notation of~\cite[Prop.\ 5.6(2b)]{Sho16},
it corresponds to the ideal $(X_1, X_3)$.

If~$n \geq 1$, we say that~$\rho_n$ is \emph{new at~$v$} if its restriction to~$G_{F_v}$ defines a point $R_v \to \cO/\varpi^n$.
Note that, by definition of a good prime, if~$v$ is good for~$\rhobar$ then $\rhobar$ is new at~$v$.

Assume that~$\rho_{P_\rho}$ is new at~$v$, and choose an isomorphism $\cO\llbracket X_1, X_2, X_3\rrbracket  \isom R_v$, inducing $(\Spec R_v)(\cO) \isom (\varpi\cO)^3$.
We choose an arbitrary lift~$\rho'_v$ of~$\rho_{P_\rho} |_{G_{F_v}}$
to an $\cO$-point of $\Spec R_v$, 
and for all~$n \geq 1,$ we let $D_{n, v}$ be the mod $\varpi^n$-reduction of $\rho'_v + (\varpi^{P_\rho}\cO)^3$.
Note that reduction modulo~$\varpi^{n}$ defines a surjection $D_{n+r, v} \to D_{n, v}$ for all~$n, r \geq 1$.

If~$1 \leq r \leq P_\rho$, then $\rho_r$ is new at~$v$, and $\rho_r \cong \rho'_v \otimes_\cO \cO/\varpi^r$. 
We write $Z^1_{r, v}$ for the image in
$Z^1(G_{F_v}, \ad^0 \rho_r)$
of $\Hom_{\cO}(\Omega_{R_v/\cO} \otimes_{R_v, \rho'_v} \cO/\varpi^r, \cO/\varpi^r)$,
and we write~$\mathcal{L}_{r, v}$ for the image of~$Z^1_{r, v}$ in~$H^1(G_{F_v}, \ad^0 \rho_r)$.
For all~$n \geq P_\rho$ and~$1 \leq r \leq P_\rho$, the fibres of $D_{n+r, v} \to D_{n, v}$
are principal homogeneous spaces for~$Z^1_{r, v}$.

\begin{rem}
    $D_{n, v}$ is the set of lifts of~$\rhobar|_{G_{F_v}}$ to~$\cO/\varpi^n,$
    which are new at~$v$, and congruent to $\rho_{P_\rho}$ modulo~$\varpi^{P_\rho}$.
    Our construction is therefore a special case of~\cite[Prop.\ 4.7]{FKPI}, applied to the lift~$\rho'_v$ and the parameter~$r_0 \coloneqq  P_\rho$ (which produces~$n_0 = P_\rho$, by~\cite[Lem.\ 4.3]{FKPI} and its proof).

    The remaining parts of~\cite[Prop.\ 4.7]{FKPI} therefore apply, and we conclude that $|\mathcal{L}_{r, v}| = |H^0(G_{F_v}, \ad^0 \rho_r)|$ for all~$1 \leq r \leq P_\rho$.

\end{rem}

\begin{rem}\label{rem:explicit local condition n-good primes}
The module~$\mathcal{L}_v \coloneqq  \mathcal{L}_{1, v}$ 
is by definition the image of %
\[
\Hom_k(\Omega_{R_v/\cO} \otimes_{R_v, \rhobar} k, k) = \Hom_k(\fm_{R_v}/(\varpi, \fm_{R_v}^2), k)
\]
in 
\[
Z^1(G_{F_v}, \ad^0 \rhobar) = \Hom_k(\fm_{R^\square_{\rhobar|_{G_{F_v}}}}/(\varpi, \fm^2_{R^\square_{\rhobar|_{G_{F_v}}}}), k).
\]
We make it explicit as follows.
After conjugation by an element of~$\rhobar(G_F)$, we can assume that~$\rhobar(\Frob_v) = \diag(-1, 1)$,
and so
\[
\rhobar |_{G_{F_v}} = \omega^{-1}|_{G_{F_v}} \oplus 1.
\]
Writing $\ad^0 \rhobar(1) = \fn^- \oplus \ft^0 \oplus \fn$ for the standard triangular decomposition of $\ad^0 \rhobar = \mathfrak{sl}_2$
with respect to this basis of the representation space of~$\rhobar$,
we find that
\begin{gather*}
    \ad^0 \rhobar|_{G_{F_v}} \isom 
    \left ( \begin{smallmatrix}
        k & k\omega^{-1}\\k\omega& k
    \end{smallmatrix} \right ) = \fn \omega^{-1} \oplus \ft^0 \oplus \fn^- \omega,
    \\
    \ad^0 \rhobar(1)|_{G_{F_v}} \isom 
    \left ( \begin{smallmatrix}
        k\omega & k\\k\omega^2 & k\omega
    \end{smallmatrix} \right ) = \fn \oplus \ft^0 \omega \oplus \fn^- \omega^2.
\end{gather*}
Then
\[
\mathcal{L}_v = H^1(G_{F_v}, \fn^- \omega),
\]
and the orthogonal complement of~$\mathcal{L}_v$ is
\[
\mathcal{L}_v^\perp = H^1(G_{F_v}, \fn^- \omega^2 \oplus \ft^0\omega).
\]
We deduce the following equivalent condition for a class $[\psi] \in H^1(G_{F_v}, \ad^0 \rhobar(1))$ to have nonzero image in
$H^1(G_{F_v}, \ad^0 \rhobar(1))/\mathcal L_v^\perp$.
Let~$\pi_\fn$ be the idempotent
endomorphism of $\ad^0\rhobar(1)$
corresponding to the
projection onto~$\fn$ in the triangular decomposition of $\ad^0 \rhobar(1)$ just discussed.
Then~$\pi_\fn$ is $G_{F_v}$-equivariant, and its kernel on $H^1(G_{F_v}, \ad^0 \rhobar(1))$ coincides with $\mathcal L_v^\perp$.
Hence~$[\psi]$ is not zero in $H^1(G_{F_v}, \ad^0 \rhobar(1))/\mathcal L_v^\perp$ if and only if $\pi_\fn(\psi(\sigma)) \ne 0$ for any lift~$\sigma$ of~$\Frob_v$ to~$G_{F_v}$, and any representing cocycle~$\psi$ of~$[\psi]$. 
Note furthermore that the image of~$\pi_\fn$ coincides with $\operatorname{soc}_{k[G_F]} \ad^0 \rhobar(1)$ (by Lem.\ \ref{lem:uniserial adjoint}).

\end{rem}

\subsubsection{Local condition at~$S_{\aux}$}\label{subsec:local condition at S_aux}
Choose~$N \geq 2 M_\rho$. We write~$\cQ_N$ for the set of primes~$v$ of~$F$ such that:
\begin{enumerate}
    \item $q_v \equiv 1 \mod \varpi^{M_\rho}$ but~$q_v \not \equiv 1  \mod \varpi^{M_\rho+1}$;
    \item $\rho|_{G_{F_v}}$ is unramified.
    \item There exists $i \in \{ 1, \dots, s \}$ such that $\rho_N(\Frob_v)$ is conjugate, in $\rho_N(G_F)$, to $h g_i \diag(1, q_v) g_i^{-1} h^{-1} = 1 + (q_v - 1)E_{g_i}$ (notation as in the statement of Lem.\ \ref{lem_existence_of_well_placed_elements}; observe that $1+(q_v-1)E_{g_i} \in M_2(\cO)$ since $M_\rho>\lambda_{g_i}$).  
\end{enumerate}
Suppose that $v \in \cQ_N$. In this case, we will define, for each $n \geq 2 M_\rho$ and $1 \leq r \leq P_\rho$, local lifting sets $D_{n, v}$ and subgroups of cocycles $Z_{r, v}$, that are the main input into the lifting argument. We give a generalisation of the (specialisation to $\GL_2$ of the) data given in \cite[\S 3]{FKPI}, since we want to allow the case where the split torus of \emph{loc.\ cit.} is defined only over $E$, not necessarily over $\cO$. 

Choose $\gamma_v \in G_F$ such that 
\[
\rho_N(\gamma_v \Frob_v \gamma_v^{-1}) = 1 + (q_v - 1) E_{g_i} \text{ mod } \varpi^N,
\]
and $\alpha_v \in \GL_2(\cO)$ such that 
\[ \alpha_v (1 + (q_v - 1) E_{g_i}) \alpha_v^{-1} = \begin{pmatrix} 1 & x_v \\ 0 & q_v \end{pmatrix} \]
for some $x_v \in \cO$. We will define the data $D_{n, v}$ and $Z_{r, v}$ assuming that $\rho(\gamma_v) = \alpha_v = 1$, with the understanding that the required data in the general case is obtained simply by conjugating by $(\alpha_v \rho(\gamma_v) )^{-1}$. 

Let $s_v \coloneqq \min(\ord_\varpi x_v, \ord_\varpi (q_v - 1))$. 
We have the formula $E_{g_i} = \left ( \begin{smallmatrix} 0 & x_v / (q_v - 1) \\ 0 & 1 \end{smallmatrix} \right )$. Recalling that $\lambda_{g_i}$ is defined to be the minimal non-negative integer such that $\varpi^{\lambda_{g_i}} E_{g_i} \in M_2(\cO)$, we see that $s_v = M_\rho - \lambda_{g_i} \geq P_\rho$. In particular, the action of $G_{F_v}$ on $\ad^0 \rho_{P_\rho}$ and $\ad^0 \rho_{P_\rho}(1)$ is trivial. 

If $n \geq 2 M_\rho$, then we take $D_{n, v}$ to be the set of all representations $r : G_{F_v} \to \GL_2(\cO / \varpi^n)$ of the form
\[ r(\Frob_v) = \begin{pmatrix} 1 & 0 \\ y & 1 \end{pmatrix} \begin{pmatrix} 1 & x \\ 0 & q_v \end{pmatrix} \begin{pmatrix} 1 & 0 \\ -y & 1 \end{pmatrix}, \]
\[ r(t_v) = \begin{pmatrix} 1 & 0 \\ y & 1 \end{pmatrix} \begin{pmatrix} 1 & z \\ 0 & 1 \end{pmatrix} \begin{pmatrix} 1 & 0 \\ -y & 1 \end{pmatrix}, \]
for some $x, y, z \in \cO$ satisfying $x \equiv x_v \text{ mod }\varpi^{2M_\rho}$, $y \equiv 0 \text{ mod }\varpi^{2M_\rho - s_v}$, and $z \equiv 0 \text{ mod }\varpi^{2M_\rho}$.
We will sometimes use the notation~$r_{x, y, z}$ to emphasize the dependence of~$r$ on~$x, y, z$.

If $1 \leq r \leq P_\rho$ then we define $Z_{r, v}$ to be the $\cO$-submodule of $Z^1(F_v, \ad^0 \rho_r)$ generated by the unramified cocycles sending $\Frob_v$ to the matrices
\[ A \coloneqq \begin{pmatrix} 0 & 1 \\ 0 & 0 \end{pmatrix} \text{ and } B \coloneqq \varpi^{-s_v} \begin{pmatrix} - x_v & 0 \\ ( 1 - q_v) & x_v \end{pmatrix}, \]
and by the ramified cocycle sending $\Frob_v$ to $0$ and the generator $t_v$ of tame inertia to 
\[ C \coloneqq \begin{pmatrix} 0 & 1 \\ 0 & 0 \end{pmatrix}. \]
Since $G_{F_v}$ acts trivially on $\ad^0 \rho_{P_\rho}$, $Z_{r, v}$ is a free $\cO / \varpi^r$-module of rank 3.
\begin{lemma}\leavevmode
    \begin{enumerate}
        \item $\rho_N|_{G_{F_v}} \in D_{N, v}$. 
        \item For each $n \geq 2M_\rho$ and $1 \leq r \leq P_\rho$, reduction modulo $\varpi^n$ defines a surjection $D_{n+r, v} \to D_{n, v}$.
        \item For each $n \geq 2M_\rho$ and $1 \leq r \leq P_\rho$, the fibres of the map $D_{n+r, v} \to D_{n, v}$ are principal homogeneous spaces for the action of $Z_{r, v}$. 
        \item If $\rho' : G_{F_v} \to \GL_2(\cO)$ is a homomorphism such that $\rho' \text{ mod }\varpi^n \in D_{n, v}$ for all $n \geq 2M_\rho$, then $\rho'$ defines a point of the Steinberg lifting ring $R_v = R^{\textnormal{St}}_{\overline{\rho}|_{G_{F_v}}}$.
        \item For any $a, b \geq 1$ such that $a +b \leq P_\rho$, the natural maps induce a short exact sequence
        \[ 0 \to Z_{a, v} \to Z_{a+b, v} \to Z_{b, v} \to 0. \]
    \end{enumerate}
\end{lemma}
\begin{proof}
    The first two points are immediate from the definitions. Before proving (3), we give a more explicit description of $D_{n, v}$ for $n \geq 2M_\rho$. For $x \in x_v + \varpi^{2M_\rho} \cO / \varpi^n \cO$, $y \in \varpi^{2M_\rho - s_v} \cO / \varpi^n \cO$, and $z \in \varpi^{2M_\rho} \cO / \varpi^n \cO$, we have %
    \[ r_{x, y, z}(\Frob_v) = \begin{pmatrix} 1 - x y & x \\ (1 - q_v) y - x y^2 & q_v + x y \end{pmatrix} \]
    and
    \[ r_{x, y, z}(t_v) = \begin{pmatrix} 1 - y z & z \\ - y^2 z & 1 + y z \end{pmatrix}. \]
    We claim that for $x', y', z'$ satisfying the same conditions, we have $r_{x', y', z'} = r_{x, y, z}$ if and only if $x = x'$, $y \equiv y' \text{ mod }\varpi^{n - s_v} \cO$, and $z = z'$. In particular, we have $|D_{n, v}| = | \cO / \varpi^{n - 2M_\rho} \cO |^3$. To see the claim, suppose first that $r_{x', y', z'} = r_{x, y, z}$. Inspecting the upper right-hand entries of $\Frob_v$ and $t_v$ shows that $x = x'$ and $z = z'$. If $\ord_\varpi x_v = s_v$, then inspecting the upper left-hand entry of $\Frob_v$ shows that $x y \equiv x y' \text{ mod }\varpi^n$, hence $y \equiv y' \text{ mod }\varpi^{n - s_v}$. If $\ord_\varpi x_v > s_v$, then inspecting the lower left-hand entry of $\Frob_v$ shows that $(1-q_v) (y - y') - x (y^2 - (y')^2) \equiv 0 \text{ mod }\varpi^n$, hence $(1 - q_v - x (y + y'))(y - y') \equiv 0 \text{ mod }\varpi^n$, hence $y - y' \equiv 0 \text{ mod }\varpi^{n - s_v}$ (because in this case $\ord_\varpi(1 - q_v - x (y + y')) = s_v$). To finish the proof of the claim, we then need to show that if $x = x'$, $y \equiv y' \text{ mod }\varpi^{n - s_v}$, and $ z= z'$, then $r_{x, y, z} = r_{x', y', z'}$, but this is clear from the displayed formulae above.

    We can now prove (3). Suppose that $n \geq 2M_\rho$ and $1 \leq r \leq P_\rho$. Since $1 \leq r \leq P_\rho$ and $\rho_{P_\rho}|_{G_{F_v}}$ is trivial, we see that if $\phi \in Z^1(F_v, \ad \rho_r)$ and $r_{x, y, z} \in D_{n+r, v}$, then $(\phi \cdot r_{x, y, z})(\sigma) = r_{x, y, z}(\sigma) + \varpi^n \phi(\sigma)$. Let $A, B, C$ be the defining generating cocycles of $Z_{r, v}$. To prove (3), it's enough to show that if $a, b, c \in \cO / (\varpi^r)$, then 
    \[ (a A + b B + c C) \cdot r_{x, y, z} = r_{x + \varpi^n a, y + \varpi^{n - s_v} b, z + \varpi^n c}. \]
    Using that $2M_\rho \geq P_\rho + s_v$, we compute
    \[ r_{x + \varpi^n a, y + \varpi^{n - s_v} b, z + \varpi^n c}(\Frob_v) = r_{x, y, z}(\Frob_v) + a \varpi^n A + b \varpi^{n} B \]
    and 
    \[ r_{x + \varpi^n a, y + \varpi^{n - s_v} b, z + \varpi^n c}(t_v) = r_{x, y, z}(t_v) + c \varpi^n C, \]
    as required.

    Finally, (4) and (5) follow easily from the definitions.
\end{proof}
If~$1 \leq r \leq P_\rho$, then we write $\mathcal{L}_{r, v}$ for the image of $Z_{r, v}$ in $H^1(F_v, \ad^0 \rho_r)$. Recalling that $G_{F_v}$ acts trivially on $\ad^0 \rho_r$, we see that the following further properties hold:
\begin{itemize}
    \item $|\mathcal{L}_{v, r}| = |H^0(G_{F_v}, \ad^0 \rho_r)|$;
    \item $H^1_{\nr}(G_{F_v}, \ad^0 \rho_r)/(\mathcal{L}_{v, r} \cap H^1_{\nr}(G_{F_v}, \ad^0 \rho_r)) \cong \cO/\varpi^r$;
    \item $H^1_{\nr}(G_{F_v}, \ad^0 \rho_r(1))/(\mathcal{L}^{ \perp}_{v, r} \cap H^1_{\nr}(G_{F_v}, \ad^0 \rho_r(1))) \cong \cO/\varpi^r$.
\end{itemize}
(Compare~\cite[Lem.\ 6.7]{FKPI}.) In fact, we can make these isomorphisms more explicit. The isomorphism 
\[ H^1_{\nr}(G_{F_v}, \ad^0 \rho_r)/(\mathcal{L}_{v, r} \cap H^1_{\nr}(G_{F_v}, \ad^0 \rho_r)) \cong \cO/\varpi^r \]
may be given by the formula $\phi \mapsto \tr \phi(\Frob_v) \varpi^{\lambda_{g_i}} E_{g_i}$. The isomorphism 
\[ H^1_{\nr}(G_{F_v}, \ad^0 \rho_r(1))/(\mathcal{L}^{\perp}_{v, r} \cap H^1_{\nr}(G_{F_v}, \ad^0 \rho_r(1))) \cong \cO/\varpi^r \]
may be given by the formula $\psi \mapsto \tr \psi(\Frob_v) \varpi^{\mu_{g_i}} N_{g_i}$. We will use these expressions in killing the relative dual Selmer group below. 

\subsubsection{Selmer groups}\label{subsec:Selmer groups II}
Let~$T$ be a finite set (possibly empty) of good primes for~$\rhobar$, and assume that~$\rho_{P_\rho}$ is new at all places of~$T$.
For all~$1 \leq r \leq P_\rho$, we write
$H^1_{T}(\ad^0 \rho_r),$ resp.\ $H^1_{T}(\ad^0 \rho_r(1))$ for the Selmer group, resp.\ dual Selmer group denoted
\[
H^1_{\mathcal{L}_r}(G_{F, S \cup T}, \rho_r(\fg^{\der})), \text{ resp.\ } H^1_{\mathcal{L}_r^\perp}(G_{F, S \cup T}, \rho_r(\fg^{\der})^*)
\]
in~\cite[Section~6.1]{FKPI}.
By definition, we have 
\begin{gather*}
H^1_T(\ad^0 \rho_r) \coloneqq  \ker \left ( H^1(G_{F, S \cup T}, \ad^0 \rho_r) \to \prod_{v \in S \cup T} H^1(G_{F_v}, \ad^0 \rho_r)/\mathcal{L}_{r, v} \right )\\
H^1_T(\ad^0 \rho_r(1)) \coloneqq  \ker \left ( H^1(G_{F, S \cup T}, \ad^0 \rho_r(1)) \to \prod_{v \in S \cup T} H^1(G_{F_v}, \ad^0 \rho_r(1))/\mathcal{L}_{r, v}^\perp \right ).
\end{gather*}
Recall that if~$v \in S \cup T$, then
\[
|\mathcal{L}_{v, r}| = \left\{ \begin{array}{cc} | H^0(G_{F_v}, \ad^0 \rho_r) | & v \not\in S_p \\ |H^0(G_{F_v}, \ad^0 \rho_r) | \cdot |\cO / (\varpi^r)|^{[F_v : \mathbb{Q}_p]} & v \in S_p, \end{array}\right. 
\]
i.e.\ the local condition at~$v$ is balanced.
Thus the Greenberg--Wiles formula implies that 
\begin{equation}\label{eqn:Greenberg--Wiles}
h^1_T(\ad^0 \rho_r) = h^1_T(\ad^0 \rho_r(1)).
\end{equation}

\begin{pp}\label{prop:bounding mod p Selmer group}
    Fix~$N \geq P_\rho$.
    For all~$a \geq 0$,    
    there exists a set~$T_{a, N}$ of good primes for~$\rhobar$, such that:
\begin{enumerate}
    \item $T_{a, N}$ is disjoint from~$S,$
    \item $|T_{a, N}| = a,$
    \item $\rho_N$ is new at all places in~$T_{a, N}$, and
    \item $h^1_{T_{a, N}}(\ad^0 \rhobar) \leq h^1_\varnothing(\ad^0 \rhobar)+1.$
\end{enumerate}
If~$L/F$ is a finite Galois extension, linearly disjoint from~$F(\rho, \epsilon)$, and such that $p \nmid [L:F]$, 
the set~$T_{a, N}$ can furthermore be chosen to split completely in~$L$.
\end{pp}
\begin{proof}
    Recall that there exists a complex conjugation~$c \in G_F$ such that $\rhobar(c) = \diag(-1, 1)$.
    Let $\gamma \coloneqq  \rho_N(c)$, which is conjugate to $\diag(-1, 1)$.
    If~$v$ is unramified in~$\rho_N$ and~$\rho_N(\Frob_v) = \gamma$, then $\rho_N$ is new at~$v$, and~$v$ is automatically a good prime for~$\rhobar$.

    We now claim that, for every finite~$X$ disjoint from~$S$, and any non-zero
    \[
    [\psi] \in H^1(G_{F, S \cup X}, \ad^0 \rhobar(1)), 
    \]
    there exists a prime~$v$ disjoint from~$S \cup X$, unramified in~$\rho_N$, and such that
    \begin{enumerate}
        \item $v$ is totally split in~$L$,
        \item $\rho_N(\Frob_v) = \gamma$ (hence~$v$ is good for~$\rhobar$), and
        \item $[\psi]$ has nonzero image in $H^1(G_{F_v}, \ad^0 \rhobar(1))/\mathcal{L}_v^\perp.$
    \end{enumerate}
    Assuming the claim, 
    we prove the proposition by induction on~$a$.
    The case~$a = 0$ is immediate (with~$T_{0, N}$ being the empty set).

    Assume given~$T_{a, N}$.
    If~$h^1_{T_{a, N}}(\ad^0 \rhobar(1)) \ne 0,$ choose a nonzero $[\psi] \in H^1_{T_{a, N}}(\ad^0 \rhobar(1))$.
    Choose~$v$ such that $\rho_N(\Frob_v) = \gamma$ and $[\psi]$ is not zero in $H^1(G_{F_v}, \ad^0 \rhobar(1))/\mathcal{L}_v^\perp,$
    and let~$T_{a+1, N} \coloneqq  T_{a, N} \cup \{v\}$.
    We have the inflation-restriction exact sequence
    \[
    0 \to H^1(G_{F, S \cup T_{a, N}}, \ad^0 \rhobar(1)) \to H^1(G_{F, S \cup T_{a+1, N}}, \ad^0 \rhobar(1)) \to H^1(G_{F_v}, \ad^0 \rhobar(1))/H^1_{\nr}(G_{F_v}, \ad^0 \rhobar(1)),
    \]
    in which the final term is a one-dimensional $k$-vector space.
    By our choice of~$v$,
    \[
    H^1_{T_{a+1, N}}(\ad^0 \rhobar(1)) \cap H^1(G_{F, S \cup T_{a, N}}, \ad^0 \rhobar(1))
    \]
    is a proper subspace of 
    \[
    H^1_{T_{a, N}}(\ad^0 \rhobar(1)).
    \]
    Since it has codimension at most one in
    \[
    H^1_{T_{a+1, N}}(\ad^0 \rhobar(1)),
    \]
    we conclude that
    \[
    h^1_{T_{a+1, N}}(\ad^0 \rhobar(1)) \leq h^1_{T_{a, N}}(\ad^0 \rhobar(1))
    \]
    Hence~\eqref{eqn:Greenberg--Wiles} shows that $h^1_{T_{a+1, N}}(\ad^0 \rhobar) \leq h^1_{T_{a, N}}(\ad^0 \rhobar) \leq h^1_\varnothing(\ad^0 \rhobar)+1,$ as desired.
    
    On the other hand, if~$h^1_{T_{a, N}}(\ad^0 \rhobar(1)) = 0,$ then we choose~$v$ unramified in~$\rho_N$, totally split in~$L$, and such that $\rho_N(\Frob_v) = \gamma.$
    Then~$v$ is good for~$\rhobar$,
    and $h^1_{T_{a+1, N}}(\ad^0 \rhobar(1)) \leq 1$.
    We conclude from~\eqref{eqn:Greenberg--Wiles} that
    \[
    h^1_{T_{a+1, N}}(\ad^0 \rhobar) \leq 1 \leq h^1_\varnothing(\ad^0 \rhobar) + 1,
    \]
    as desired.
    This concludes the proof of the proposition.

    We now prove the claim.
    By Lem.\ \ref{lem:extra cohomology vanishing},
    the restriction of~$[\psi]$ to $G_{K_N}$ defines a nonzero $G_F$-equivariant homomorphism
    $f: G_{K_N} \to \ad^0 \rhobar(1).$
    Note that since $p \nmid [L:F],$ the restriction of~$f$ to $G_{K_N} \cap G_L$ is not zero.
    Choose a prime~$v$, unramified in~$\rho_N$, totally split in~$L$, and such that $\rho_N(\Frob_v) = \gamma$.
    Choose a lift~$\sigma$ of~$\Frob_v$ to~$G_{F_v}$.

    By Rem.\ \ref{rem:explicit local condition n-good primes}, 
    $[\psi]$ is not zero in $H^1(G_{F_v}, \ad^0 \rhobar(1))/\mathcal{L}_v^\perp$
    if and only if $\pi_\fn(\psi(\sigma)) \ne 0$ (for any representing cocycle~$\psi$).
    If this is true, then we are done.
    
    Otherwise, choose $\tau \in G_{K_N} \cap G_L$ %
    such that $f(\tau) \ne 0$, and 
    \[
    f(\tau) \in f(G_{K_N} \cap G_L) \cap \operatorname{soc}_{k[G_F]} \ad^0 \rhobar(1).
    \]
    As explained in Rem.\ \ref{rem:explicit local condition n-good primes}, 
    the restriction of~$\pi_\fn$ to $\operatorname{soc}_{k[G_F]} \ad^0 \rhobar(1)$ is the identity, and so~$\pi_\fn f(\tau) \ne 0$.    

    By the Chebotarev density theorem, we can find $w \not \in S \cup X,$ unramified in $L\cdot K_N(\psi)$, and such that~$\Frob_w$ has the same image as $\tau\sigma$ in~$\rho_N \times \psi \times \Gal(L/F)$.
    Then~$w$ is totally split in~$L$, $\rho_N(\Frob_w) = \gamma,$
    and $\psi(\Frob_w) = \psi(\tau \sigma) = \psi(\tau) + \psi(\sigma) = f(\tau) + \psi(\sigma)$
    has nonzero image under~$\pi_\fn$.    
    This concludes the proof of the claim, and completes the proof of the proposition. \qedhere
\end{proof}

\begin{pp}\label{prop:bounding mod p^M Selmer group}
    Fix~$N \geq P_\rho, a \geq 1$, %
    and let~$T_{a, N}$ be as in Prop.\ \ref{prop:bounding mod p Selmer group}.
    Then 
    \[
    \operatorname{length}_\cO H^1_{T_{a, N}}(\ad^0 \rho_{P_\rho}) \leq P_\rho(h^1_\varnothing(\ad^0 \rhobar)+1).
    \]
\end{pp}
\begin{proof}
    This follows from Prop.\ \ref{prop:bounding mod p Selmer group} and~\cite[Lem.\ 6.1]{FKPI}.
\end{proof}

\subsubsection{Relative Selmer groups}\label{subsec:relative Selmer groups}
We now fix positive integers~$N \geq \max\{P_\rho + 2M_\rho, P_\rho+\max_{v \in S}M_v\}$
and $a \geq 0$,
and we let $T\coloneqq T_{a, N}$ be a set produced by Prop.\ \ref{prop:bounding mod p Selmer group}.
Choose a finite subset $S_{\aux} \subset \cQ_N$ (as defined in Section~\ref{subsec:local condition at S_aux}).

For all~$1 \leq r \leq P_\rho,$
we write
$H^1_{T \cup S_{\aux}}(\ad^0 \rho_r)$ for the Selmer group denoted  
\[
H^1_{\mathcal{L}_r \cup \{\mathcal{L}_{r, v}\}_{v \in S_{\aux}}}(G_{F, S \cup T \cup S_{\aux}}, \rho_r(\fg^{\der}))
\]
in the notational scheme of~\cite{FKPI}, and similarly for $H^1_{T \cup S_{\aux}}(\ad^0 \rho_r(1)).$
We also introduce relative Selmer and dual Selmer groups:
\begin{gather*}
    \lbar{H^1_{T \cup S_{\aux}}(\ad^0 \rho_r)} \coloneqq  \image \left (H^1_{T \cup S_{\aux}}(\ad^0 \rho_r) \to H^1_{T \cup S_{\aux}}(\ad^0 \rhobar) \right )\\
    \lbar{H^1_{T \cup S_{\aux}}(\ad^0 \rho_r(1))} \coloneqq  \image \left (H^1_{T \cup S_{\aux}}(\ad^0 \rho_r(1)) \to H^1_{T \cup S_{\aux}}(\ad^0 \rhobar(1)) \right ).
\end{gather*}
Since the local conditions at $S \cup T \cup S_{\aux}$ are balanced, 
\cite[Lem.\ 6.3]{FKPI} implies that
\begin{equation}\label{eqn:relative Greenberg--Wiles}
\dim_k \lbar{H^1_{T \cup S_{\aux}}(\ad^0 \rho_r)} = \dim_k \lbar{H^1_{T \cup S_{\aux}}(\ad^0 \rho_r(1))}
\end{equation}
for all~$1 \leq r \leq P_\rho$.

\begin{pp}\label{prop:FKP double cocycle}
    Choose a finite subset $S_{\aux} \subset \cQ_N.$
    Let $\phi \in H^1_{T \cup S_{\aux}}(\ad^0 \rho_{P_\rho})$ and $\psi \in H^1_{T \cup S_{\aux}}(\ad^0 \rho_{P_\rho}(1))$
    be such that their images $\lbar \phi, \lbar \psi$ in the relative Selmer, resp.\ relative dual Selmer group are not zero.
    Then there exists~$v \in \cQ_N \setminus (S \cup T \cup S_{\aux})$ such that
    \begin{enumerate}
        \item the element $\psi|_{G_{F_v}} \text{ mod } \mathcal{L}_{P_\rho, v}^\perp \cap H^1_{\textnormal{nr}}(F_v, \ad^0 \rho_{P_\rho}(1))$ of
        \[ H^1_{\textnormal{nr}}(F_v, \ad^0 \rho_{P_\rho}(1)) /  \mathcal{L}_{P_\rho, v}^\perp \cap H^1_{\textnormal{nr}}(F_v, \ad^0 \rho_{P_\rho}(1)) \cong \cO / (\varpi^{P_\rho}) \]
        is not annihilated by $\varpi^{P_\rho - C_0 - 1}$; and
        \item the element $\phi|_{G_{F_v}} \text{ mod } \mathcal{L}_{P_\rho, v} \cap H^1_{\textnormal{nr}}(F_v, \ad^0 \rho_{P_\rho})$ of 
        \[ H^1_{\textnormal{nr}}(F_v, \ad^0 \rho_{P_\rho}) /  \mathcal{L}_{P_\rho, v} \cap H^1_{\textnormal{nr}}(F_v, \ad^0 \rho_{P_\rho}) \cong \cO / (\varpi^{P_\rho}) \]
        is not annihilated by $\varpi^{C_0}$. (See Lem.\ \ref{lem_existence_of_well_placed_elements} for the definition of~$C_0$.)
    \end{enumerate}
    If~$L/F$ is a finite Galois extension, linearly disjoint from~$F(\rho, \epsilon)$ and such that $p \nmid [L:F]$, then~$v$ can furthermore be chosen to split completely in~$L$.
\end{pp}
\begin{proof} 
    Let us use a subscript $i$ to denote image in cohomology with coefficients in $\ad^0 \rho_i$ or  $\ad^0 \rho_i(1)$ (so e.g.\ $\psi_1 \coloneqq \overline \psi$) and let $N_1 \coloneqq N + \max_j \lambda_j$. We begin by showing that $\psi_1|_{G_{F_{N_1}^\ast} \cap G_L} \neq 0$ and $\phi_{{N}_\rho}|_{G_{F_{N_1}^\ast} \cap G_L} \neq 0$. Since~$p$ is prime to $|G_{F_{N_1}^*}/(G_{F_{N_1}^*} \cap G_L)|$, it suffices to prove that $\psi_1|_{G_{F_{N_1}^\ast}} \neq 0$ and $\phi_{N_\rho}|_{G_{F_{N_1}^\ast}} \neq 0$. 
    Since 
    \[
    H^1(\Gal(F_{N_1}^*/F), \ad^0 \rhobar(1)) = 0, 
    \]
    by Lem.\ \ref{lem:definition of M}(1), the fact that $\psi_1|_{G_{F_{N_1}^\ast}} \neq 0$ follows from inflation-restriction. If $\phi_{N_\rho}|_{G_{F_{N_1}^\ast}} = 0$, then $\phi_{N_\rho}$ is inflated from $H^1(\Gal(F_{N_1}^\ast / F), \ad^0 \rho_{N_\rho})$; Lem.\ \ref{lem:definition of M}(2) then implies that $\phi_1$ must be 0, contradicting that $\phi_1$ is a nonzero element of the relative dual Selmer group. 
    
    Next, fix~$\gamma_2 \in G_{F_{N_1}^*} \cap G_L$ such that $\psi_1(\gamma_2) \neq 0$ and $\phi_{N_\rho}(\gamma_2) \ne 0$, and choose elements $X, Y \in \ad^0 \rho \setminus \varpi \ad^0 \rho$ and $m \geq 0$ such that $\psi(\gamma_2) = Y \text{ mod }\varpi^{P_\rho}$ and $\phi(\gamma_2) = \varpi^m X \text{ mod }\varpi^{P_\rho}$. 
    Then $m \leq N_\rho - 1$. By Lem.\ \ref{lem_existence_of_well_placed_elements}, we can find $1 \leq i \leq s$ such that $\ord_\varpi \tr \varpi^{\lambda_{g_i}} E_{g_i} X$ and $\ord_{\varpi} \tr \varpi^{\mu_{g_i}} N_{g_i} Y$ are both $\leq C_0$, hence $\tr \varpi^{\lambda_{g_i}} E_{g_i} \phi(\gamma_2) \in \cO / \varpi^{P_\rho}$ is not annihilated by $\varpi^{C_0}$  and  $\tr \varpi^{\mu_{g_i}} N_{g_i} \psi(\gamma_2) \in \cO / \varpi^{P_\rho}$ is not annihilated by $\varpi^{P_\rho - C_0 - 1}$. 

    By Lem.\ \ref{lem_rationally_defined_elements_in_image}, we can find $\gamma_1 \in G_L$ satisfying the following conditions:
    \begin{itemize}
        \item Writing $q \coloneqq \det \rho(\gamma_1)$, we have $q \equiv 1 \text{ mod }\varpi^{M_\rho}$ but $q \not\equiv 1 \text{ mod }\varpi^{M_\rho+1}$.
        \item $\rho(\gamma_1) = 1 + (q-1) E_{g_i}$.
    \end{itemize}
    Note that $\rho_{P_\rho}(\gamma_1)$ is then trivial, because $\varpi^{\lambda_{g_i}} E_{g_i}$ is integral and, by definition, we have $\ord_\varpi (q-1) = M_\rho \geq P_\rho + \lambda_{g_i}$. We claim that there exists $r \in \{ 0, 1, 2 \}$ such that if $g \coloneqq  \gamma_2^r \gamma_1$ then $ \tr  \varpi^{\lambda_{g_i}} E_{g_i} \phi(g)$ is not annihilated by $\varpi^{C_0}$ and $\tr \varpi^{\mu_{g_i}} N_{g_i} \psi(g) \in \cO / (\varpi^{P_\rho})$ is not annihilated by $\varpi^{P_\rho-C_0-1}$. The argument is the same as in \cite[Footnote~9]{FKPII}: we have
    \[ \phi(\gamma_2^r \gamma_1) = r \phi(\gamma_2) + \phi(\gamma_1) \]
    and
    \[ \psi(\gamma_2^r \gamma_1) = r \psi(\gamma_2) + \psi(\gamma_1). \]
    If $\tr \varpi^{\lambda_{g_i}} E_{g_i} \phi(\gamma_1)$ is not annihilated by $\varpi^{C_0}$ and $\tr \varpi^{\mu_{g_i}} N_{g_i} \psi(\gamma_1)$ is not annihilated by $\varpi^{P_\rho - C_0 - 1}$, then we're done. Suppose next that $\tr \varpi^{\lambda_{g_i}} E_{g_i}\phi(\gamma_1)$ is annihilated by $\varpi^{C_0}$ and $\tr \varpi^{\mu_{g_i}} N_{g_i} \psi(\gamma_1)$ is not annihilated by $\varpi^{P_\rho - C_0 - 1}$. Then $\tr \varpi^{\lambda_{g_i}} E_{g_i} \phi(\gamma_2^r \gamma_1)$ is not annihilated by $\varpi^{C_0}$ for either of $r = 1, 2$, and $\tr \varpi^{\mu_{g_i}} N_{g_i} \psi(\gamma_2^r \gamma_1)$ is not annihilated by $\varpi^{P_\rho - C_0 - 1}$ for at least one of these two values of $r$ (here we use that $p$ is odd). If $\tr \varpi^{\lambda_{g_i}} E_{g_i}\phi(\gamma_1)$ is not annihilated by $\varpi^{C_0}$ and $\tr \varpi^{\mu_{g_i}} N_{g_i} \psi(\gamma_1)$ is annihilated by $\varpi^{P_\rho - C_0 - 1}$, then the same argument works, reversing the roles of $\psi$ and $\phi$. Finally, if $\tr \varpi^{\lambda_{g_i}} E_{g_i} \phi(\gamma_1)$ and $\tr \varpi^{\mu_{g_i}} N_{g_i} \psi(\gamma_1)$ are both annihilated by the given powers of $\varpi$, then $r = 1$ works. 

    To complete the proof, let $L_{N_1, \psi, \phi} / F$ be the extension of $F_{N_1}^\ast$ cut out by $\operatorname{Res}(\phi) \times \operatorname{Res}(\psi)$, and choose (using the Chebotarev density theorem) a finite place $v \not\in S \cup T \cup S_{\textnormal{aux}}$ of $F$ such that the image of $\Frob_v$ in $\Gal(L_{{N_1}, \psi, \phi} / F)$ is conjugate to the image of $g$. Then we have $q_v \equiv q \text{ mod }\varpi^{N + \lambda_i}$, hence $1 + (q-1) E_{g_i} \equiv 1 + (q_v - 1)E_{g_i} \text{ mod }\varpi^N$. This shows that $v \in \cQ_N$; the remaining required properties of $v$ hold by construction. 
    \end{proof}

\begin{pp}\label{prop:FKP relative Selmer group}
    Let $N \geq \max\{P_\rho + 2M_\rho, P_\rho + \max_{v \in S} M_v\}$, let~$a \geq 0$, and let $T := T_{a, N}$ be a set produced by Prop.\ \ref{prop:bounding mod p Selmer group}. 
    Then there exists a finite subset $S_{\aux} \subset \cQ_N$, of size at most
    \[
    \length_\cO H^1_T(\ad^0 \rho_{P_\rho}), 
    \]
    such that
    \begin{gather*}
        \lbar{H^1_{T \cup S_{\aux}}(\ad^0 \rho_{P_\rho})} = 0\\
        \lbar{H^1_{T \cup S_{\aux}}(\ad^0 \rho_{P_\rho}(1))} = 0.
    \end{gather*}
\end{pp}
\begin{proof}
    The proof is similar to the proof of ~\cite[Thm.\ 6.9]{FKPI}, using Prop.\ \ref{prop:FKP double cocycle}. By induction, and formula \eqref{eqn:relative Greenberg--Wiles}, it suffices to show that for any finite subset $S_{\aux} \subset \cQ_N$ such that  
    \[
    \overline{H^1_{T \cup S_{\aux}}(\ad^0 \rho_{P_\rho}(1))} \neq 0,
    \]
    we can find $v \in \cQ_N \setminus (T \cup S_{\aux})$ such that
     \[
    \length_\cO H^1_{T \cup S_{\aux} \cup \{v\}}(\ad^0 \rho_{P_\rho}(1))
    < \length_\cO H^1_{T \cup S_{\aux}}(\ad^0 \rho_{P_\rho}(1)).
    \]
    Suppose then that $\overline{H^1_{T \cup S_{\aux}}(\ad^0 \rho_{P_\rho}(1))}$ (and hence $\overline{H^1_{T \cup S_{\aux}}(\ad^0 \rho_{P_\rho})}$) is nonzero, and let $\psi, \phi$ be classes as in Prop.\ \ref{prop:FKP double cocycle} with nonzero images in the relative Selmer groups, and let $v$ be the place provided by that Proposition.

    Let $H^1_{\mathcal{L}}(\ad^0 \rho_{P_\rho})$ be the Selmer group with the same local conditions as those defining  $H^1_{T \cup S_{\aux}}(\ad^0 \rho_{P_\rho})$ away from $v$, and with condition $H^1_{\textnormal{nr}}(F_v, \ad^0 \rho_{P_\rho}) + \mathcal{L}_{P_\rho, v}$ at the place $v$; and let $H^1_{\mathcal{L}^\perp}(\ad^0 \rho_{P_\rho}(1))$ be the corresponding dual Selmer group. Then the Greenberg--Wiles formula gives the equality
    \[ \length_\cO H^1_{\mathcal{L}}(\ad^0 \rho_{P_\rho}) - \length_\cO H^1_{\mathcal{L}^\perp}(\ad^0 \rho_{P_\rho}(1)) = P_\rho. \]
    We have an inclusion
    \[ H^1_{\mathcal{L}^\perp}(\ad^0 \rho_{P_\rho}(1)) \leq H^1_{T \cup S_{\aux}}(\ad^0 \rho_{P_\rho}(1)), \]
    and the defining property of $\psi$ implies that we must have
    \[ \length_\cO H^1_{\mathcal{L}^\perp}(\ad^0 \rho_{P_\rho}(1)) \leq \length_\cO H^1_{T \cup S_{\aux}}(\ad^0 \rho_{P_\rho}(1)) - (P_\rho - C_0), \]
    hence 
    \[ \length_\cO H^1_{\mathcal{L}}(\ad^0 \rho_{P_\rho}) \leq \length_\cO H^1_{T \cup S_{\aux}}(\ad^0 \rho_{P_\rho})+ C_0. \]
    On the other hand, we have an inclusion
    \[ H^1_{T \cup S_{\aux} \cup \{ v \}}(\ad^0 \rho_{P_\rho}) \leq H^1_{\mathcal{L}}(\ad^0 \rho_{P_\rho}), \]
    with $\phi$ contained in the right-hand group, hence
    \begin{multline*} \length_\cO H^1_{T \cup S_{\aux} \cup \{v  \}}(\ad^0 \rho_{P_\rho}) \leq \length_\cO H^1_{\mathcal{L}}(\ad^0 \rho_{P_\rho}) - (C_0 + 1) \\  \leq \length_\cO H^1_{T \cup S_{\aux} }(\ad^0 \rho_{P_\rho}) + C_0 - (C_0 + 1) \\  = \length_\cO H^1_{T \cup S_{\aux}}(\ad^0 \rho_{P_\rho}) - 1. 
    \end{multline*}
    This completes the proof. 
\end{proof}

\subsection{Proof of Prop.\ \ref{prop:T-new lift}}
Continue to let~$N$ and~$T \coloneqq  T_{a, N}$ be as in the statement of Prop.\ \ref{prop:FKP relative Selmer group}, 
and let~$S_{\aux} \subset \cQ_N$ be a set produced by Prop.\ \ref{prop:FKP relative Selmer group}. 
Then
\[
|S_{\aux}| \leq \operatorname{length}_\cO H^1_T(\ad^0 \rho_{P_\rho}) \leq P_\rho(h^1_\varnothing(\ad^0 \rhobar(1))+1),
\]
where the second inequality is a consequence of Prop.\ \ref{prop:bounding mod p^M Selmer group}. 
So the sets~$S_{\aux}$ %
and~$T$ satisfy the requirements of Prop.\ \ref{prop:T-new lift}. 
We now construct the required lift~$\rho'$ of~$\rho_{N-P_\rho}$, following the proof of~\cite[Claim~6.14]{FKPI}.
Recall that for all~$v \in S \cup T \cup S_{\aux}$
we have specified:
\begin{itemize}
\item for all~$n \geq N-P_\rho$, a set $D_{n, v}$ of lifts of $\rhobar |_{G_{F_v}}$ to~$\cO/\varpi^n$, and
\item for all $1 \leq r \leq P_\rho$, an $\cO$-submodule $Z_{r, v} \subset Z^1(G_{F_v}, \ad^0 \rho_r)$, containing $B^1(G_{F_v}, \ad^0 \rho_r)$, with image $\mathcal{L}_{r, v} \subset H^1(G_{F_v}, \ad^0 \rho_r)$,
\item an integral quotient~$R_v$ of $R^\square_{\rhobar|_{G_{F_v}}}$ (which is $R^{\St, 1}_{\rhobar|_{G_{F_v}}}$ if $v \in T,$ and $R^{\St}_{\rhobar|_{G_{F_v}}}$ if $v \in S_{\aux}$) 
\end{itemize}
such that:
\begin{itemize}
\item for all $1 \leq r \leq P_\rho$,
reduction modulo~$\varpi^n$ defines a function $D_{n+r, v} \to D_{n, v}$, whose fibres above lifts of $\rho_{2P_\rho}|_{G_{F_v}}$ 
are nonempty and stable under the translation action of~$Z_{r, v}$; 
\item $\rho_{N}|_{G_{F_v}} \in D_{N, v}$;
\item every element of $D_{n, v}$ defines a $\cO/\varpi^n$-point of~$R_v$.
\end{itemize}
For each~$n \geq N-P_\rho$, we now construct a pair~$(\tau_n, \tld \rho_{n+P_\rho})$ of homomorphisms $\tau_n : G_{F, S \cup T \cup S_{\aux}} \to \GL_2(\cO/\varpi^n)$ and
$\tld \rho_{n+P_\rho} : G_{F, S \cup T \cup S_{\aux}} \to \GL_2(\cO/\varpi^{n+P_\rho})$,
with determinant~$\epsilon^{-1}$, and such that for all~$n \geq N-P_\rho$ the following are true:
\begin{enumerate}
    \item $\tau_n = \tld\rho_{n+P_\rho} \otimes_{\cO/\varpi^{n+P_\rho}} \cO/\varpi^n,$
    \item for all $v \in S \cup T \cup S_{\aux}$ we have $\tau_n |_{G_{F_v}} \in D_{n, v}$ and $\tld \rho_{n+P_\rho} |_{G_{F_v}} \in D_{n+P_\rho, v}$, and furthermore $\tau_n |_{G_{F_v}}$ and $\tld \rho_{n+P_\rho} |_{G_{F_v}}$ are lifts of $\rho_{2P_\rho}|_{G_{F_v}}$;
    \item $\tau_{n+1} \otimes_{\cO/\varpi^{n+1}} \cO/\varpi^n = \tau_n.$
\end{enumerate}
Furthermore, we will also have
$\tau_{N-P_\rho} = \rho_{N-P_\rho}$,
so that putting $\rho' \coloneqq  \varprojlim_{n \geq N-P_\rho} \tau_{n},$ 
and bearing in mind Lem.\ \ref{lem:deducing bound on image of rho},
we will have proved Prop.\ \ref{prop:T-new lift}.

We begin the construction by putting $\tau_{N-P_\rho} \coloneqq  \rho_{N-P_\rho}$ and $\tld \rho_N \coloneqq  \rho_N$, which guarantees that the first two properties are true for the index~$N-P_\rho$.

Assume now that $n \geq N-P_\rho$ and that we are given~$(\tau_n, \tld \rho_{n+P_\rho})$ such that the first two properties are true for the index~$n$.
Define
\[
\tau_{n+1} \coloneqq  \tld \rho_{n+P_\rho} \otimes_{\cO/\varpi^{n+P_\rho}} \cO/\varpi^{n+1}.
\]
Then property~(3) is true for the index~$n$.
It thus suffices to construct a lift~$\tld \rho_{n+1+P_\rho}$ of~$\tau_{n+1}$ to~$\cO/\varpi^{n+1+P_\rho},$
such that $\tld \rho_{n+1+P_\rho}|_{G_{F_v}} \in D_{n+1+P_\rho, v}$ for all~$v \in S \cup T \cup S_{\aux}$.
Then properties~(1) and~(2) will be true for the index~$n+1$, and we will be able to continue the construction inductively.

To construct~$\tld \rho_{n+1+P_\rho},$ we begin by showing that there exists a lift
\[
\tld \rho'_{n+1+P_\rho} : G_{F, S \cup T \cup S_{\aux}} \to \GL_2(\cO/\varpi^{n+P_\rho+1})
\]
of~$\tld \rho_{n+P_\rho}$.
Since reduction modulo~$\varpi^{n+P_\rho}$
defines, for every $v \in S \cup T \cup S_{\aux}$, 
a function $D_{n+1+P_\rho, v} \to D_{n+P_\rho, v}$
whose fibre above~$\tld \rho_{n+P_\rho}|_{G_{F_v}}$ is not empty,
the local obstruction class 
\[
\operatorname{cl}(\tld \rho_{n+P_\rho}|_{G_{F_v}}) \in H^2({F_v}, \ad^0 \rhobar)
\]
vanishes.
Hence the existence of~$\tld \rho'_{n+1+P_\rho}$ is a consequence of 
\[
\Sha^2_{S \cup T \cup S_{\aux}}(F_{S \cup T \cup S_{\aux}}/F, \ad^0 \rhobar) = 0
\]
which is a consequence of~\eqref{eqn:vanishing of Sha1}.

The restrictions $\tld \rho'_{n+1+P_\rho}|_{G_{F_v}}$ need not be elements of~$D_{n+1+P_\rho, v}$.
However, we are going to find some~$f \in Z^1(F_{S \cup T \cup S_{\aux}}, \ad^0 \rho_{P_\rho})$
such that $f \cdot \tld \rho'_{n+1+P_\rho}|_{G_{F_v}} \in D_{n+1+P_\rho, v}$
for all~$v \in S \cup T \cup S_{\aux}$.
Since $f \cdot \tld \rho'_{n+1+P_\rho}$ is still a lift of~$\tau_{n+1}$, this will conclude the proof.

To find~$f$, we begin by choosing preimages $\tld r_v$ of~$\tld \rho_{n+P_\rho}|_{G_{F_v}}$ under $D_{n+1+P_\rho, v} \to D_{n+P_\rho, v}$.
We then let~$f_v \in Z^1(F_v, \ad^0 \rhobar)$ be the cocycle such that $f_v \cdot  \tld \rho'_{n+1+P_\rho}|_{G_{F_v}} = \tld r_v$.
We claim that $([f_v]: v \in S \cup T \cup S_{\aux})$ is contained in the image of
\[
H^1(F_{S \cup T \cup S_{\aux}}/F, \ad^0 \rho_{P_{\rho}}) \to \prod_{v \in S \cup T \cup S_{\aux}}H^1(F_v, \ad^0 \rho_{P_\rho})/\mathcal{L}_{P_\rho, v}.
\]
Given the claim, we let~$f$ be a representative of a preimage of $([f_v]: v \in S \cup T \cup S_{\aux})$.
Then there exist~$z_v \in Z_{P_\rho, v}$ such that
\[
(f \cdot \tld \rho'_{n+1+P_\rho})|_{G_{F_v}} = z_v\cdot \tld r_v.
\]
On the other hand, $z_v\cdot \tld r_v \in D_{n+1+P_\rho, v}$,
because the fibres of $D_{n+1+P_\rho, v} \to D_{n+1, v}$ above lifts of $\rho_{2P_\rho}|_{G_{F_v}}$ are acted on by $Z_{P_\rho, v}$.
This concludes the proof of Prop.\ \ref{prop:T-new lift}.

There remains to prove the claim, which can be done by the same diagram chase as in the proof of~\cite[Claim~6.14]{FKPI}.
In more detail, the Poitou--Tate exact sequence implies that there is a commutative diagram with exact rows
\[
\resizebox{\textwidth}{!}{\xymatrix{H^1(F_{S \cup T \cup S_{\aux}} / F, \ad^0 \rho_{P_\rho}) \ar[r] \ar[d] & \prod_{v \in S \cup T \cup S_{\aux}}H^1(F_v, \ad^0 \rho_{P_{\rho}})/\mathcal{L}_{{P_\rho}, v} \ar[r] \ar[d] & H^1_{T \cup S_{\aux}}(\ad^0 \rho_{P_{\rho}}(1))^\vee \ar[d]\\
H^1(F_{S \cup T \cup S_{\aux}} / F, \ad^0 \rho_{P_\rho-1}) \ar[r] & \prod_{v \in S \cup T \cup S_{\aux}}H^1(F_v, \ad^0 \rho_{P_{\rho}-1})/\mathcal{L}_{{P_\rho-1}, v} \ar[r] & H^1_{T \cup S_{\aux}}(\ad^0 \rho_{P_{\rho}-1}(1))^\vee}}
\]
The third vertical arrow is injective, by Prop.\ \ref{prop:FKP relative Selmer group} (which, together with~\cite[Lem.\ 6.1]{FKPI}, shows surjectivity of the dual of the third vertical arrow).
Since $f_v \in Z^1(F_v, \ad^0 \rhobar_v),$ the class $([f_v] : v \in S \cup T \cup S_{\aux})$ maps to zero under the second vertical arrow.
It is therefore contained in the image of $H^1(F_{S \cup T \cup S_{\aux}} / F, \ad^0 \rho_{P_\rho}),$ as desired.

\section{Level-lowering mod~$\varpi^N$}\label{sec_level_lowering}
In this section we combine the work done in Sections \ref{sec_an_ALT} and \ref{sec: construction of T-new lifts} with a new level-lowering argument in order to prove an automorphy lifting result for $p$-adic Galois representations without any hypothesis of auxiliary ramification. (There are some additional simplifying conditions, as well as a residual automorphy hypothesis; our final task, undertaken in Section \ref{sec_deduction_of_main_theorem}, will be to reduce the general case to this one by base change, and show that residual automorphy always holds.) 

Our level-lowering argument is based on Mazur's principle, but goes through even when $\rhobar^{\textnormal{ss}}$ is reducible and $q_v$ (the size of the residue field at the place $v$ of interest) is congruent to 1 mod $p$. Key roles are played by the fact that as long as $\rhobar^{\textnormal{ss}}$ is not a twist of $1 \oplus \epsilon^{-1}$, the $H^0$ and $H^2$ of the Shimura curves that intervene still vanish (a weaker condition that the usual ``non-Eisenstein'' condition); and by Prop.\ \ref{prop_pseudocharacters_with_large_image}, which is a purely ``pseudocharacter-theoretic'' result that can be seen as a generalisation of \cite{BLR} beyond the residually irreducible case.

\subsection{Statement of main result}\label{subsec:statement of level-lowering}
Our main result in this section is as follows. 
\begin{pp}\label{prop:level-lowering}
    Let~$F$ be a totally real field of even degree. Let $p$ be an odd prime, let $E / \mathbb{Q}_p$ be a coefficient field, let $\rho : G_F \to \GL_2(E)$ be continuous with Zariski dense image, and let $\lambda \in (\bZ_+^2)^{\Hom(F, E)}$ be a weight. If the following conditions are satisfied, then~$\rho$ is modular:
    \begin{enumerate}
        \item 
        There is an isomorphism $\rhobarss = \overline \chi_1 \oplus \overline \chi_2$ for some characters $\overline \chi_i: G_F \to k^\times$.
        \item Writing $\overline \chi \coloneqq  \overline \chi_1/\overline \chi_2,$ we have $\overline \chi |_{G_{F(\zeta_p)}} \ne 1$.
        \item There exists a finite set~$S_0$ of finite places of~$F$, such that
        \begin{enumerate}
            \item $\rho$ is unramified away from $S_p \cup S_0$;
            \item for all~$v \in S_0$, $q_v \equiv 1 \text{ mod } p,$ $\rhobarss |_{G_{F_v}}$ is trivial, and $\rho|_{G_{F_v}}^{\mathrm{ss}} = 1 \oplus \epsilon^{-1}$;
            \item for all $v \in S_p$, $\rhobarss |_{G_{F_v}}$ is trivial, $[F_v:\bQ_p] > 2$, and $\rho|_{G_{F_v}}$ is ordinary and semistable of weight~$\lambda_v$.
            Moreover, if we write $\rho |_{G_{F_v}} \sim \left( \begin{smallmatrix} \psi_{v, 1} & * \\ 0 & \psi_{v, 2} \end{smallmatrix} \right)$ for characters $\psi_{v, i} : G_{F_v} \to E^\times$, as in the definition of $\lambda_v$-ordinary, then $\psi_{v, 1} \circ \Art_{F_v}$ is trivial on the subgroup of $p$-power roots of unity of $\cO_{F_v}^\times$.
        \end{enumerate}
        \item $\det \rho = \epsilon^{-1}$.
        \item (Residual modularity at minimal level) There exists a cuspidal, regular algebraic automorphic representation~$\pi$ of~$\GL_2(\bA_F)$, of weight~$0$ and trivial central character, and an isomorphism $\iota: \overline{\bQ}_p \to \mathbb{C}$, %
        all satisfying the following conditions:
        \begin{enumerate}
            \item There is an isomorphism $\overline{r}_{\pi, \iota}^{\textnormal{ss}} \cong \rhobarss$. 
            \item $\pi$ is $\iota$-ordinary.
            \item $\pi$ is unramified outside $S_p$, and semistable at~$S_p$.
        \end{enumerate}
    \end{enumerate}
\end{pp}

\subsection{Preliminaries}\label{subsec_preliminaries_to_level_lowering}
By Ribet's lemma, there exists an $\cO[G_F]$-lattice $\mathcal{L} \subset E^{\oplus 2}$ such that $\mathcal{L} \otimes_\cO k$ is a nonsplit extension of $\overline{\chi}_2$ by~$\overline{\chi}_1$.
By Lem.\ \ref{lem_ordinary_implies_split_image},
we can choose an $\cO$-basis of~$\mathcal{L}$ in such a way that $\rho(G_F) \subset \GL_\cO(\mathcal{L}) \isom \GL_2(\cO)$
contains a $\GL_2(E)$-conjugate of an open subgroup of~$\GL_2(\bZ_p)$.
We will apply the framework of Section~\ref{sec: construction of T-new lifts} to the resulting homomorphism $\rho : G_F \to \GL_2(\cO)$,
which satisfies Assumptions~(1)--(4) of Section~\ref{subsubsec:assumptions on Galois representation}. 

\subsubsection{Construction of auxiliary lifts}
We fix from the outset a finite place~$\mathfrak{w}$ of~$F$, disjoint from~$S_p \cup S_0$, and such that $q_{\mathfrak{w}} > 4^{[F:\mathbb Q]}$.
We will adopt notation from Section~\ref{sec: construction of T-new lifts} without further comment.
Amongst other things, this means that
we write $\rhobar \coloneqq \mathcal{L} \otimes_{\cO} k$ (this is consistent with our usage of $\rhobarss$),
and that
we choose a finite set~$S_\Sha$, disjoint from~$S_0 \cup S_p$, such that: 
\begin{itemize}
    \item For all~$v \in S_{\Sha}$, $q_v \equiv 1 \text{ mod } p$, and $\rhobar |_{G_{F_v}}$ is trivial;
    \item We have $\Sha^1_{S_0 \cup S_p \cup S_{\Sha}}(G_{F, S_0 \cup S_p \cup S_{\Sha}}, \ad^0 \rhobar(1)) = 0$.
\end{itemize}
As explained in Rem.\ \ref{rem:intended choice of components}, we make the following choice of local lifting rings:
\begin{itemize}
\item If~$v \in S_p$, then $R_v$ is an ordinary component of $R^{\lambda_v, \mathrm{st}}_{\rhobar|_{G_{F_v}}}$, containing $\rho |_{G_{F_v}},$ and such that for every $\overline{\mathbb{Q}}_p$-point~$\rho'_v$ of~$\Spec R_v$, writing $\rho'_v \sim \left( \begin{smallmatrix} \psi'_{v, 1} & * \\ 0 & \psi'_{v, 2} \end{smallmatrix} \right)$ for characters $\psi'_{v, i} : G_{F_v} \to \overline{\mathbb{Q}}_p^\times$, as in the definition of $\lambda_v$-ordinary, then $\psi'_{v, 1} \circ \Art_{F_v}^\times$ is trivial on the subgroup of $p$-power roots of unity of $\cO_{F_v}^\times$.
\item If~$v \in S_0$, then $R_v$ is the Steinberg lifting ring of~$\rhobar|_{G_{F_v}}$.
\item If~$v \in S_{\Sha}$, then $R_v$ is the unramified lifting ring of $\rhobar|_{G_{F_v}}$.
\end{itemize}
For~$N$ large enough (to be specified below) we would like to apply Prop.\ \ref{prop:T-new lift} to~$\rho_N \coloneqq  \rho \otimes_\cO \cO/\varpi^N$ (with the choice of~$R_v$ just described), and to apply Thm.\ \ref{thm_tame_SW_black_box} to the resulting lift~$\rho'$.
Note that Prop.\ \ref{prop:T-new lift} will then produce lifts unramified at~$S_{\Sha}$, and for this reason, in the rest of this section we will write, for simplicity, $S \coloneqq  S_p \cup S_0$ (as opposed to $S\coloneqq  S_0 \cup S_p \cup S_{\Sha},$ which is the convention in force in Section~\ref{sec: construction of T-new lifts}).

Since Thm.\ \ref{thm_tame_SW_black_box} comes with a parity assumption on $[F:\bQ]+|T \cup S_{\aux}|$, we will also fix an auxiliary totally real quadratic extension $L/F$, not contained in $F(\rho, \epsilon)$, 
and such that~$S_p \cup S_0$ is totally split in~$L$. 
For every set~$X$ of finite places of~$F$, we write $\widetilde X$ for the set of places of~$L$ above~$X$. %
We will actually apply Thm.\ \ref{thm_tame_SW_black_box} to $\rho'|_{G_L},$ and deduce our modularity results via base change.

We introduce the following quantities, depending only on~$\rho$ and~$S$: 
\begin{itemize}
\item $s_{\rho} \coloneqq  P_\rho(h^1_\varnothing(\ad^0 \rhobar)+1) + 1$, 
\item $a \coloneqq h^1(L_{\widetilde S}/L, \chibar) + 2s_\rho + 2$,  %
\item $t_\rho \coloneqq  P_\rho + 4 C_\rho (s_\rho+a)+2M_\rho s_\rho$.
\end{itemize}
We choose~$N$ such that 
\begin{equation}\label{eqn:bound on N for level-lowering}
N > \max\{P_\rho + 2M_\rho, P_\rho+\max\nolimits_{v \in S_0 \cup S_p \cup S_{\Sha}}M_v, P_\rho+C_\rho, t_\rho\}.
\end{equation}
We now apply Prop.\ \ref{prop:T-new lift} to these choices of~$\rho, N, a$.
We obtain sets~$T, S_{\aux},$ disjoint from~$S$, and a lift
\[
\rho' : G_{F, S \cup T \cup S_{\aux}} \to \GL_2(\cO)
\]
of~$\rho_{N-P_\rho}$, such that the following are true:
\begin{enumerate}
    \item $\det \rho' = \epsilon^{-1}$;
    \item For all~$v \in S_p$: $\rho' |_{G_{F_v}}$ is ordinary of weight~$\lambda_v$, and if we write $\rho'|_{G_{F_v}} \sim \left( \begin{smallmatrix} \psi'_{v, 1} & * \\ 0 & \psi'_{v, 2} \end{smallmatrix} \right)$ for characters $\psi'_{v, i} : G_{F_v} \to \overline{\mathbb{Q}}_p^\times$, as in the definition of $\lambda_v$-ordinary, then $\psi'_{v, 1} \circ \Art_{F_v}^\times$ is trivial on the subgroup of $p$-power roots of unity of $\cO_{F_v}^\times$.
    \item For all $v \in T$: $q_v \equiv -1 \text{ mod }p$, $(\rho'|_{G_{F_v}} \otimes_\cO E)^{\textnormal{ss}} \cong 1 \oplus \epsilon^{-1}$, and $\rhobar (\Frob_v)$ is conjugate to $\diag(-1, 1)$ within~$\rhobar(G_F)$. 
    \item For all $v \in S_0$: $q_v \equiv 1\text{ mod }p$, $\overline{\rho}|_{G_{F_v}}$ is trivial, and $(\rho'|_{G_{F_v}} \otimes_\cO E)^{\textnormal{ss}} \cong 1 \oplus \epsilon^{-1}$.
    \item For all $v \in S_{\aux}$: $q_v \equiv 1\text{ mod } \varpi^{M_\rho}$, $q_v \not \equiv 1\text{ mod } \varpi^{M_\rho+1}$, $\overline{\rho}|_{G_{F_v}}$ is trivial, and $(\rho'|_{G_{F_v}} \otimes_\cO E)^{\textnormal{ss}} \cong 1 \oplus \epsilon^{-1}$.
    \item $|S_{\aux}| \leq s_\rho$.
    \item $|T| = a$.
    \item $\rho'(\cO[G_F])$ contains $\varpi^{C_\rho} M_2(\cO)$.
    \item The sets~$T$ and~$S_{\aux}$ are totally split in~$L$, and $\mathfrak{w} \not \in S \cup T \cup S_{\aux}$.
\end{enumerate}
For simplicity, we will write $Z \coloneqq  T \cup S_{\aux}.$
For all~$v \in Z$, we will write $m_v \coloneqq  \ord_\varpi(q_v-1)$; 
note that $m_v = 0$ if~$v \in T$, and $m_v = M_\rho$ if~$v \in S_{\aux}$. 

\subsubsection{Level-raising}
We prove a level-raising result, which will be used in Section~\ref{subsec:proof of level lowering} to show that assumption~(6) of Thm.\ \ref{thm_tame_SW_black_box} is satisfied by $\rho'|_{G_{L}}$.

\begin{lemma}\label{lem_raising_the_level}
    Let $F$ be a totally real field with $[F : \mathbb{Q}]$ even, let $p$ be an odd prime, and  let $\iota : \overline{\mathbb{Q}}_p \to \mathbb{C}$ be an isomorphism. Suppose given  a cuspidal, regular algebraic automorphic representation $\pi$ of $\GL_2(\mathbb{A}_F)$ of weight 0 and trivial central character satisfying the following conditions:
    \begin{enumerate}
        \item $\pi$ is $\iota$-ordinary.
        \item $\pi$ is unramified away from $S_p$, and semistable at the $p$-adic places. 
        \item There is an isomorphism $\overline{r}_{\pi, \iota}^{\textnormal{ss}} \cong \overline{\chi}_1 \oplus \overline{\chi}_2$, where $(\overline{\chi}_1 / \overline{\chi}_2)|_{G_{F(\zeta_p)}}$ is non-trivial.
    \end{enumerate}
    Let $Z$ be a finite set of finite, prime-to-$p$ places of $F$ such that $|Z|$ is even
    and for each $v \in Z$, $q_v^2 \equiv 1 \text{ mod }p$ and  there is an isomorphism $\overline{r}_{\pi, \iota}|^{\textnormal{ss}}_{G_{F_v}} \cong 1 \oplus \epsilon^{-1}$. Then we can find a cuspidal, regular algebraic automorphic representation $\pi_Z$ of $\GL_2(\mathbb{A}_F)$ of weight 0 and trivial central character satisfying the following conditions:
    \begin{enumerate}
        \item $\pi_Z$ is $\iota$-ordinary.
        \item $\pi_Z$ is unramified away from $S_p \cup Z$, and semistable at the $p$-adic places.
        \item For each place $v \in Z$, $\pi_{Z, v}$ is isomorphic to the Steinberg representation. 
        \item There is an isomorphism $\overline{r}_{\pi_Z, \iota}^{\textnormal{ss}} \cong \overline{r}^{\textnormal{ss}}_{\pi, \iota}$.
    \end{enumerate}
\end{lemma}
\begin{proof}
    Let $D$ denote the quaternion algebra over $F$ ramified at $Z$ and the infinite places, and let $m_0$ be the exponent of $p$ in the least common multiple of the integers $|F^\times \backslash U_0(1, 0) \cap t^{-1} D^\times t|$, $t \in (D \otimes_F \mathbb{A}_F^\infty)^\times$. We first choose a place $u$ of $F$ satisfying the following conditions:
    \begin{itemize}
        \item $u \not\in S_p \cup Z$.
        \item $q_u \equiv 1 \text{ mod } p^{1 + m_0}$. 
        \item $q_u > 4^{[F : \mathbb{Q}]}$. 
        \item $\overline{\chi}_1(\Frob_u) \neq \overline{\chi}_2(\Frob_u)$. 
    \end{itemize}
    We claim that for each subset $W \subset Z$ of even cardinality there exists a cuspidal, regular algebraic automorphic representation $\Pi_W$ of weight 0 satisfying the following conditions:
    \begin{itemize}
        \item $\Pi_W$ is $\iota$-ordinary.  
        \item $\Pi_W$ is unramified away from $S_p \cup W \cup \{ u \}$, and semistable at the $p$-adic places.
        \item For each place $v \in W$, $\Pi_{W, v}$ is isomorphic to an unramified twist of the Steinberg representation. If $q_v \equiv -1 \text{ mod }p$, then the $U_v$-eigenvalue of $\Pi_{W, v}^{\Iw_v}$ is congruent to $1 \text{ mod }p$.
        \item $\Pi_{W, u}^{\Iw_{u,1 }} \neq 0$. 
        \item $\overline{r}_{\Pi_W, \iota} \cong \overline{r}_{\pi, \iota}$.
    \end{itemize}
    We first explain why the claim implies the statement of the lemma. The representation $\Pi_Z$ has central character of finite, $p$-power order (because $\det \overline{r}_{\Pi_W, \iota} = \epsilon^{-1}$). Since $p$ is odd, we can twist by the reciprocal square-root of this character and assume, without disturbing any of the other conditions, that the central character of $\Pi_Z$ is trivial. We then refer to the statement of Prop.\ \ref{prop_freeness_of_modular_forms_at_TW_level}. The existence of $\Pi_Z$, with the Jacquet--Langlands correspondence, implies that $H(U_{\{ u \}(1), 1}(1, 1))_{\mathfrak{m}_{\{ u \}}}$ is non-zero (with $T$ there taken to be the set of places $v \in Z$ with $q_v \equiv -1 \text{ mod }p$, and $\mathfrak{m}_{\{u\}}$ the maximal ideal of the Hecke algebra corresponding to the Hecke eigenvalues of $\Pi_Z$). Then Prop.\  \ref{prop_freeness_of_modular_forms_at_TW_level}(2) implies that $H(U_{\{ u \}}(1, 1))_{\mathfrak{m}_{\{ u \}}} \neq 0$. The Jacquet--Langlands correspondence now implies the existence of an automorphic representation $\pi_Z$ with the desired properties (noting that $H(U_{\{ u \}}(1, 1))_{\mathfrak{m}_{\{ u \}}}$ only sees cusp forms, because $\overline{\chi}_1 / \overline{\chi}_2 \neq \epsilon^{\pm 1}$). 

    We now establish the claim, by induction on $|W|$. The case $W = \varnothing$ holds by hypothesis. In general, suppose that $\Pi_W$     with the claimed properties exists, and let $v_1, v_2 \in Z - W$ be distinct. We can now copy the proof of \cite[Lem.\ 4.11]{Tho16}. First, using Ihara's lemma for the quaternion algebra ramified at the infinite places and the elements of $W$, we can raise the level at $v_1$ to obtain an automorphic representation $\Pi_{W \cup \{ v_1 \}}$ having the given properties. (It is supposed in \emph{loc.\ cit.} that $q_v \not\equiv 1 \text{ mod }p$, but this is just to force the residual eigenvalues of Frobenius to be distinct; we can allow this case too without changing the argument.) The essential point here is that the kernel of the map $i$ of \emph{loc.\ cit.} is supported above maximal ideals of the Hecke algebra with Galois representations that are a twist of $1 \oplus \epsilon^{-1}$, a stronger statement than in the statement of \cite[Lem.\ 3.28]{Ski99}, but which follows from the proof and which is recorded explicitly e.g.\ in the proof of \cite[Prop.\ 4.4.4]{Pan22}. Second, arguing as in the proof of \cite[Thm.\ 5]{Raj01}, we can raise the level to produce $\Pi_{W \cup \{ v_1, v_2 \}}$ from $\Pi_{W \cup \{ v_1 \}}$. The point here is again that \cite[Prop.\ 5]{Raj01} (``component group is Eisenstein'') can again be refined to ``component group is a twist of trivial'', cf.\ Cor.\ \ref{cor:component group vanishes} below. 
    This completes the proof.\qedhere %
\end{proof}

\subsubsection{Modularity of auxiliary lifts}
We are going to apply Thm.\ \ref{thm_tame_SW_black_box} to~$\rho'|_{G_L}$, with the set there denoted~$S_{\aux}$ being our $\widetilde S_0 \cup \widetilde S_{\aux}$, and the set there denoted~$T$ being our $\widetilde T$.
We verify its assumptions: 
\begin{itemize}
    \item (1) holds because $\overline \chi|_{G_{F(\zeta_p)}} \ne 1$, and $G_F \to \Gal(L/F) \times \Gal(F(\rho, \epsilon)/F)$ is surjective.
    \item (2) and (3) hold because~$S_0, S_p, T, S_{\aux}$ are totally split in~$L/F$.
    \item (4) holds because $[L:\bQ]$ is even, and $T, S_0, S_{\aux}$ are totally split in~$L/F$.
    \item (5) holds because 
    \begin{multline*}
    h^1_{\widetilde{T}\text{-split}}(L_{\widetilde{S}_p \cup \widetilde{S}_0 \cup \widetilde{S}_{\aux}}/L, \overline \chi) + 1 \leq h^1(L_{\widetilde{S}_p \cup \widetilde{S}_0}/L, \overline \chi)+ 2 |S_{\aux}| + 1 < a = |T|. 
    \end{multline*}
    \item To check~(6), we first let~$\pi_L$ be the base change through~$L/F$ of the automorphic representation~$\pi$ asserted to exist in Prop.\ \ref{prop:level-lowering}~(6). We then apply Lem.\ \ref{lem_raising_the_level} to $\pi_L$, with the set there denoted~$Z$ being our $\widetilde S_0 \cup \widetilde T \cup \widetilde S_{\aux}$ (which has even cardinality). Then the representation constructed in Lem.\ \ref{lem_raising_the_level} has the desired properties.
\end{itemize}
We have now checked the assumptions of Thm.\ \ref{thm_tame_SW_black_box}, which thus shows that $\rho'|_{G_L}$ is modular.
By cyclic descent through~$L/F$, the same is true for~$\rho'$. 
We will next use the global Jacquet--Langlands correspondence to move the automorphic representation associated to~$\rho'$ to
a Shimura curve, allowing for level-lowering; the next section introduces the requisite framework.

\subsection{Automorphic forms and Hecke algebras, II}\label{subsec:automorphic forms and Hecke algebras II}
We introduce some Hecke algebras and Hecke modules. 
To economize notation, we will sometimes use the same symbols as in Section~\ref{subsec:automorphic forms and Hecke algebras}
with a different meaning: for example, $H(U)$ will denote a Hecke module arising from the \'etale cohomology of a Shimura curve (as opposed to a space of algebraic modular forms).

Choose a place $z \in T$, and an infinite place~$\tau_1$ of~$F$. 
Define quaternion $F$-algebras $B$ and $B_\varnothing$ as follows:
\begin{itemize}
    \item $B$ is ramified precisely at $(S_\infty \setminus \tau_1) \cup z$, and $B_\varnothing$ is ramified precisely at~$S_\infty$.
\end{itemize}
We write~$\cJ$ for the set
of compact open subgroups 
$U \subset (B \otimes_F \bA_F^\infty)^\times$ such that 
\begin{itemize}
    \item $U$ factors as a product~$\prod_v U_v$, where $U_v \coloneqq U \cap (B \otimes_F F_v)^\times$,
    \item for all $v \not \in S \cup Z \cup \mathfrak{w},$ $U_v = \GL_2(\cO_{F_v})$,
    \item for all~$v \in S_p$, $U_v = \Iw_v$, %
    \item for all $v \in S \setminus S_p,$ $U_v = \Iw_v,$ 
    \item for all $v \in Z\setminus z$, $U_v \in \{\GL_2(\cO_{F_v}), \Iw_v\}$.
    \item $U_z$ is the unique maximal compact subgroup of $(B \otimes_F F_z)^\times$.
    \item $U_{\mathfrak{w}} = \Iw_{1, \mathfrak{w}}$.
\end{itemize}
For all $X \subset Z$, we write
\[
\bT_X^{S \cup Z, \mathrm{univ}} \coloneqq  \cO[T_v, S_v, U_x: v \not \in S \cup Z \cup \{\mathfrak{w}\}, x \in X].
\]
If~$U \in \cJ$, we introduce the following objects, associated with~$B$:
\begin{itemize}
    \item $R(U) \coloneqq  \{v \in Z: U_v \ne \GL_2(\cO_{F_v})\}$. (This is the ``ramification set" of~$U$ that we would like to eliminate through level-lowering).
    \item $X(U)$ is the Shimura curve over~$F$, associated to~$B$, of level~$U$. (We work with the same normalization of Shimura data as~\cite{Tho16}. As usual, the condition on~$U_{\mathfrak{w}}$ guarantees that~$X(U)$ is a proper smooth curve over~$F$.)
    \item If~$v \in R(U)\setminus z$, so that $U_v = \Iw_v,$ then $\mathbf X_v(U)$ is the regular semistable model of $X(U) \otimes_F F_v$ over $\Spec \cO_{F_v}$ defined in \cite[Section~4.4]{Tho16}, and there denoted $_v \mathbb M_{\{z\}}(U)$.

    If~$v = z$, so that $U_v$ is the maximal subgroup of the nonsplit group $(B \otimes_F F_z)^\times$, then $\mathbf X_v(U)$ is the regular semistable model of $X(U) \times_F \breve F_v$ over $\Spec \breve \cO_{F_v}$ defined in~\cite[Section~4.5]{Tho16}, and there denoted $_z \mathbb{M}_{\{z\}}(U)$.
    \item $\mathcal F$ is the lisse $\mathcal O$-sheaf on $\mathbf X_v(U)$ associated to the representation 
    \[
    \bigotimes_{\tau : F \to E} \operatorname{Sym}^{\lambda_{\tau,1}-\lambda_{\tau,2}} \mathcal O^{\oplus 2} \otimes \det\nolimits^{-(\lambda_{\tau,1}-\lambda_{\tau,2})/2}
    \]
    of $\prod_{v \mid p}\GL_2(\cO_{F_v})$.
    Observe that the generic fibre~$\mathcal F_{F_v}$ of~$\mathcal F$ is pulled back from a sheaf on~$X(U)$.
    \item $H(U) \coloneqq  H^1(X(U) \otimes_F \overline F, \mathcal F_{\overline F})$, with its natural (covariant) action of $\bT^{S \cup Z, \textnormal{univ}}_{R(U)}$. 
    \item $\bT^{S \cup Z}(H(U))$ is the
    image of $\bT^{S \cup Z, \mathrm{univ}}_{R(U)}$ in $\operatorname{End}_\cO(H(U))$. 
    (We interpret $U_z$ to be the Hecke operator associated to any choice of uniformizer of the nonsplit quaternion $F_z$-algebra $B \otimes_F F_z$.)
\end{itemize}

If~$U \in \cJ$, we also introduce the following objects, associated with~$B_\varnothing$:
\begin{itemize}
\item $\overline U \coloneqq  U^z \GL_2(\cO_{F_z})$, viewed as a subgroup of $(B_\varnothing \otimes_F \bA_F^\infty)^\times$.
\item $R(\overline U) \coloneqq  R(U) \setminus z$. 
\item $H_\varnothing(\overline U)$ 
is the space of $\overline U$-equivariant functions 
\[
f: B_\varnothing^\times \setminus (B_\varnothing \otimes_F \bA_F^\infty)^\times \to \left (\bigotimes_{\tau : F \to E} \operatorname{Sym}^{\lambda_{\tau_1}-\lambda_{\tau_2}} \mathcal O^{\oplus 2} \otimes \det\nolimits^{-(\lambda_{\tau_1}-\lambda_{\tau_2})/2} \right )^\vee.
\]
\item $\bT^{S \cup Z}(H_\varnothing(\overline U))$ 
is the image of $\bT^{S \cup Z, \mathrm{univ}}_{R(\overline U)}$
in $\operatorname{End}_\cO(H_\varnothing(\overline U))$.
\end{itemize}

We let~$\fm$ be the maximal ideal of $\bT^{S \cup Z, \mathrm{univ}}_\varnothing$
generated by
\[
(\varpi, T_v - \trace \rhobar(\Frob_v), S_v - q_v^{-1}\det \rhobar(\Frob_v): v \not \in S \cup Z \cup \mathfrak{w}).
\]
For all~$U \in \cJ$  we write $\CH^{S \cup Z}(U)$ for the (maximal) Cayley--Hamilton quotient of $\bT^{S \cup Z}(H(U))_\fm[G_{F, S \cup Z \cup \mathfrak{w}}]$
associated to the unique pseudocharacter
\[
\rho^{S \cup Z}(U)_\fm : \bT^{S \cup Z}(H(U))_\fm[G_{F, S \cup Z \cup \mathfrak{w}}] \to \bT^{S \cup Z}(H(U))_\fm
\]
such that, for all~$v \not \in S \cup Z \cup \mathfrak{w},$
the characteristic polynomial of $\rho^{S \cup Z}(U)_\fm(\Frob_v)$ is $X^2-S_v^{-1} T_v X +q_vS_v^{-1}$.
In other words, using notation from Section~\ref{subsec:pseudocharacters for level-lowering},
\[
\CH^{S \cup Z}(U) \coloneqq  \bT^{S \cup Z}(H(U))_\fm[G_{F, S \cup Z \cup \mathfrak{w}}]/\mathcal{I}_{\CH}(\rho^{S \cup Z}(U)_\fm).
\]

\subsubsection{Properties of Hecke modules}\label{subsubsec:properties of Hecke modules}
We now recall, and slightly refine, some standard properties of the modules~$H(U)_\fm$.

\begin{lemma}\label{lem:CH factorization}
Let~$U \in \cJ$.
Then the action of $\bT^{S \cup Z}(H(U))_\fm[G_{F, S \cup Z \cup \mathfrak{w}}]$ on $H(U)_\fm$ factors through $\CH^{S \cup Z}(U)$. 
\end{lemma}
\begin{proof}
The action of $\bT^{S \cup Z}(H(U))_\fm[G_{F, S \cup Z \cup \mathfrak{w}}]$ on $H(U)_\fm$  factors through $\bT^{S \cup Z}(H(U))_\fm\llbracket G_{F, S \cup Z \cup \mathfrak{w}}\rrbracket$, so this follows from Lem.\ \ref{lem:computation of CH ideals} and the Eichler--Shimura relation on $H(U)$. \qedhere
\end{proof}

\begin{lemma}\label{lem:Hecke support of H^0 and H^2}
    Let $V' \in \mathcal J$, let $v \not \in S$, and let $\mathcal G \in \{\mathcal F, \mathcal F/\varpi^i\}$.
    Then:
    \begin{enumerate}
    \item $H^j(X_v(V') \times_F {\overline F}, \mathcal G_{\overline F})_\fm = 0$ for~$j = 0, 2$.
    \item 
    Assume that $v \in R(Z)\setminus z$, and let $r : \widetilde{\mathbf{X}}_v(V')_{\overline k_v} \to \mathbf X_v(V')_{\overline k_v}$ be the normalization of the geometric special fibre of $\mathbf X_v(V')$.
    Then
    $H^2(\mathbf X_v(V')_{\overline k_v}, \mathcal G_{\overline k_v})_\fm = 0$, and $H^0(\widetilde{\mathbf{X}}_v(V')_{\overline k_v}, r^*\mathcal G_{\overline k_v})_\fm =~0$.
    \item Assume that $v = z$. Then $H^2(\mathbf X_z(V')_{\overline k_z}, \mathcal F_{\overline k_z})_\fm$ is $\cO$-torsion free. 
    \end{enumerate}
\end{lemma}
\begin{proof}
Proof of~(1): 
The essential point is that the support of these cohomology groups consists of maximal ideals of the Hecke algebra such that the associated Galois representations are a twist of $1 \oplus \epsilon^{-1}$. 
This refines the statement that these maximal ideals are Eisenstein, and it can be extracted from \cite[Lem.\ 18.1, Lem.\ 18.2]{Jar99}; we give some details.
It suffices to prove this when $\mathcal G = \mathcal F/\varpi^i$, and by Poincar\'e duality, when~$j = 0$.
After increasing the level at~$p$, we can assume without loss of generality that~$\mathcal G = \mathcal O/\varpi^i$.
Then the vanishing follows from our assumption that $\overline \chi |_{G_{F(\zeta_p)}} \ne 1$. 
Indeed, this assumption implies that there exists $w \not \in S \cup Z \cup a$ such that $q_w \trace (\Frob_w)^2 - (q_w+1)^2\det(\Frob_w) \not \in \fm,$
and so acts invertibly on $H^0(X(V') \times_F \overline F_v, \mathcal O/\varpi^i)_\fm$;
but an analysis of the set of connected components of $X(V') \times_F \overline F_w$ shows that
$T_w^2 - (q_w+1)^2S_w$ annihilates $H^0(X(V') \times_F \overline F_w, \mathcal O/\varpi^i)$ (because the group action on the set of connected components factors through the reduced norm of~$B$).

Proof of~(2): 
Again, it suffices to prove this when $\mathcal G = \mathcal F/\varpi^i$.
Since $\mathcal G_{\overline k_v} \to r_*r^*\mathcal G_{\overline k_v}$ is a monomorphism whose cokernel is supported at the singular points of $\mathbf X_v(V')_{\overline k_v}$,
it suffices to prove that $H^2(\widetilde{\mathbf X}_v(V')_{\overline k_v}, r^*\mathcal G_{\overline k_v})_\fm = 0$ and
$H^0(\widetilde{\mathbf{X}}_v(V')_{\overline k_v}, r^*\mathcal G_{\overline{k}_v})_\fm = 0$.
By Poincar\'e duality, it suffices to prove the vanishing for~$H^0$.
Let $V \coloneqq  V'^v \GL_2(\cO_{F_v})$. 
A standard analysis of the geometry of $\mathbf{X}_v(V')_{\overline k_v}$ (see e.g.\ \cite[Prop.\ 4.4]{Tho16}) now shows that %
\[
H^0(\widetilde{\mathbf{X}}_v(V')_{\overline k_v}, r^*\mathcal G_{\overline{k}_v})_\fm \isom H^0(X(V) \times_F \overline F_v, \mathcal G_{\overline F_v})_\fm^{\oplus 2},
\]
and so it vanishes by part~(1).

Proof of~(3): continue to write $r : \widetilde{\mathbf X}_z(V')_{\overline k_z} \to \mathbf X_z(V')_{\overline k_z}$
for the normalization morphism.
Since the cokernel of $\mathcal F_{\overline k_v} \to r_*r^*\mathcal F_{\overline k_v}$ has zero-dimensional support, it suffices to prove that
$H^2(\widetilde{\mathbf X}_z(V')_{\overline k_z}, r^* \mathcal F_{\overline k_z})_\fm$ is $\cO$-torsion free. In fact, $H^2(\widetilde{\mathbf X}_z(V')_{\overline k_z}, r^* \mathcal F_{\overline k_z})$ is $\cO$-torsion free: $\widetilde{\mathbf X}_z(V')_{\overline k_z}$ is a disjoint union of projective lines, so the lisse sheaf $r^* \mathcal F_{\overline k_z}$ is constant on every connected component, so this follows from the usual calculation of the \'etale cohomology of $\mathbb{P}^1$.
\end{proof}

\begin{cor}\label{cor:component group vanishes}
    Let~$V' \in \mathcal J$, and let $v \in R(V')\setminus z$.
    Then the component group associated to the sheaf~$\mathcal{F}$ on $\mathbf{X}_v(V')$ (as defined in~\eqref{eqn:component group}) vanishes after localization at~$\mathfrak{m}$. 
\end{cor}
\begin{proof}
    Recall that the component group is defined as the cokernel of the second vertical arrow in the diagram
    \begin{equation}\label{eqn:component group}
        \resizebox{\textwidth}{!}{\xymatrix{  
        H^1(X(V') \times_F \overline{F}_v, \mathcal F_{\overline F_v})(1) \ar[r] \ar[d]^N & \operatorname{image}(\beta) \ar[r] \ar[d] & \bigoplus_{x \in \operatorname{Sing}\mathbf{X}_v(V')_{\overline k_v}} R^1\Phi(\mathcal F)_x(1) \ar[d]^{\bigoplus_x N_x}\\
        H^1(X(V') \times_F \overline{F}_v, \mathcal F_{\overline F_v}) & \operatorname{coimage}(\beta') \ar[l] & \bigoplus_{x \in \operatorname{Sing}\mathbf{X}_v(V')_{\overline k_v}} H^1_x(\mathbf{X}_v(V')_{\overline k_v}, R\Psi \mathcal F) \ar[l]
        }}
    \end{equation}
    where~$N, N_x$ are the (global, resp.\ local) monodromy operators, and~$\beta, \beta'$ are morphisms in the specialization, resp.\ cospecialization sequence.
    Since $\operatorname{coker}(\beta)$ is a submodule of $H^2(\mathbf X_v(V')_{\overline k_v}, \mathcal F_{\overline k_v})(1)$, and 
    $\operatorname{ker}(\beta')$ is a quotient of $H^0(\widetilde{\mathbf X}_v(V')_{\overline k_v}, \mathcal F_{\overline k_v})$,
    it follows from Lem.\ \ref{lem:Hecke support of H^0 and H^2}~(2) that~$\beta_{\mathfrak m}$ is surjective and $\beta'_{\mathfrak m}$ is injective.
    On the other hand, the fact that~$\mathbf X_v(V')$ is regular implies that~$N_x$ is an isomorphism for all~$x$.
    Hence the map $\operatorname{image}(\beta) \to \operatorname{coimage}(\beta')$ becomes an isomorphism after localizing at~$\mathfrak{m}$, as desired.
\end{proof}

\begin{pp}\label{prop:level-lowering sequence for curves}
    Let~$V' \in \cJ$.
    Choose a place~$v \in R(V')\setminus z$, and let~$V \coloneqq  V'^v\GL_2(\cO_{F_v}).$
    Then, for any~$i > 0$, there is a $\bT^{S \cup Z, \mathrm{univ}}_{R(V)}[G_{F_v}]$-linear morphism
    \begin{equation}\label{eqn:level-lowering sequence for curves}
    (H(V') \otimes_\cO \cO/\varpi^i)_{\fm}^{I_{F_v}} \to (H(V) \otimes_\cO \cO/\varpi^i)^{\oplus 2}_{\fm}
    \end{equation}
    whose kernel is annihilated by~$\Frob_v^{-1} - U_v$.
\end{pp}
\begin{proof}
    We generalize the proof of~\cite[Thm.\ 4.8]{Tho16}, which covers the weight zero case, i.e.\ when $\mathcal F$ is the constant sheaf~$\cO$.
    It suffices to prove the following statements:
    \begin{enumerate}
        \item The natural map $H(V')^{I_{F_v}}_\fm \otimes_\cO \cO/\varpi^i \to (H(V') \otimes_\cO \cO/\varpi^i)_{\fm}^{I_{F_v}}$ is an isomorphism.
        \item There exists a $\bT^{S \cup Z, \mathrm{univ}}_{R(V)}[G_{F_v}]$-linear morphism $H(V')^{I_{F_v}}_\fm \to H(V)^{\oplus 2}_\fm$ whose kernel is annihilated by $\Frob_v^{-1} - \mathbf{U}_v$.
    \end{enumerate}
    Proof of~(1): 
    Since $H(V') \isom H^1(\mathbf{X}_v(V')_{\overline F_v}, \mathcal{F}_{\overline F_v})$, 
    we need to prove that
    \[
    H^1(\mathbf{X}_v(V')_{\overline F_v}, \mathcal{F}_{\overline F_v})_\fm^{I_{F_v}} \otimes_\cO \cO/\varpi^i \to (H^1(\mathbf{X}_v(V')_{\overline F_v}, \mathcal{F}_{\overline F_v}) \otimes_\cO \cO/\varpi^i)_\fm^{I_{F_v}}
    \]
    is an isomorphism.
    Since
    \[
    H^1(\mathbf{X}_v(V')_{\overline k_v}, \mathcal{F}_{\overline k_v})_\fm \isom H^1(\mathbf{X}_v(V')_{\overline F_v}, \mathcal{F}_{\overline F_v})_\fm^{I_{F_v}},
    \]
    this follows from~\cite[Prop.\ 4]{Raj01} and Cor.\ \ref{cor:component group vanishes}.

    Proof of~(2): Since $\mathcal F_{\overline k_v} \to r_*r^*\mathcal F_{\overline k_v}$ is a monomorphism whose cokernel is supported at the singular points of $\mathbf X_v(V')_{\overline k_v}$, a standard analysis of the geometry of $\widetilde{\mathbf{X}}_v(V')_{\overline k_v}$ shows that there is a $\bT^{S \cup Z, \mathrm{univ}}_{R(V)}[G_{F_v}]$-linear surjection
    \[
    H^1(\mathbf{X}_v(V')_{\overline k_v}, \mathcal F_{\overline k_v})_\fm \to H(V)_\fm^{\oplus 2},
    \]
    and we have already observed that $H^1(\mathbf{X}_v(V')_{\overline k_v}, \mathcal F_{\overline k_v})_\fm \cong H(V')^{I_{F_v}}_\fm.$ 
    The kernel of the resulting morphism
    $H(V')^{I_{F_v}}_\fm \to H(V)_\fm^{\oplus 2}$
    can be shown to be annihilated by $\Frob_v^{-1} - \mathbf{U}_v$ by the same argument as in~\cite[Thm.\ 4.8]{Tho16}.
    This concludes the proof.    
\end{proof}

\begin{pp}\label{prop:Cherednik--Drinfeld switch}
Let~$V \in \cJ$. 
Then, for all~$i > 0$, there exists a $\bT_{R(\overline V)}^{S \cup Z, \mathrm{univ}}$-linear morphism
\begin{equation}\label{eqn:Cherednik--Drinfeld switch}
(H(V) \otimes_\cO \cO/\varpi^i)_\fm^{I_{F_z}} \to (H_\varnothing(\lbar V) \otimes_\cO \cO/\varpi^i)^{\oplus 2}_{\fm}
\end{equation}
whose kernel is annihilated by $\Frob_z^{-1}-U_z$.
\end{pp}
\begin{proof}
    We generalize the proof of the weight zero case $\mathcal F = \cO$ in~\cite[Thm.\ 4.10]{Tho16}. 
    First, for any lisse sheaf~$\mathcal G$ of $\cO$-modules (or $\cO/\varpi^i$-modules) on~$\mathbf{X}_z(V)$, there is a Hecke-equivariant exact sequence
        \[
        0 \to H^1(\mathbf X_z(V), \mathcal G) \to H^1(X(V) \times_F \breve F_z, \mathcal G_{\breve F_z}) \to 
        \prod_{J \in \operatorname{Irr} \mathbf X_z(V)_{\overline k_z}}H^0(J^{\textnormal{reg}}, \mathcal G_{\overline k_z})(-1).
        \]
    (The existence of this sequence with $\mathbf X_z(V)$ replaced with $\mathbf X_z(V)\setminus \operatorname{Sing}\mathbf X_z(V)_{\overline k_z}$ is a consequence of absolute cohomological purity~\cite[Thm.\ 3.1.1]{Rio14}.
    Since $\operatorname{Sing}\mathbf X_z(V)_{\overline k_z}$ has codimension~2,
    the same result shows that
    $H^1(\mathbf X_z(V), \mathcal G) = H^1(\mathbf X_z(V)\setminus \operatorname{Sing}\mathbf X_z(V)_{\overline k_z}, \mathcal G)$.)

    Now, if~$\mathcal G \in \{\mathcal F, \mathcal F/\varpi^i\}$, 
    the natural map 
    \[
    H^1(X(V) \times_F \breve F_z, \mathcal G_{\breve F_z})_\fm \to H^1(X(V) \times_F \overline{\breve F_z}, \mathcal G_{\overline{\breve F_z}})_\fm^{I_{F_z}}
    \]
    is an isomorphism: given the Hochschild--Serre spectral sequence, this is a consequence of 
    \[
    H^0(X(V) \times_F \overline{\breve F_z}, \mathcal G_{\overline{\breve F_z}})_\fm = 0,
    \]
    which follows from Lem.\ \ref{lem:Hecke support of H^0 and H^2}~(1).
    Furthermore, $H^1(X(V) \times_F \overline{\breve F_z}, \mathcal F_{\overline{\breve F_z}})_\fm \isom H(V)_\fm,$
    and 
    \[
    H^1(X(V) \times_F \overline{\breve F_z}, (\mathcal F/\varpi^i)_{\overline{\breve F_z}})_\fm \isom H(V)_\fm\otimes_\cO \cO/\varpi^i
    \]
    (since $H^2(X(V) \times_F \overline F, \mathcal F)_\fm = 0$, again by Lem.\ \ref{lem:Hecke support of H^0 and H^2}~(1)).
    
    Putting all these together, and applying~\cite[Cor.\ 4.6]{Tho16}, we obtain a commutative diagram with exact rows
    \[
    \xymatrix{
    0 \ar[r] & H^1(\mathbf{X}_z(V), \mathcal F)_\fm \ar[r] \ar[d] & H(V)^{I_{F_z}}_\fm \ar[r] \ar[d] & H_\varnothing(\overline V)^{\oplus 2}_\fm \ar[d]\\
    0 \ar[r] & H^1(\mathbf{X}_z(V), \mathcal F/\varpi^i)_\fm \ar[r] & (H(V) \otimes_\cO \cO/\varpi^i)_\fm^{I_{F_z}} \ar[r] & H_\varnothing(\overline V)^{\oplus 2}_\fm \otimes_\cO \cO/\varpi^i
    }
    \]
    which generalizes~\cite[(4.17)]{Tho16}.
    Since $\Frob_z^{-1}$ equals~$\mathbf{U}_z$ on $H(V)^{I_{F_z}}_\fm$, 
    it now suffices to prove that the first vertical arrow is surjective.
    This follows from the fact that $H^2(\mathbf X_z(V), \mathcal F)_\fm$ is $\cO$-torsion free, which is a consequence of Lem.\ \ref{lem:Hecke support of H^0 and H^2}~(3) and the proper base change theorem.
\end{proof}

\subsection{Proof of Prop.\ \ref{prop:level-lowering}}\label{subsec:proof of level lowering}
We can now begin the proof of Prop.\ \ref{prop:level-lowering}.
Let~$U'$ be the unique element of~$\cJ$ such that $U'_v \ne \GL_2(\cO_{F_v})$ for all~$v \in Z \setminus \{z\}$. 
By the global Jacquet--Langlands correspondence,
and the modularity of~$\rho'$, there exists a unique $\cO$-algebra homomorphism 
\[
\alpha: \bT^{S \cup Z}(H(U'))_{\fm} \to \cO,
\]
such that, for all~$v \not \in S \cup Z \cup \mathfrak{w},$ we have
$\alpha(T_v) = \trace(\rho'(\Frob_v))$ and $\alpha(S_v) = q_v^{-1}\det(\rho'(\Frob_v))$.
Since, for all~$v \in Z$, $\operatorname{WD}(\rho'|_{G_v})^{F\textnormal{-ss}}$ 
is the Langlands parameter of the Steinberg representation of~$\GL_2(F_v)$,
the morphism~$\alpha$ has the additional property that $\alpha(U_v) = 1$ for all~$v \in Z$. %

It will be convenient to have some notation for the pullback of~$\alpha$ to the universal Hecke algebra $\mathbb{T}^{S \cup Z, \mathrm{univ}}_{Z} \to \bT^{S \cup Z}(H(U'))_{\fm}$ (see Section~\ref{subsec:automorphic forms and Hecke algebras II} for our normalizations, bearing in mind that~$R(U') = Z$).
More generally, for any $X \subset Z$, we will denote by~$\alpha_X$ the pullback of~$\alpha$ to a morphism $\bT^{S \cup Z, \mathrm{univ}}_{X} \to \cO$.

Now let $U \in \cJ$ be defined by $U_v \coloneqq  U_v'$ if $v \not \in Z$, and $U_v \coloneqq  \GL_2(\cO_{F_v})$ if~$v \in Z \setminus z$. 
We will show that there exists 
an $\cO$-algebra homomorphism
\[
\overline \alpha_N : \bT^{S}(H_\varnothing(\overline U))_\fm \to \cO/\varpi^{N-t_\rho},
\]
such that, for all $v \not \in S \cup \mathfrak{w}$, we have $\overline \alpha_N(T_v) = \trace(\rho(\Frob_v))$ and $\overline \alpha_N(S_v) = q_v^{-1}\det(\rho(\Frob_v))$.
Here, $t_\rho$ is a quantity defined in Section~\ref{subsec_preliminaries_to_level_lowering}, and depends only on~$\rho$ and~$S$, and $\bT^{S}(H_\varnothing(\overline U))$ denotes the algebra obtained by adjoining~$(T_v, S_v: v \in Z)$ to $\bT^{S \cup Z}(H_\varnothing(\overline U))$.
Therefore, the source of~$\overline \alpha_N$ depends only on~$\rho$ and~$S$, and we will be able to conclude by taking an inverse limit in~$N$. 

\begin{pp}\label{prop:level-lowering away from z}
    Let $V' \in \cJ$.
    Choose a place~$v \in R(V')\setminus z$, and let~$V \coloneqq  V'^v\GL_2(\cO_{F_v}).$
    Let~$i \geq \nu \geq 4C_\rho + 2m_v$, 
    assume that $(\rho' \otimes_\cO \cO/\varpi^i)|_{G_{F_v}}$ is unramified,
    and
    assume that there exists a vector
    \[
    f' \in (H(V') \otimes_\cO \cO/\varpi^i)_\fm[\alpha_{R(V')}]
    \]
    of exact order~$\varpi^\nu$.
    Then there exists a vector
    \[
    f \in (H(V) \otimes_\cO \cO/\varpi^i)_\fm[\alpha_{R(V)}]
    \]
    of exact order $\varpi^{\nu-4C_\rho-2m_v}$.
\end{pp}
\begin{proof}
    We apply Prop.\ \ref{prop_pseudocharacters_with_large_image}(4) with $M \coloneqq (H(V') \otimes_\cO \cO/\varpi^i)_\fm[\alpha_{R(V')}, \varpi^\nu], y \coloneqq f'$, 
    $C \coloneqq  C_\rho$, $D\coloneqq  m_v$, $N \coloneqq \nu$, and~$\phi$ an arbitrary lift of~$\Frob_v$ in~$\Gamma \coloneqq G_{F, S \cup Z \cup \mathfrak{w}}$.
    We verify its assumptions.
    Firstly, the existence of~$f'$ implies that the morphism $\alpha_{R(V')}: \mathbb{T}_{R(V')}^{S \cup Z, \textnormal{univ}} \to \cO/\varpi^\nu$ factors through $\mathbb{T}^{S \cup Z}(H(V'))_\fm$.
    By Lem.\ \ref{lem:CH factorization}, we deduce that $M$ is a module over 
    \[
    \CH^{S \cup Z}(V') \otimes_{\bT^{S \cup Z}(H(V'))_\fm, \alpha_{R(V')
    }} \cO/\varpi^\nu.
    \]
    Since the formation of maximal Cayley--Hamilton quotients commutes with base change,
    there is a natural isomorphism
    \[
    \CH^{S \cup Z}(V') \otimes_{\bT^{S \cup Z}(H(V'))_\fm, {\alpha_{R(V')}}} \cO/\varpi^\nu \isom \CH_{\rho'}/\varpi^\nu,
    \]
    where~$\CH_{\rho'}$ is the maximal Cayley--Hamilton quotient of $\cO[G_{F, S \cup Z \cup \mathfrak{w}}]$ corresponding to the pseudocharacter of~$\rho'$.
    Thus~$M$ is a $\CH_{\rho'}/\varpi^\nu$-module. 
    Furthermore, by construction, the eigenvalues of $\phi$ on~$\rho'$ are~$q_v$ and~$1$ (because $v \in R(V') \subset Z$). 
    So we can indeed apply
    Prop.\ \ref{prop_pseudocharacters_with_large_image}(4) to~$M$,
    and conclude that
    $(\CH_{\rho'}f')[\phi-q_v]$
    contains an $\cO$-submodule isomorphic to $\cO/\varpi^{\nu-2C_\rho-m_v}$.
    We replace~$f'$ by a generator of this submodule.

    We now apply Prop.\ \ref{prop_pseudocharacters_with_large_image}(2) to~$x \coloneqq  \tau_v-1$, where~$\tau_v$ is a generator of the tame quotient of~$I_{F_v}$.
    Since we are assuming that $(\rho' \otimes_\cO \cO/\varpi^i)|_{G_{F_v}}$ is unramified, 
    we have $\rho'(x) \equiv 0 \text{ mod } \varpi^i,$ and we conclude from Prop.\ \ref{prop_pseudocharacters_with_large_image}~(2) that $\varpi^{2C_\rho}x \in \varpi^i \CH_{\rho'}.$ 
    Hence
    \[
    \varpi^{2C_\rho}f' \in (H(V') \otimes_\cO \cO/\varpi^i)_\fm^{I_{F_v}}[\alpha_{R(V')}, \Frob_v - q_v].
    \]
    Let~$f$ be the image of~$\varpi^{2C_\rho}f'$ under~\eqref{eqn:level-lowering sequence for curves}.
    Since the kernel of~\eqref{eqn:level-lowering sequence for curves} is annihilated by $\Frob_v^{-1} - U_v$, and $\alpha_{R(V')}(U_v) = 1,$
    the intersection 
    \[
    (H(V') \otimes_\cO \cO/\varpi^i)_\fm^{I_{F_v}}[\alpha_{R(V')}, \Frob_v-q_v] \cap \ker \eqref{eqn:level-lowering sequence for curves}
    \]
    is annihilated by~$q_v-1$. 
    Bearing in mind that $m_v \coloneqq  \operatorname{ord}_\varpi(q_v-1)$, we conclude that
    \[
    \varpi^{\nu-4C_\rho-2m_v-1}f \ne 0,
    \]
    as desired.
\end{proof}

\begin{pp}\label{prop:level-lowering at z}
    Let~$V \in \cJ$.
    Let~$i \geq \nu \geq 2C_\rho$, assume that $(\rho' \otimes_{\cO} \cO/\varpi^i) |_{G_{F_z}}$ is unramified, and assume that there exists a vector
    \[
    f \in M \coloneqq  (H(V) \otimes_\cO \cO/\varpi^i)_\fm[\alpha_{R(V)}]
    \]
    of exact order~$\varpi^\nu$.
    Then there exists a vector
    \[
    \overline f \in (H_\varnothing(\overline V) \otimes_\cO \cO/\varpi^i)_\fm[\alpha_{R(\overline V)}]
    \]
    of exact order $\varpi^{\nu-4C_\rho}$. %
\end{pp}
\begin{proof}
    This follows from the same argument as Prop.\ \ref{prop:level-lowering away from z}, applying~\eqref{eqn:Cherednik--Drinfeld switch} instead of~\eqref{eqn:level-lowering sequence for curves}. 
\end{proof}

We now choose a vector
\[
f' \in (H(U') \otimes_\cO \cO/\varpi^{N-P_\rho})_\fm[\alpha_{R(U')}]
\]
of exact order~$\varpi^{N-P_\rho}$;
this is possible because $H(U')_\fm[\alpha_{R(U')}]$ is a finite free $\cO$-module,
and the map
\[
H(U')_\fm[\alpha_{R(U')}] \otimes_\cO \cO/\varpi^{N-P_\rho} \to (H(U') \otimes_\cO \cO/\varpi^{N-P_\rho})_\fm[\alpha_{R(U')}]
\]
is injective.
Repeated applications of Prop.\ \ref{prop:level-lowering away from z} (starting with~$i \coloneqq  N-P_\rho$) show that 
there exists a vector
\[
f \in (H(U) \otimes_\cO \cO/\varpi^{N-P_\rho})_\fm[\alpha_z]
\]
of exact order 
\[
\varpi^{N-P_\rho -4C_\rho|Z \setminus z|-2\sum_{v \in Z\setminus z}m_v}.
\]
By an application of Prop.\ \ref{prop:level-lowering at z}, we conclude that there exists a vector
\[
\lbar f \in (H_\varnothing(\overline U) \otimes_\cO \cO/\varpi^{N-P_\rho})_\fm[\alpha_\varnothing]
\]
of exact order 
\[
\varpi^{N-P_\rho - 4 C_\rho|Z|-2\sum_{v \in Z}m_v} = \varpi^{N-P_\rho - 4 C_\rho|Z|-2M_\rho|S_{\aux}|},
\]
where the equality is because $m_v = 0$ if~$v \in T$ and~$m_v = M_\rho$ if~$v \in S_{\aux}$.

Recall that we have introduced the notation $t_\rho \coloneqq  P_\rho + 4 C_\rho (s_\rho+ a)+2M_\rho s_\rho$, 
and that $P_\rho + 4 C_\rho |Z|+2M_\rho |S_{\aux}| \leq t_\rho,$
since $Z = T \cup S_{\aux}$ and 
$|T| = a,$ $|S_{\aux}| \leq s_\rho$, by construction.

Examining the action of $\bT^{S \cup Z}(H_\varnothing(\overline U))_\fm$ on~$\overline f$, we conclude that there exists a $\cO$-algebra homomorphism
\[
\overline \alpha_N : \bT^{S \cup Z}(H_{\varnothing}(\overline U))_{\fm} \to \cO/\varpi^{N-t_\rho}
\]
such that, for all~$v \not \in S \cup Z \cup \mathfrak{w},$ we have
$\overline\alpha_N(T_v) = \alpha(T_v) = \trace(\rho'(\Frob_v))$ and $\overline\alpha_N(S_v) = \alpha(S_v) = q_v^{-1}\det(\rho'(\Frob_v))$.
Since~$\rho'_{N-P_\rho} = \rho_{N-P_\rho}$, we also have
$\overline\alpha_N(T_v) = \trace(\rho(\Frob_v))$ and $\overline\alpha_N(S_v) = q_v^{-1}\det(\rho(\Frob_v))$.
Furthermore, by~\eqref{eqn:bound on N for level-lowering}, the exponent $N-t_\rho$ is positive.

We now observe that
\begin{equation}\label{eqn:equality of Hecke algebras}
\bT^{S \cup Z}(H_{\varnothing}(\overline U))_{\fm}
 = \bT^{S}(H_{\varnothing}(\overline U))_{\fm}.
\end{equation}
In fact, the inclusion 
\[
\bT^{S \cup Z}(H_{\varnothing}(\overline U))_{\fm} \subset \bT^{S}(H_{\varnothing}(\overline U))_{\fm}
\]
holds by definition, 
and the $G_{F, S}$-pseudocharacter valued in $\bT^{S}(H_{\varnothing}(\overline U))_{\fm}$ (which we will denote~$\rho^{S}(\overline U)_\fm$)
factors through the closed subring
$\bT^{S \cup Z}(H_{\varnothing}(\overline U))_{\fm},$
since this is true on the dense subset consisting of Frobenius elements away from~$S \cup Z \cup \mathfrak{w}$.     
Hence $\bT^{S \cup Z}(H_{\varnothing}(\overline U))_{\fm}$ also contains $T_v = \rho^{S}(\overline U)_\fm(\Frob_v), S_v = q_v^{-1}\det \rho^{S}(\overline U)_\fm(\Frob_v)$ for~$v \in Z$, which implies~\eqref{eqn:equality of Hecke algebras}.   (Note that $\overline U_v = \GL_2(\cO_{F_v})$ for all~$v \in Z$, hence there are no $U_v$-operators in $\bT^{S \cup Z}(H_{\varnothing}(\overline U))_{\fm}$ or $\bT^{S}(H_{\varnothing}(\overline U))_{\fm}$.)

The equality~\eqref{eqn:equality of Hecke algebras} implies that
there is an $\cO$-algebra homomorphism
\begin{equation}\label{eqn:alpha_N at correct level}
\overline\alpha_N : \bT^{S}(H_{\varnothing}(\overline U))_{\fm} \to \cO/\varpi^{N-t_\rho}
\end{equation}
such that, for all~$v \not \in S \cup \mathfrak{w}$, we have
$\overline\alpha_N(T_v) = \trace(\rho(\Frob_v))$ and $\overline\alpha_N(S_v) = q_v^{-1}\det(\rho(\Frob_v))$.
Letting~$N$ tend to $+ \infty$, and noting that the source of~\eqref{eqn:alpha_N at correct level} does not depend on~$N$, we obtain an $\cO$-algebra homomorphism
\[
\overline\alpha \coloneqq  \varprojlim_N \overline \alpha_N : \bT^S(H_\varnothing(\overline U))_\fm \to \cO
\]
such that, for all~$v \not \in S \cup \mathfrak{w}$, we have
$\overline\alpha(T_v) = \trace(\rho(\Frob_v))$ and $\overline\alpha(S_v) = q_v^{-1}\det(\rho(\Frob_v))$.
This concludes the proof of Prop.\ \ref{prop:level-lowering}.

\section{Deduction of Thm.\ \ref{introthm_automorphy_lifting}}\label{sec_deduction_of_main_theorem}

In this section, we show how Prop.\ \ref{prop:level-lowering} implies the theorem stated in the introduction. Our main tasks are to verify residual automorphy (which we do following \cite[\S 5.6]{Pan22}), and to show that the local conditions of Section \ref{sec_level_lowering} can be always to taken to hold after a solvable base change. Finally, the condition, used in Section~\ref{sec: construction of T-new lifts} and~\ref{sec_level_lowering}, that the image of the target Galois representation contain a conjugate of an open subgroup of $\GL_2(\mathbb{Z}_p)$, is shown to follow from the ordinary hypothesis (see Lem.\ \ref{lem_ordinary_implies_split_image}). 
Any generalisation of our results to the non-ordinary case would need to account for this.  

\subsection{Statement of main result} For the convenience of the reader, we first recall the statement of Thm.\ \ref{introthm_automorphy_lifting}.
\begin{thm}\label{thm_automorphy_lifting}
    Let $F$ be a totally real field, let $p$ be an odd prime, and let $\rho : G_F \to \GL_2(\overline{\mathbb{Q}}_p)$ be a continuous, irreducible representation satisfying the following conditions:
    \begin{enumerate}
        \item $\rho$ is unramified at all but finitely many places, and $\det \rho = \epsilon^{-1}$. 
        \item There exists $\mathbf{\lambda} \in (\mathbb{Z}^2_+)^{\Hom(F, \overline{\mathbb{Q}}_p)}$ such that $\rho$ is ordinary of weight $\mathbf{\lambda}$. 
        \item The semisimple residual representation $\overline{\rho}^{\textnormal{ss}} : G_F \to \GL_2(\overline{\mathbb{F}}_p)$ is reducible, isomorphic to $\overline{\chi}_1 \oplus \overline{\chi}_2$ for $\overline{\chi}_1, \overline{\chi}_2 : G_F \to \overline{\mathbb{F}}_p^\times$ continuous characters, and $ (\overline{\chi}_1 / \overline{\chi}_2)|_{G_{F(\zeta_p)}}$ is non-trivial. 
    \end{enumerate}
    Then $\rho$ is automorphic: there exists an isomorphism $\iota : \overline{\mathbb{Q}}_p \to \mathbb{C}$ and a cuspidal, regular algebraic, $\iota$-ordinary automorphic representation $\pi$ of $\GL_2(\mathbb{A}_F)$ of weight $\iota \mathbf{\lambda}$ such that $\rho \cong r_{\pi, \iota}$. 
\end{thm}
We begin with some preparatory lemmas. 
\begin{lemma}\label{lem_ordinary_implies_split_image}
        Let $E/\bQ_p$ be a coefficient field.
        Let $F$ be a number field, let $\lambda \in (\mathbb{Z}^2_+)^{\Hom(F, \overline{\mathbb{Q}}_p)}$, and let $\rho : G_F \to \GL_2(E)$  be a continuous representation with the following properties:
    \begin{enumerate}
        \item $\det \rho = \epsilon^{-1}$.
        \item The image of $\rho(G_F)$ in $\operatorname{PGL}_2(E)$ is Zariski dense in~$\operatorname{PGL}_{2, E}$.
        \item $\rho$ is ordinary of weight $\lambda$.
    \end{enumerate}
    Then~$\rho(G_F)$ contains a $\GL_2(E)$-conjugate of an open subgroup of $\GL_2(\mathbb{Z}_p)$.
\end{lemma}
\begin{proof}
    By \cite[Thm.\ 0.7]{Pin98}, we can find an intermediate extension $\mathbb{Q}_p \subset E_0 \subset E$, a central simple $E_0$-algebra $D_0$, and an isomorphism $D_0 \otimes_{E_0} E \cong M_2(E)$ such that $\rho(G_F) \leq E^\times D_0^\times$ and the image of $\rho(G_F)$ in $E^\times D_0^\times / E^\times \cong D_0^\times / E_0^\times$ is open. We claim that $D_0$ is split. Since any automorphism of $M_2(E)$ is inner, this will imply that we can conjugate $\rho$ so that $\rho(G_F) \leq E^\times \GL_2(E_0)$, with open image in $\operatorname{PGL}_2(E_0)$. It follows that the Lie algebra of the $p$-adic Lie group $\rho(G_F)$ is a $\mathbb{Q}_p$-subalgebra of $\mathfrak{gl}_2(E)$ that contains $\mathfrak{sl}_2(E_0)$ and whose trace is equal to $\mathbb{Q}_p$ (because $\det \rho = \epsilon^{-1}$ and $\epsilon^{-1}(G_F)$ is an open subgroup of $\mathbb{Z}_p^\times$). In particular, this Lie algebra contains $\mathfrak{gl}_2(\mathbb{Q}_p)$, and therefore $\rho(G_F)$ contains an open subgroup of $\GL_2(\mathbb{Z}_p)$. 

    We now prove the claim that $D_0$ is split. Suppose for contradiction that $D_0$ is non-split. Then if $g \in \rho(G_F)$, we can write $g = \lambda g_0$ for some $\lambda \in E^\times$, $g_0 \in D_0^\times$. If $g_0 \in E_0^\times$, then $g$ is scalar. If $g_0 \not\in E_0^\times$, then $E_0(g_0) / E_0$ is a quadratic extension of $E_0$ inside $D_0$. We conclude that for any $g \in \rho(G_F)$, the eigenvalues $\alpha, \beta \in \overline{\mathbb{Q}}_p$ of $g$ are either equal, or there exists a quadratic extension $L_0 / E_0$ such that $\alpha / \beta \in L_0^\times$ and $\alpha / \beta \in \ker( \mathbb{N}_{L_0 / E_0} : L_0^\times \to E_0^\times )$. Note that $\mathbb{Q}_p^\times \cap \ker( \mathbb{N}_{L_0 / E_0} ) = \{ \pm 1 \}$.

    On the other hand, we are assuming that $\rho$ is ordinary of weight $\lambda$, so for any place $v | p$ there is an isomorphism
    \[ \rho|_{G_{F_v}} \sim \begin{pmatrix} \psi_{v, 1} & * \\ 0 & \psi_{v, 2} \end{pmatrix}, \]
    where $\psi_{v, 1} \circ \Art_{F_v}$ agrees on an open subgroup of $\cO_{F_v}^\times$ with the character $x \mapsto \prod_{\tau \in \Hom(F_v, E)} \tau(x)^{-\lambda_{\tau, 2}}$, and $\psi_{v, 2} \circ \Art_{F_v}$ agrees on an open subgroup of $\cO_{F_v}^\times$ with the character $x \mapsto \prod_{\tau \in \Hom(F_v, E)} \tau(x)^{-(\lambda_{\tau, 1} + 1)}$. Therefore $(\psi_{v, 1}  / \psi_{v, 2}) \circ \Art_{F_v}$ agrees on an open subgroup of $\cO_{F_v}^\times$ with $x \mapsto \prod_{\tau \in \Hom(F_v, E)} \tau(x)^{1 + (\lambda_{\tau, 1} - \lambda_{\tau, 2})}$. 

    We conclude that there is an open subgroup $U \leq \mathbb{Z}_p^\times$ such that for any $u \in U$, there exists $g \in G_{F_v}$ with eigenvalues $\alpha, \beta \in \overline{\mathbb{Q}}_p$ satisfying the formula
    \[ \alpha / \beta = u^{[F_v : \mathbb{Q}_p] + \sum_{\tau \in \Hom(F_v, E)} (\lambda_{\tau, 1} - \lambda_{\tau, 2})}. \]
    The exponent in the right-hand side is a positive integer, so we can certainly choose $u$ so that the right-hand side is not equal to $\pm 1$. This contradicts the restriction we have placed on the eigenvalue ratios of elements of $E^\times D_0^\times$, so we conclude that $D_0$ must be split, as required. \qedhere

\end{proof}
\begin{lemma}\label{lem_basic_reductions}
    In order to prove Thm.\ \ref{thm_automorphy_lifting}, it suffices to prove it under the following additional hypotheses:
    \begin{enumerate}
    \item The image of $\rho(G_F)$ in $\operatorname{PGL}_{2, \overline{\mathbb{Q}}_p}$ is Zariski dense. 
    \item $|S_p| > 1$ and $[F : \mathbb{Q}]$ is even. 
        \item For each $v \in S_p$, $\overline{\rho}^{\textnormal{ss}}|_{G_{F_v}}$ is trivial, $[F_v : \mathbb{Q}_p] > 2$, and, if we write $\rho|_{G_{F_v}} \sim \left( \begin{smallmatrix} \psi_{v, 1} & \ast \\ 0 & \psi_{v, 2} \end{smallmatrix} \right)$ as in the definition of $\lambda_v$-ordinary, the character $\psi_{v, 1} \circ \Art_{F_v}$ is trivial on the subgroup of $p$-power roots of unity in $\cO_{F_v}^\times$. 
        \item There is a set $S_0$ of finite places of $F$ such that if $v \in S_0$, then $q_v \equiv 1 \text{ mod }p$, $\overline{\rho}^{\textnormal{ss}}|_{G_{F_v}}$ is trivial, and $\rho|_{G_{F_v}}^{\textnormal{ss}} \cong 1 \oplus \epsilon^{-1}$.
        \item For each finite place $v \not\in S_p \cup S_0$ of $F$, $\rho|_{G_{F_v}}$ is unramified.
    \end{enumerate}
\end{lemma}
\begin{proof}
    If the image of $\rho(G_F)$ in $\operatorname{PGL}_{2, \overline{\mathbb{Q}}_p}$ is not Zariski dense then $\rho$ is dihedral and its automorphy follows by automorphic induction. The remaining conditions can be satisfied after passage to a suitable soluble, totally real extension $F'/F$ (using soluble base change and descent). Note in particular that the condition that $\psi_{v, 1} \circ \Art_{F_v}$ be trivial on $p$-power roots of unity is implied by the condition that $\psi_{v, 1}(G_{F_v})$ contains no non-trivial elements of $p$-power order. If $E / \mathbb{Q}_p$ is a coefficient field such that $\psi_{v, 1}$ takes values in $\cO^\times$, then $\psi_v^{-1}(1 + p \cO)$ is an open subgroup whose image contains no non-trivial elements of $p$-power order, so it suffices to choose $F'$ such that, if $v'$ is a $p$-adic place of $F'$ lying above $v$, then $G_{F'_{v'}} \subset \psi_v^{-1}(1 + p \cO)$.     
\end{proof}

\begin{lemma}\label{lem_residual_automorphy}
    Let~$p$ be an odd prime, and let~$E/\mathbb Q_p$ be a coefficient field.
    Let $F$ be a totally real field satisfying $|S_p| > 1$, and let $\overline{\chi}_1, \overline{\chi}_2 : G_F \to k^\times$ be continuous characters satisfying the following conditions:
    \begin{enumerate}
        \item For each $v \in S_p$, $\overline{\chi}_1|_{G_{F_v}}, \overline{\chi}_2|_{G_{F_v}}$ are trivial. 
        \item $\overline{\chi}_1 \overline{\chi}_2 = \epsilon^{-1}$.
        \item $\overline{\chi}_1 / \overline{\chi}_2|_{G_{F(\zeta_p)}}$ is non-trivial. 
    \end{enumerate}
    Then we can find a soluble %
    totally real extension $F' / F$, an isomorphism $\iota : \overline{\mathbb{Q}}_p \to \mathbb{C}$, and a cuspidal, regular algebraic automorphic representation $\pi$ of $\GL_2(\mathbb{A}_{F'})$ of weight 0 and trivial central character, satisfying the following conditions:
    \begin{enumerate}
        \item $F' / F$ is linearly disjoint from the extension $F(\overline{\chi}_1, \overline{\chi}_2, \epsilon) / F$.
        \item There is an isomorphism $\overline{r}_{\pi, \iota}^{\textnormal{ss}} \cong (\overline{\chi}_1 \oplus \overline{\chi}_2)|_{G_{F'}}$
        \item $\pi$ is $\iota$-ordinary.
        \item $\pi$ is unramified away from the $p$-adic places of $F'$, and semistable at the $p$-adic places.
    \end{enumerate}
\end{lemma}
\begin{proof}
    Let $\ell \neq p$ be a prime not dividing $|k^\times|$, let $v_0 \in S_p$, and let $\alpha \in \overline{\mathbb{F}}_p^\times$ be an element of exact order $\ell$. Let $\overline{\theta} : G_F \to \overline{\mathbb{F}}_p^\times$ be a continuous character of order $\ell$, unramified at the $p$-adic places, and such that $\overline{\theta}(\Frob_{v_0}) = \alpha$ and $\overline{\theta}(\Frob_v) = \alpha^{-1}$ if $v \in S_p - \{ v_0 \}$. 

    The argument of \cite[\S 5.6]{Pan22} (multiply a weight 1 Eisenstein series with Dirichlet series $L^{S_p - \{ v_0 \}}(F, s, [ \overline{\chi}_1 \overline{\theta} ]) L^{\{ v_0 \}}(F, s, [ \overline{\chi}_2 \overline{\theta}^{-1}])$ by a high weight modular form whose $q$-expansion is congruent  to 1 mod $\varpi$) produces a cuspidal, regular algebraic automorphic representation $\pi'$ of $\GL_2(\mathbb{A}_F)$ and an isomorphism $\iota : \overline{\mathbb{Q}}_p \to \mathbb{C}$ with the following properties:
    \begin{itemize}
        \item $\pi'$ is $\iota$-ordinary.
        \item $\overline{r}_{\pi', \iota} \cong \overline{\chi}_1 \overline{\theta} \oplus \overline{\chi}_2 \overline{\theta}^{-1}$. 
    \end{itemize}
    By Hida theory, we can moreover assume that $\pi'$ has weight 0 and trivial central character. Then, making use of the main result of \cite{Ski01}, we can find a soluble extension $F' / F$ and a cuspidal, regular algebraic automorphic representation $\pi$ of $\GL_2(\mathbb{A}_{F'})$, all satisfying the following conditions:
    \begin{itemize}
        \item $F' / F$ is disjoint from the extension $F(\overline{\chi}_1, \overline{\chi}_2, \epsilon) / F$ and contains the extension $F(\overline{\theta}) / F$.
        \item $\pi$ is $\iota$-ordinary.
        \item $\pi$ is unramified away from the $p$-adic places of $F'$, and semistable at the $p$-adic places.
        \item There is an isomorphism $\overline{r}_{\pi, \iota}^{\textnormal{ss}} \cong \overline{r}_{\pi', \iota}^{\textnormal{ss}}|_{G_{F'}}$.
    \end{itemize}
    This $\pi$ now has the desired properties, by construction. 
\end{proof}

We now have all the ingredients for the proof of our main result.

\begin{proof}[Proof of Thm.\ {\em\ref{thm_automorphy_lifting}}]
    Let $\rho : G_F \to \GL_2(\overline{\mathbb{Q}}_p)$ be as in the theorem.
    Choose a sufficiently large coefficient field $E/\mathbb{Q}_p$ such that the image of~$\rho$ is contained in~$\GL_2(E)$, and the image of $\overline{\chi}_1, \overline{\chi}_2$ is contained in~$k^\times$.
    By Lem.\ \ref{lem_basic_reductions} and Lem.\ \ref{lem_residual_automorphy}, we can assume, without loss of generality, that $\rho$ satisfies the additional conditions of Lem.\ \ref{lem_basic_reductions} and that there is a cuspidal, regular algebraic automorphic representation $\pi$ of $\GL_2(\mathbb{A}_F)$ satisfying the conditions of Lem.\ \ref{lem_residual_automorphy}.     
    At this point, all the hypotheses of Prop.\ \ref{prop:level-lowering} are satisfied:
    and this concludes the proof.    
\end{proof}

\bibliographystyle{alpha}
\bibliography{references}
\end{document}